\documentclass{amsart}
\usepackage[foot]{amsaddr}
\usepackage{bbm,amsmath,amsthm,amsfonts, amssymb,mathdots,enumitem}
\usepackage{mathrsfs}
\usepackage{xcolor}
\usepackage{tikz, pgfplots}
\usetikzlibrary{positioning}
\usetikzlibrary{shapes,fit,arrows}
\usepackage{graphicx} 

 \newtheorem{thm}{Theorem} [section]
 \newtheorem{thrm}{Theorem}
 \newtheorem*{theorem*}{Theorem}
\newtheorem{definition}[thm]{Definition}
\newtheorem{lem}[thm]{Lemma}
 \newtheorem{prop}[thm]{Proposition}
\newtheorem{cor}[thm]{Corollary}

 \newtheorem{conj}[thm]{Conjecture}
 \newtheorem{prob}[thrm]{Problem}
 \newtheorem*{prob*}{Problem}

\newtheorem{assumption}[thm]{Assumption}
\theoremstyle{remark}
\newtheorem{rem}[thm]{{\bf Remark}}
\newcommand{\mP}{\mathcal{P}}  \newcommand{\mQ}{\mathcal{Q}}
 \newcommand{\mH}{\mathcal{H}} \newcommand{\mS}{\mathcal{S}}
\newcommand{\mO}{\mathcal{O}}  
\newcommand{\fO}{\mathfrak{O}}  \newcommand{\bZ}{\mathbb{Z}}
 \newcommand{\bC}{\mathbb{C}}\newcommand{\bP}{\mathbb{P}} 
 \newcommand{\bF}{\mathbb{F}}  \newcommand{\mF}{\mathcal{F}}
  \newcommand{\mN}{\mathcal{N}}
\newcommand{\ep}{\epsilon}

\newcommand{\bbbm}{ \begin{bmatrix}} \newcommand{\bebm}{\end{bmatrix}}
\newcommand{\bbm}{ \begin{pmatrix}} \newcommand{\bem}{\end{pmatrix}}
\newcommand{\bbsm}{ \left( \begin{smallmatrix}} \newcommand{\besm}{\end{smallmatrix} \right)}
\newcommand{\beq}{\begin{equation}}      \newcommand{\eeq}{\end{equation}}
\newcommand{\beqn}{\begin{eqnarray}}      \newcommand{\eeqn}{\end{eqnarray}}
\newcommand{\beqs}{\begin{eqnarray*}}      \newcommand{\eeqs}{\end{eqnarray*}}
\newcommand{\bep}{\begin{proof}}      \newcommand{\eep}{\end{proof}}

\begin{document}
\title{On the $PGL_2(q)$-orbits of  lines of $PG(3,q)$ and binary quartic forms}

\author{Krishna Kaipa$^{1,\ast}$ }
\address{$^1$Department of Mathematics, Indian Institute of Science Education and Research, Pune, Maharashtra, 411008 India.}
\address{$^{\ast}$Corresponding author}
\address{$^2$Department of Electrical Communication Engineering, Indian Institute of Science, Bangalore, 560012, India}
\address{$^3$Department of Mathematics, Indian Institute of Technology Bombay, Mumbai 400076, India
}

\author{Nupur Patanker$^2$}
\author{Puspendu Pradhan$^3$}
\email{$^1$kaipa@iiserpune.ac.in, $^2$nupurp@iisc.ac.in, $^3$puspendupradhan1@gmail.com}
\thanks{The second author is supported by NBHM, DAE, Govt. of India. The second and third author want to thank IISER Pune where this work was carried out.\\
}
\subjclass{51E20, 51N35, 14N10,  05B25, 05E10, 05E14, 05E18} 
\keywords{Rational normal curve, binary quartic forms, Klein quadric, $\jmath$-invariant}
\date{}

\begin{abstract}
We study the problem of classifying the lines of the projective $3$-space $PG(3,q)$ over a finite field $\bF_q$  into orbits of the group $G=PGL_2(q)$ of linear symmetries of the twisted cubic $C$. A generic line neither intersects $C$  nor lies in any of its osculating planes.  While the non-generic lines  have been classified into $G$-orbits in literature, it  has been an open problem to  classify the generic lines into $G$-orbits. For a general field $F$ of characteristic different from  $3$,  the twisted cubic determines a symplectic polarity on $\bP^3$. In the Klein representation of lines of $\bP^3$, the tangent lines of $C$ are represented by  a degree $4$ rational normal curve in a hyperplane $\mathcal H$ of the second exterior power $\bP^5$ of $\bP^3$. Atiyah \cite{Atiyah} studied the lines of $\bP^3$ with respect to $C$, in terms of the geometries of these two curves. 
We show that $\mathcal H$  can be identified in a $PGL_2$-equivariant way with  the space of binary quartic forms over $F$, and that pairs of polar dual lines of $\bP^3$ correspond to binary quartic forms whose apolar invariant is a square. We first solve the open problem of classifying binary quartic forms over $\bF_q$ into $G$-orbits, and then use it to solve the main problem.
\end{abstract}
 \maketitle

\section{Introduction} \label{introduction}
Let $PG(n,q)$ denote the projective space $\bP(\bF_q^{n+1})$ over a finite field $\bF_q$, and let $PGL_{n+1}(q)$ denote the projective general linear group acting on $PG(n,q)$.
Let $C$ be the twisted cubic curve in  $PG(3,q)$  defined by the image of the embedding 
\[PG(1,q) \hookrightarrow PG(3,q), \quad  (s,t) \mapsto (s^3,s^2t,st^2,t^3). \]
The standard action of $G=PGL_2(q)$ on the points $(s,t)$ of $PG(1,q)$, induces an action on the points $(s^3,s^2t,st^2,t^3)$ of $C$, and this action extends to an action of $G$ on $PG(3,q)$. We assume $q>4$ throughout this work. Under this condition,  the subgroup of  $PGL_4(q)$ preserving $C$ is isomorphic to $G$ via the action mentioned above.  The problem studied in this work is:
\begin{prob}  \label{mainproblem}
 Classify the lines of $PG(3,q)$ into $G$-orbits.
\end{prob}
The osculating plane at a point $P$ of $C$ is the unique plane of $PG(3,q)$ that intersects $C$ at $P$ with multiplicity $3$.   A line $L$ of $PG(3,q)$ is generic if it neither intersects  $C$, nor lies in any of its osculating planes. In the paper \cite{BH}, and also in the book \cite[p.236]{Hirschfeld3}, the class of generic lines is denoted by $\mO_6$, and the non-generic lines have been  partitioned into $8$ classes if char$(\bF_q) \neq 3$, and  $7$ classes if char$(\bF_q) =3$. Each of these classes is a union of $G$-orbits. The decomposition of these classes of nongeneric lines into $G$-orbits is known in literature, for example, \cite[Theorem 8.1]{BPS}, \cite[Theorem 3.1]{DMP1}, \cite{GL}. These form $9$ or $10$ orbits according as $q$ is even or odd.  On the other hand, it has been an open problem to decompose  the class of generic lines into $G$-orbits.   In \cite[Conjecture 20]{GL}, the authors conjecture that the number of orbits of generic lines  is $2q-2$ or $2q-4$ according as $q \equiv 1 \mod 6$  or $q  \equiv 5 \mod 6$. A more detailed conjecture was made by Davydov, Marcugini and Pambianco in \cite[Conjecture 8.2]{DMP1}. 

\begin{conj} \label{conj1}\cite{DMP1} The  number of $G$-orbits of generic lines of $PG(3,q)$ and their sizes are:\\

\begin{tabular}{c| *{5}{c}| c}
$q \mod 6$   & $|G|$ & $|G|/2$ & $|G|/3$ & $|G|/4$ & $|G|/12$   &  \# orbits  \\
&&&&&&\\
$3$   & $q/3$ & $q-1$ & $0$  & $(2q-6)/3$ &  $0$ &   $2q-3$  \\ 
$2$ & $1$ & $2q-4$ & $0$ & $0$ & $0$ &  $2q-3$  \\
$-2$  & $0$ & $2q-4$ & $3$ & $0$ &  $0$  &  $2q-1$  \\
$1$  & $(q-1)/3$ & $q-1$ & $2$ & $(2q-11)/3$ &  $1$  &  $2q-2$  \\
$-1$  & $(q+1)/3$ & $q-1$ & $0$ & $(2q-10)/3$ &  $0$ &   $2q-4$  \\
\end{tabular}
\end{conj}
The first column in the above table is the value of $q\!\!\mod 6$, and the last column is the total number of $G$-orbits. The columns in between give the number of orbits of various sizes listed in the first row.  Our main result is: 
\begin{theorem*} \label{main} The above conjecture is true for char$(\bF_q) \neq 2, 3$.
\end{theorem*}
The detailed result is given in Theorem \ref{thm_main} below. Our approach to the problem is to consider the ways in which a line of $PG(3,q)$ can meet the tangent lines of the twisted cubic $C$. We study this incidence in the Klein representation of the lines of $PG(3,q)$ as the Klein quadric $\mQ$ in $PG(5,q)$.  The analysis of this incidence proceeds differently in characteristic $2, 3$ and will be treated in a future work. In this work, we restrict to the case when char$(\bF_q) \neq 2,3$.  Let $\mathcal{T}$ be the quartic surface in $PG(3,q)$ defined by the equation \[(X_0X_3-X_1X_2)^2-4(X_0X_2-X_1^2)(X_1X_3-X_2^2)=0.\]
This surface parametrizes the union of all tangent lines to the twisted cubic $C$ in $PG(3,q)$ (see \cite[Lemma 21.1.10]{Hirschfeld3}). Since $\mathcal{T}$ has degree $4$, any line $L$ of $PG(3,q)$ not lying on $\mathcal{T}$ will intersect $\mathcal T$ in $4$ points counting multiplicities and points over extension fields. Suppose $L$ meets tangent lines of $C$ at the points 
$(s_i^3,s_i^2t_i,s_it_i^2,t_i^3), i=1,\dots ,4$. We can then associate to the line  $L$ a binary quartic form $f$ whose zeroes are points $(s_i,t_i), i=1,\dots ,4$ (see Lemma \ref{lem_pi} and the discussion immediately preceding  it for more details). As we explain below,  there are two steps in our approach to Problem \ref{mainproblem}. The first step is to classify binary quartic forms over $\bF_q$ into $G$-orbits. The second step relates $G$-orbits of lines of $PG(3,q)$ to $G$-orbits of a subset of binary quartic forms. 

\subsection{Approaches to Problem \ref{mainproblem} } Before explaining our approach, we  briefly discuss the approach to Problem \ref{mainproblem} in the recent works \cite{Ceria_Pavese}, \cite{BPS}, \cite{GL}, and \cite{DMP1}-\cite{DMP2}. In \cite{Ceria_Pavese} Ceria-Pavese prove  Conjecture \ref{conj1} for the case when   $q$ is even. 
\begin{theorem*} [Ceria-Pavese] \cite[Theorem 4.8]{Ceria_Pavese} 
Conjecture  \ref{conj1}  is true for $q$ even.
\end{theorem*}
They  not only treat the twisted cubic, but  also obtain the orbits of points, lines, and planes under the group of linear symmetries of a general arc of length $(q+1)$ in $PG(3,q)$. They also solve Problem \ref{prob1} (mentioned below) for the twisted cubic. In characteristic $2$, the analysis of the incidence of the lines of $PG(3,q)$ with the tangent lines of $C$ is simpler. This is because while the tangential variety $\mathcal T$  of $C$ mentioned above is a hypersurface of degree $4$ when $q$ is odd, in the case  when $q$ is even, $\mathcal T$  is a degree $2$ hypersurface given by a hyperbolic quadric. The tangent lines to $C$ form a regulus on this quadric, and the image of this regulus as well as its opposite regulus are conics in the Klein quadric $\mQ$. In the case of odd $q$, the tangent lines of $C$ form a rational normal curve of degree $4$ in $\mQ$.\\

 The problem of classifying binary quartic forms over $\bF_q$ into $G$-orbits can be restated in terms of  places and divisors of   the projective line $\bP^1$ over $\bF_q$: the problem is  to classify simple divisors of degree $4$ on $\bP^1$ into $PGL_2(q)$ orbits. This problem also arises in  the work \cite[Remark 8.2]{BPS} of Blokhuis-Pellikaan-Sz\H{o}nyi. We briefly discuss  the approach of these authors to Problem \ref{mainproblem}: viewing the dual space of $PG(3,q)$ as the space of binary cubic forms, they consider the pencil of binary cubic forms annihilating a given line $L$ of $PG(3,q)$. If $f(X,Y)$ and $g(X,Y)$ generate this pencil, then the rational function $\varphi= f/g$ is morphism  $\bP^1 \to \bP^1$  in the case when $L$ is a generic line. The `different divisor' of  $\varphi$ is the simple divisor of degree $4$ on $\bP^1$.\\

A problem closely related to Problem \ref{mainproblem} is as follows. First, we recall that the points of $PG(3,q)$ decompose as $\amalg_{i=1}^5 O_i$ into five $G$-orbits (see Lemma \ref{cubic}).
\begin{prob} \label{prob1} For representatives $L$ of each of the line orbits, determine the proportion of the set $\mathcal S$ of the $(q+1)$ points of $L$, lying in each of the $5$ point orbits:
\[ \frac{|\mathcal S \cap O_i|}{q+1}, \quad i=1\dots 5.\]
\end{prob}
In \cite{GL} Lavrauw-G\"{u}nay solve this problem for each  of the $10$ orbits of nongeneric lines when char$(\bF_q) \neq 2, 3$. As mentioned above, they also present the conjecture about the number of orbits of generic lines. The authors treat the lines of $PG(3,q)$ as pencils of binary cubic forms. We explain this formulation and its advantages: when char$(\bF_q) \neq 3$, we can 
identify points of $PG(3,q)$ with the projective  space of binary cubic forms  $y_0 Y^3- 3 y_1 Y^2X +3 y_2 YX^2-y_3 X^3$. The curve $C$  consists of forms $(Y s-X t)^3$ and the osculating plane to $C$ at this point consists of cubic forms divisible by $(Ys-Xt)$. For $g \in G$ represented by $g=\bbsm a & b\\ c & d \besm$ the action of $g$ is given by  $(g \cdot \varphi)(X,Y)=\varphi(dX-bY,aY-cX)$. Now, Problem \ref{mainproblem} is equivalent to classifying pencils of binary cubic forms into $G$-orbits.\\

The authors Davydov, Marcugini and Pambianco have a series of works on Problems \ref{mainproblem} and \ref{prob1}, for example, \cite{DMP3}  and the references [14]-[21] therein. They have obtained some of the orbits of the class of generic lines, and they have also solved Problem \ref{prob1} for the non-generic lines and some of the generic lines. As mentioned above, the detailed Conjecture \ref{conj1} about the description of orbits of generic lines is due to them.

\subsection{Outline of our approach to Problem \ref{mainproblem}} \hfill \\
We represent the lines of $PG(3,q)$ by the points of the Klein quadric $\mQ$ in $PG(5,q)$. The tangent lines to the twisted cubic $C$ form a degree $4$ rational normal curve $C_4$ in a hyperplane $\mH$ of $PG(5,q)$. The $G$ action on $C$ induces an action on $C_4$ which extends to an action of $G$ on $\mH$. We will show that there is a $G$-equivariant isomorphism between $\mH$ and the projective space of binary quartic forms
\[ f(X,Y)=z_0 Y^4-4 z_1 Y^3X+6 z_2 Y^2X^2 -4 z_3YX^3 + z_4X^4, \]
with $G$ acting  as $(g\cdot f)(X,Y) = f(dX-bY, aY-cX)$ for $g = \bbsm a&b \\ c&d \besm$.
Identifying $\mH$ with the space of binary quartic forms, the points of the curve $C_4$ correspond to quartic forms $(Xt-Ys)^4$ with a single repeated root. It will be necessary to introduce some more notation. The quadratic form $I(f)=(z_0z_4-4z_1z_3 + 3z_2^2)/3$ in the coefficients of $f$  is known as the apolar invariant of $f$. The cubic form $J(f)=  \det \bbsm z_0 & z_1 & z_2 \\ z_1 & z_2 & z_3 \\z_2 & z_3 & z_4 \besm$ in the coefficients of $f$ is known as the catalecticant invariant of $f$. The quantities $I(f)$ and $J(f)$  are fundamental   $GL_2(\bF_q)$-invariants of binary quartic forms.
The quantity $\jmath(f)$ defined by $1-1728/\jmath(f) =J^2(f)/I^3(f)$ is known as the absolute invariant or the $\jmath$-invariant of $f$. The discriminant of $f$ is the degree $6$ form in the coefficients of $f$ given by $I^3(f)-J^2(f)$. When char$(\bF_q) \neq 3$, the
map that  sends a point $P$ of $C$ to the osculating plane to $C$ at $P$ induces a $G$-invariant symplectic polarity on $PG(3,q)$. For a line $L$ of $PG(3,q)$, let $L^{\perp}$ denote the polar dual of $L$. In this way, the lines of $PG(3,q)$ always come in pairs $\{L, L^\perp\}$ (with $L^\perp=L$ possibly). We will show that:\\

\emph{There is a $G$-equivariant correspondence between the set of pairs of lines $\{L, L^{\perp}\}$ of $PG(3,q)$ and the set of  binary quartic forms over $\bF_q$ whose apolar-invariant is a square. }\\

This correspondence arises as follows.  We will show that there is an involutive linear map $\ast$ on $PG(5,q)$, such that at the projective level, the restriction of $\ast$ to $\mQ$ is the polar duality involution $L \mapsto L^\perp$. The fixed point set of $\ast$ consists of $\mH \cup \mP_0$ where $\mH$ is the hyperplane of $PG(5,q)$ mentioned above, and the point $\mP_0$ does not lie on $\mH$ or $\mQ$ (see Section \S \ref{chFneq23} for the geometric details).  The  projection from $\mP_0$ to $\mH$ when restricted to the Klein quadric $\mQ$ realizes $\mQ$ as a $G$-equivariant $2$-sheeted branched covering of the subset of  binary quartic forms whose apolar invariant  is a square in $\bF_q$.  The branches correspond to $\pm \sqrt{I}$ and represent a pair $\{L, L^{\perp}\}$, and the $\ast$ operator interchanges the sheets.  A line $L$ of $PG(3,q)$  lying over a quartic form $f(X,Y)$  intersects  the tangent line to $C$ at $(s^3,s^2t,st^2,t^3)$  if and only if $f(s,t)=0$, that is, $(Xt-Ys)$ is a factor of $f(X,Y)$.   This setup for studying the lines of $\bP^3$ including the projection of $\mQ$ from $\mP_0$ to $\mH$, the ramification locus $\mH \cap \mQ$ are already present in   the early paper \emph{`A note on the tangents of a twisted cubic'} \cite{Atiyah} of M.F. Atiyah, see \S \ref{Atiyah_note} for details. In fact, we may view  \S \ref{chFneq23} of this work as a continuation of Atiyah's note. \\

The two steps in our method of solving Problem \ref{mainproblem} are:
\begin{enumerate}  
\item [Step 1:] Decompose the space of binary quartic forms with no repeated roots into $G$-orbits.
\item [Step 2:] Determine the orbits in Step 1 of  quartic forms whose apolar invariant is a square, and use the covering map above to determine the $G$-orbits on $\mQ$.
\end{enumerate}
\subsection{Statement of results}~\\
For $i \in {1,2,4}$, we define
\[ J_i = \{ r \in \bF_q\setminus\{0,1728\} :  Z^3-4 Z^2 + \tfrac{256}{27}(1-\tfrac{1728}{r}) \text{ has $(i-1)$ roots over $\bF_q$}\}.\]
Here, $J_4 \cup J_2 \cup J_1= \bF_q\setminus\{0,1728\}$ is a partition (see Lemma \ref{Ji_lem}). The sizes of $J_4, J_2$ and $J_1$ are  $\tfrac{q-6-\mu}{6}, \tfrac{q-2+\mu}{2}$ and $\tfrac{q-\mu}{3}$ respectively, where  $\mu \in \{ \pm 1\}$ with  $q \equiv \mu \mod 3$.  (see Lemma \ref{Ji_def}).

\begin{thm} \label{result2} The set of $(q^4-q^2)$  binary quartic forms over $\bF_q$ with non-zero discriminant, decomposes into the following $G$-orbits parametrized by the $\jmath$-invariant.  The number of orbits of size indicated by the column index, and of $\jmath$-invariant indicated by the row index, is given by the corresponding entry of the table below. The total number of orbits is $2q+2+\mu$ where $q \equiv \mu \mod 3$ with $\mu \in \{\pm 1\}$.
\begin{table}[h!] 
\begin{tabular}{c| *{6}{c}}
   & $|G|$ & $\tfrac{|G|}{2}$ & $\tfrac{|G|}{3}$ & $\tfrac{|G|}{4}$ & $\tfrac{|G|}{12}$   &  $\tfrac{|G|}{8}$  \\
&&&&&&\\  \hline \\
$\jmath(f)\in J_4$ & $0$ & $0$ & $0$ & $4$ & $0$ &  $0$  \\
$\jmath(f)\in J_2$ & $0$ & $2$ & $0$ & $0$ & $0$ &  $0$  \\
$\jmath(f)\in J_1$ & $1$ & $0$ & $0$ & $0$ & $0$ &  $0$  \\
$\jmath(f)=1728$ & $0$ & $0$ & $0$ & $3$ & $0$ &  $2$  \\
$\jmath(f)=0$  & $0$ & \small{$(1-\mu)$} &  \small{$(1+\mu)$} & $\tfrac{1+\mu}{2}$ &  $\tfrac{1+\mu}{2}$  &  $0$
  \\ [1ex]
\hline  \\ [1ex]
\end{tabular}
\caption{Table of $G$-orbits of quartic forms with nonzero discriminant}
 \label{table:1}
\end{table}
\end{thm}
A list of representatives of these $(2q+2+\mu)$-orbits together with the isomorphism class of the stabilizer subgroup of each orbit  is given in Table \ref{table:3}.\\

We can now state our main result in more detail. For a line orbit $\fO$ represented by $L$, we use $\jmath(\fO)$ to denote $\jmath(f)$ for the quartic form $f$ associated to $
L, L^\perp$.    The last row of the  table below gives  the total number of orbits and their sizes, which is exactly as given in Conjecture \ref{conj1}.\par
We first define 
\[ J_i^+= \{r \in J_i :  1-\tfrac{1728}{r} \text{ is a quadratic residue in $\bF_q$}\}.\]
The sizes of the sets $J_i^+$ are given in Proposition \ref{J+prop}.
\begin{thm} \label{thm_main}
The number of $G$-orbits $\fO$ of generic lines of  size indicated by the column index and 
$\jmath$-invariant indicated by the row index is given in the table below.
\begin{table}[h!]
\noindent \begin{tabular}{c| *{5}{c}}
$\jmath(\fO)$   & $|G|$ & $\tfrac{|G|}{2}$ & $\tfrac{|G|}{3}$ & $\tfrac{|G|}{4}$ & $\tfrac{|G|}{12}$     \\
&&&&&\\  \hline \\
$\jmath\in J_4^+$ & $0$ & $0$ & $0$ & $8$ & $0$   \\
$\jmath\in J_2^+$ & $0$ & $4$ & $0$ & $0$ & $0$   \\
$\jmath\in J_1^+$ & $2$ & $0$ & $0$ & $0$ & $0$   \\
$\jmath=1728$ & $0$ &  \small{$\begin{cases}  0 &\text{ if $q \equiv \pm 1 \mod 12$} \\
2 &\text{ if $q \equiv \pm 5 \mod 12$} \end{cases}$}  & $0$ & 
 \small{$\begin{cases} 4 &\text{ if $q \equiv \pm 1 \mod 12$} \\
0 &\text{ if $q \equiv \pm 5 \mod 12$} \end{cases}$} & $0$   \\
$\jmath=0$  & $0$ &  \small{$(1-\mu)$} &  \small{$(1+\mu)$} & $\tfrac{1+\mu}{2}$ &  $\tfrac{1+\mu}{2}$  
  \\ [1ex]
\hline
\\
\emph{total} & $\tfrac{q-\mu}{3}$ &  \small{$q-1$} &   \small{$(1+\mu)$} & $\tfrac{2q-10 -(1+\mu)/2}{3}$ &  $\tfrac{1+\mu}{2}$. \\ [1ex]
\end{tabular}
\caption{Table of $G$-orbits of generic lines}
\label{table:2}
\end{table}
 All the orbits with $\jmath(\fO)  \in J_i^+$ for $i=1,2,4$ are polar non-self-dual. All  the orbits with $\jmath(\fO)   \in \{0, 1728\}$ are polar self-dual with one exception: if $q \equiv \pm 1 \mod 12$ then, of the $4$ orbits of size $|G|/4$ having  $\jmath(\fO)=1728$, two of them are polar-self-dual and the remaining two are non-self-dual.
\end{thm}

A list of generators in $PG(3,q)$ for lines representing these $(2q-3+\mu)$ orbits  together with the isomorphism class of the stabilizer subgroup of each orbit  is given in Table \ref{table:4}. 

Although we work throughout with  fields of characteristic different from $2$ and $3$, in  Appendix \ref{appendix}, we address a basic question about the $GL_2(F)$-representation $\text{Sym}^m V$ -- the $m$-th symmetric power of the vector space $V=F^2$ over an arbitrary field $F$. Let $D_m V$ (the $m$-th divided power of $V$) denote the dual vector space to $\text{Sym}^m(V^*)$  (the vector space of binary forms of degree $m$ over $F$, see \S \ref{setup} below). The basic questions are i) when are $D_m V$ and $\text{Sym}^m V$ irreducible as representations of $GL_2(F)$ and ii) when are they isomorphic. We give 
a complete answer to these questions in  arithmetic as well as  geometric terms (see Theorem \ref{lcm_cond} for the full result.)
\begin{theorem*} Let $F$ be a field that is arbitrary but for the assumption that $|F| \geq  m+2$ in case $F$ is finite. The following four conditions are equivalent: \begin{enumerate}
\item [i)] The $GL_2(F)$-representation $D_mV$ is irreducible.
\item [ii)]  The $GL_2(F)$-representations $D_mV$ and  $\text{Sym}^m V$ are isomorphic.
\item [iii)]  The intersection of all osculating hyperplanes of the rational normal curve $C_m$ of degree $m$ in $\bP(D_mV)$ is trivial.
\item [iv)]  Either char$(F)=0$ or $\text{char}(F)=p$ and $b < p$ where $m+1=p^a b$ is the unique factorization of $m+1$ with $b$ relatively prime to $p$.
\end{enumerate}
\end{theorem*}

The rest of the paper is organized as follows. In \S \ref{setup} we describe the geometric setup of the problem  over an arbitrary  field $F$. In \S \ref{chFneq23} we reduce the problem of classifying $PGL_2(F)$-orbits  of lines of $\bP(D_3 V)$ to the problem of classifying $PGL_2(F)$-orbits of binary quartic forms (more precisely, those quartic forms whose apolar invariant is a square). Here we assume char$(F) \neq 2,3$. Continuing with the assumption char$(F) \neq 2,3$,   in \S  \ref{sec4} we discuss   and develop some results about binary quartic forms and their $GL_2(F)$-invariants. In \S \ref{orbitsP4} and \S \ref{orbitsP3} we further assume $F = \bF_q$ and obtain the main results Theorem \ref{result2} and Theorem \ref{thm_main}, respectively.

\section{Geometric Setup} \label{setup}
We recall that the rational normal curve $C_m$ of degree $m$ in the $m$-dimensional projective space $\bP^m$ is the image of the Veronese map  $\nu_m: \bP^1  \to \bP^m$ given by 
\beq \label{eq:nudef} \nu_m(s,t)=(s^{m}, s^{m-1}t, \dots, s t^{m-1}, t^{m}). \eeq
It will be useful to consider the intrinsic definition of the Veronese maps $\nu_m$.   
Throughout this paper,  $V$ denotes the vector space $F^2$ over a field $F$. Consider the vector spaces 
Sym$^m V$ which is the $m$-th symmetric power of $V$, and Sym$^m(V^*)$ which is the $m$-th symmetric power of the dual vector space $V^*$, and finally 
the dual vector space of $\text{Sym}^m(V^*)$ (known as the $m$-th divided power of $V$) which we denote as 
\[  D_m V=  (\text{Sym}^m(V^*))^*. \]
 For the standard basis $e_1 = (1,0), e_2=(0,1)$ of $F^2$ we get induced bases on the vector spaces Sym$^m V $, Sym$^m(V^*)$ and $D_m V$.
 The induced basis on Sym$^m V$ is denoted by  $e_1^{m-i}e_2^i , i=0, \dots ,m$,  
and if $X, Y$ is the basis of $V^*$ dual to the basis $e_1, e_2$ of $V$, the induced basis
on Sym$^m(V^*)$ is denoted by $X^{m-i}Y^i, i=0, \dots ,m$. The   dual of this last basis, 
denoted by $e_1^{[m-i]} e_2^{[i]} , i=0, \dots ,m$ is the induced basis of $D_m V$. 
We now recall the coordinate-free definition of the Veronese map $\nu_m$ 
of  \eqref{eq:nudef}. Given $\theta_1, \dots, \theta_m  \in V^*$ and $v \in V$,  the function $(\theta_1, \dots, \theta_m) \mapsto  \theta_1(v) \theta_2(v) \dots \theta_m(v)$
  is multilinear and symmetric in  $\{\theta_1, \dots, \theta_m\}$  and hence defines a map 
\[  V \to  (\text{Sym}^m(V^*))^*=D_m V, \qquad \text{given by} \quad   s e_1 + te_2 \mapsto \sum_{i=0}^m   s^{m-i}t^i e_1^{[m-i]} e_2^{[i]}, \] which, at the projective level,  is the same as the Veronese map $\nu_m$ defined in \eqref{eq:nudef}.
The action of $GL_2(F)$ on $V=F^2$ induces actions on the spaces  Sym$^m V$, Sym$^m(V^*)$ and $D_m V$. For $g \in GL_2(F)$, we denote these actions by 
$g^m, g_m, g^{[m]}$, respectively. For example, for $g = \bbsm a &b \\ c& d \besm$ in $GL_2(F)$ we have 
\[ g^m(e_1^{m-i}e_2^i) = (ae_1+be_2)^{m-i} (c e_1 +de_2)^i, \quad g_m f(X,Y)=\det(g)^{-m} f(dX-bY, aY-cX). \]
The action $g^{[m]}$ on $D_m V$ is dual to the action $g_m$ on $\text{Sym}^m(V^*)$. For example, the matrices $g^{[3]}$ and $g^{[4]}$ are given by:
\beq \label{eq:g_3}
g^{[3]}=\bbm a^3 &3a^2b&3ab^2&b^3\\
a^2c &a(ad+2bc) &b(bc+2ad)&b^2d\\
ac^2 &c(bc+2ad)&d(ad+2bc)&bd^2\\
c^3 &3c^2d&3cd^2&d^3\bem, 
 \eeq
\beq \label{eq:g_4}
g^{[4]}=\bbm a^4 &4a^3b&6a^2b^2&4ab^3&b^4\\
a^3c &a^2(ad+3bc)  & 3ab(ad+bc)&b^2(bc+3ad)&b^3d\\
a^2c^2 &2ac(bc+ad)& (ad+bc)^2+2abcd&2bd(ad+bc)&d^2b^2\\
c^3a &c^2(bc+3ad) & 3cd (ad+bc) &d^2(ad+3bc) &d^3b\\
c^4 &4c^3d&6c^2d^2&4cd^3&d^4\bem. \eeq
 It is clear from the construction that for $g \in GL_2(F)$, we have
\[ \label{eq:nu_equivariance}  \nu_m(g \cdot v) = g^{[m]} \nu_m(v), \]
or more explicitly $\nu_m(as+bt,cs+dt)=g^{[m]} (s^{m}, s^{m-1}t, \dots, s t^{m-1}, t^{m})$. \\

A hyperplane of $\bP(D_m V)$ is given by an element $H = \sum_{i=0}^m a_i X^{m-i} Y^i$ of the dual space $\bP( \text{Sym}^m(V^*))$. The intersection multiplicity  $\imath_{H,C_m}(P)$ of  $H$ with $C_m$ at a point $P$ of $C_m$ is the least power of $z$ in $f(z) = \sum_{i=0}^m a_i h_i (z)$ where $z \mapsto h(z) = (h_0(z),\dots, h_m(z))$ is a parametrization of $C_m$ near $P = h(0)$.  For example, if $P = \nu_m(1,0)$ and $h(z)=(1,z, \dots,z^m)$ then  $f(z) = \sum_{i=0}^m a_i z^{i}$ and hence $\imath_{H,C_m}(P)$ is the largest value of $i$ such that $a_0= \dots= a_{i-1} =0$.  In particular, $\imath_{X^{m-i}Y^i,C_m}(P)=i$ for $P = \nu_m(1,0)$. Since $PGL_2(F)$ acts transitively on $C_m(F)$ it is clear that  $\imath_{H,C_m}(P)$  can take any of the values from $0$ to $m$. For $j=0 \dots m-1$,  the $j$-th osculating space to $C_m$ at $P$ is the intersection of all hyperplanes $H$ satisfying  $\imath_{H,C_m}(P) \geq j+1$. It is denoted as $\mO_j(P)$.  For example, for $P = \nu_m(1,0)$,  the spaces $\mO_j(P)$ are the $j$-dimensional linear subspace of $\bP(D_m V)$ given by the vanishing of $X^{m-j-1}Y^{j+1},\dots, X^0Y^m$. In particular, $\mO_{m-1}(P)$,  the $(m-1)$-th osculating space to $C_m$ at $P$ is called the \emph{osculating hyperplane} to $C_m$ at $P$. If $P= \nu_m(1,0)$ then this hyperplane is given by $Y^m=0$. For $P=\nu_m(s,t)$ we have  $\mO_{m-1}(P)= g \cdot  \mO_{m-1}(\nu_m(1,0))$ for $g = \bbsm s& 0\\ t& 1 \besm$. Therefore, $\mO_{m-1}(P)$ is defined by the vanishing of 
\[ P_{(s,t)}(X,Y)=\sum_{i=0}^m (-1)^i \tbinom{m}{i} \, t^{m-i}s^i   X^{m-i}Y^i . \]
\begin{definition} \label{polardual}  {\bf Bilinear form induced by $C_m$} \hfill \\
The map $C_m \to \bP (\text{Sym}^m(V^*))$ that takes  $P$  to the osculating hyperplane $\mO_{m-1}(P)$ extends to a linear map $\Omega_m: D_m V \to  \text{Sym}^m(V^*)$ or equivalently a bilinear form 
on  $D_m V$ represented by the $(m+1) \times (m+1)$ matrix 
\beq  \label{eq:Am} A_m =
 \bbsm & & && \binom{m}{0} \\
&& &-\binom{m}{1}&  \\
&    &\iddots& &  \\
&(-1)^{m-1} \binom{m}{m-1}&&  & \\ 
(-1)^m \binom{m}{m}&& &  &  
 \besm.\eeq
 In the case when the matrix $A_m$ is invertible (i.e. all the binomial coefficients $\binom{m}{0}, \dots, \binom{m}{m}$  are non-zero in the field $F$), we also refer to $\Omega_m$ as the polarity on $D_mV$ induced by $C_m$. It is clear 
from the matrix $A_m$ that the bilinear form $\Omega_m$ is symmetric if $m$ is even, and alternating if $m$ is odd.
 \end{definition}

\begin{lem}
  The elements of $GL_2(F)$ are similitudes for the bilinear form $\Omega_m$:  
 \beq \label{eq:Am_equivariance}   A_m g^{[m]} = \det(g)^m g_m A_m. \eeq   In other words, $A_m$ is an element of the vector space 
 \[ \text{Hom}_{GL_2(F)}(D_m V, \text{Sym}^m(V^*) \otimes  (\otimes^m \wedge^2(V)) \] of $GL_2(F)$-equivariant maps from $D_m V$ to $\text{Sym}^m(V^*) \otimes  (\otimes^m \wedge^2(V))$.
\end{lem}
\begin{proof}
    The second assertion is just a restatement of the identity \eqref{eq:Am_equivariance}. This identity is readily verified for the matrices $g = \bbsm 1 &0 \\0 &d \besm,  \bbsm  0& 1\\ 1& 0\besm$ and $ \bbsm 1 & 0\\c  &1 \besm$. Since these matrices generate $GL_2(F)$, it follows that \eqref{eq:Am_equivariance} holds for all $g \in GL_2(F)$. \\
We remark that in fact,  we show in Proposition \ref{equivs} of  the Appendix  that 
 \[ \text{Hom}_{GL_2(F)}\left(D_m V, \text{Sym}^m(V^*)\otimes (\otimes^m \wedge^2V)\right)\]
  is the  one-dimensional  vector space spanned by  $A_m$, provided that $|F|\geq m+2$ in case $F$ is finite.  We also show there, that  any $PGL_2(F)$-equivariant isomorphism between $\bP (D_m V)$ and $\bP(\text{Sym}^m(V^*))$ must be represented by $A_m$. In particular, such an isomorphism exists if and only if all the binomial coefficients $\binom{m}{i}$ are relatively prime to char$(F)$.
    \end{proof}

We use the Pl\"{u}cker embedding to represent the lines  of $\bP(D_3 V)$ as the points of the  Klein quadric $\mQ$ in $\bP(\wedge^2 D_3 V)$:  Let us denote the basis $e_1^{[3-i]} e_2^{[i]}, 0 \leq i \leq 3$ of $D_3 V$ as $B_0, B_1, B_2, B_3$. Let  $B_{ij} = B_i \wedge B_j, 0\leq i < j \leq 3$ denote the induced basis on $\wedge^2 D_3 V$.  A point $\sum p_{ij} B_{ij}$ of $\bP(\wedge^2 D_3 V)$ lies on the quadric $\mQ$ if and only if it satisfies the Pl\"{u}cker relation 
\beq \label{plucker} p_{01}p_{23} -p_{02} p_{13} + p_{03} p_{12}=0. \eeq

Given  two lines $L$ and  $L'$ of $\bP(D_3V)$ represented by  points $Q$  and  $Q'$   of $\mQ$, respectively, we recall the `Klein correspondence'  which asserts  that  $L$ intersects $L'$  if and only if the line joining $Q$ and $Q'$  lies on $\mQ$. In particular, the set of all lines $L'$ that intersect  a fixed line $L$ is represented by the points  of  intersection of $\mQ$ with the  tangent hyperplane to $\mQ$ at $Q$. We denote this intersection as $\mS_Q$ (it is a Schubert divisor of $\mQ$ viewed as a Grassmannian).  If $p_{ij}$ are the coordinates of $Q$, then $\mS_Q$  consists of points $Q'$ of $\mQ$ whose coordinates $p_{ij}'$ satisfy
\beq \label{eq:TQ} p'_{01}p_{23}  +p_{01}p'_{23} -p'_{02} p_{13}  -p_{02}p'_{13} + p'_{03} p_{12}+ p_{03} p'_{12}=0.\eeq
For example, if $Q = B_{01}$ then $\mS_Q$ is the Schubert divisor of all points of $\mQ$ whose Pl\"{u}cker coordinate satisfies $p_{23}=0$. The tangent lines of $C$ are represented by  a curve $\tilde C$ on $\mQ$. 
We show that $\tilde C$ is a rational normal  curve of degree $4$ lying on $\mQ_0=\mQ \cap \mH$, where $\mH$ is the hyperplane of $\bP(\wedge^2 D_3 V)$ given by $3p_{12}-p_{03}=0$:  The tangent line to $C$ at a point $\nu_3(s,t)$ is spanned by  $(3s^2,2st,t^2,0)$ and $(0,s^2,2st,3t^2)$, and hence  the  Pl\"{u}cker coordinates of this line are  given by 
\beq \label{eq:mu} T(s,t) = s^4 B_{01} + 2ts^3 B_{02} + s^2t^2(3 B_{03} +B_{12}) + 2st^3 B_{13} +t^4 B_{23}. \eeq
Therefore, $\tilde C : (s,t) \mapsto T(s,t)$ is a rational normal curve of degree $4$ lying on $\mQ_0$. Let $\mS_{s,t}$ denote $\mS_Q$ for $Q = T(s,t)$, i.e. the points of $\mS_{s,t}$ represent  the lines of $\bP(D_3 V)$ which intersect  the tangent line to $C$ at $\nu_3(s,t)$.
Using \eqref{eq:mu} in \eqref{eq:TQ}, we see that the points of  $\mS_{s,t}$ have coordinates $p_{ij}$  satisfying:
 \begin{align}  \label{eq:incidence} t^4 p_{01} -2 st^3 p_{02} +s^2t^2(p_{03}+3 p_{12}) -2 s^3t p_{13} +s^4 p_{23} &=0\\
\nonumber  p_{01}p_{23} -p_{02} p_{13} + p_{03} p_{12}&=0. \end{align}

\section{Reduction to $PGL_2(F)$ orbits of binary quartic forms when char$(F) \neq 2, 3$} \label{chFneq23}
As mentioned in \S\ref{introduction}, we will show that there is a $2$-sheeted branched covering  map $\pi: \mQ \to \bP(\text{Sym}^4(V^*))$ which is $PGL_2(F)$-equivariant and whose image consists of  binary quartic forms $f$ whose apolar-invariant $I(f)$ is a square in $F$. The inverse image $\pi^{-1}(f)=\{L , L^{\perp}\}$  consists of a pair of  points of $\mQ$ which are Hodge duals (to be defined  below) of each other in $\bP(\wedge^2 D_3 V )$. They are  also  `polar duals' of each other (as lines of $\bP(D_3 V)$), where polar dual refers to the orthogonal complements with respect to the symplectic bilinear form $\Omega_3$).
We will now interpret this polar duality in terms of Hodge duality.  Before doing so, we quickly mention the main points, especially for the reader who wishes to skip the details.
We define the Hodge star  linear isomorphisms $\ast_i:  \wedge^i D_3 V \leftrightarrow \wedge^{4-i} D_3V$ for $i=0,1,2$. At the projective level $\ast_i : \bP( \wedge^i D_3 V) \leftrightarrow \bP (\wedge^{4-i} D_3 V)$ are easy to describe. If $i=0$, both projective spaces are zero-dimensional and $\ast_0$ just sends the unique point of $\bP(\wedge^{0} D_3 V)$ to the unique point of $\bP (\wedge^{4} D_3 V)$. If $i=1$ the points of  $\bP (\wedge^{3} D_3 V)$ are just the hyperplanes of $\bP(D_3 V)$ and $\ast_1(\xi)$ 
is the polar dual (orthogonal complement) of $\xi$ with respect to $\Omega_3$. If $i=2$ and $\xi \in \mQ$ represents a line $L$ of $\bP(D_3 V)$ then  again $\ast_2(\xi) \in \mQ$ represents  the polar dual/orthogonal complement $L^{\perp}$. We will show below
(see Definition \ref{hodge_def}) that this action of $\ast_2$ on $\mQ$ extends to a linear isomorphism on $\bP(\wedge^2 D_3V)$. If we are willing to assume this, then the linear automorphism $\ast_2$ on $\wedge^2 D_3 V$ can be completely determined: Since a tangent line to $C$ (given in coordinates by \eqref{eq:mu}) is its own polar dual, it follows that 
\[ s^4 B_{01} + 2ts^3 B_{02} + s^2t^2(3 B_{03} +B_{12}) + 2st^3 B_{13} +t^4 B_{23}\]
is fixed by $\ast_2$, and hence $\ast_2$ fixes $B_{01}, B_{23}, B_{02}, B_{13}$ and $(3 B_{03} + B_{12})$. Also, the polar dual of the secant line joining $\nu_3(1,0)$ and $\nu_3(1,t)$ represented by $(B_{01} +t B_{02} +t^2 B_{03})$ is the intersection of the osculating planes to $C$ at these two points, represented by $B_{01}+ t B_{02} + t^2/3  B_{12}$. It follows that $\ast_2$ interchanges $B_{12}$ and $3 B_{03}$. We summarize this as
\beq \label{eq:ast2} *_2 (B_{01},B_{02},B_{03},B_{12},B_{13},B_{23} )=(B_{01},B_{02}, \tfrac{1}{3} B_{12},3 B_{03},B_{13},B_{23}).\eeq

We now return to  the definition of  Hodge duality  on the exterior algebra of $D_3 V$ with respect to the symplectic form $\Omega_3$.
Viewing the linear map $\Omega_3^{-1}: (D_3 V)^* \to D_3 V$ as a symplectic bilinear form on   $(D_3 V)^*$ we get an element of $\wedge^2 ((D_3 V)^*)^*$. As  mentioned in the Remark \ref{exteriordual}, there is a $GL(V)$-equivariant isomorphism between $\wedge^2 ( (D_3V)^*)^*$ and $\wedge^2 D_3(V)$ which allows us to identify 
 $\Omega_3^{-1}$ with the following element $\gamma_3$ of $\wedge^2 D_3V$:
\[  \gamma_3 = -B_{03}+ \tfrac{1}{3} B_{12}.  \]
Since all elements of $GL_2(F)$ are similitudes for $\Omega_3$, we see that the point  $\mP_0$ of $\bP (\wedge^2 D_3 V)$ represented by $\gamma_3$,  is fixed by $PGL_2(F)$. We also consider $\Gamma = -\tfrac{1}{2} \gamma_3 \wedge \gamma_3 \in \wedge^4 D_3V$ given by 
\[ \Gamma=  \tfrac{1}{3}  B_0 \wedge B_1 \wedge B_2 \wedge B_3.\]
We are now ready to define the Hodge $\ast$ operator:

\begin{definition} \label{hodge_def}
Given $\xi \in \wedge^i D_3 V$  we define  the Hodge dual $*\xi \in  \wedge^{4-i} D_3V$ by 
\[ *\xi= \imath_{ (\wedge^i \Omega_3) \xi } \Gamma.   \]
\end{definition}
Here $\imath$ is the interior multiplication operation in exterior algebra:  for $\zeta \in \wedge^i   (D_3V)^*$ the expression $\imath_{\zeta} \Gamma$ in $\wedge^{4-i} D_3V$ is defined by requiring 
$\langle \psi , \imath_{\zeta} \Gamma \rangle = \langle \zeta \wedge \psi, \Gamma \rangle$ for all $\psi \in \wedge^{4-i} (D_3V)^*$ where $\langle, \rangle$ denotes the pairing between a vector space and its dual. We now show that $\ast$ acts on $\wedge^2 D_3V$ by \eqref{eq:ast2}. Since $\Omega_3$ carries $B_0,B_1,B_2,B_3$ to 
$Y^3,-3Y^2X,3YX^2,-X^3$ or in short $\Omega_3(B_i)=(-1)^i \binom{3}{i} X^i Y^{3-i}$, it follows that $\wedge^2 \Omega_3 (B_{ij})=(-1)^{i+j-1} 
\binom{3}{i}\binom{3}{j} (X^jY^{3-j})\wedge (X^i Y^{3-i})$. Therefore, 
\[\ast B_{ij}= \imath_{\wedge^2 \Omega_3 (B_{ij})} \Gamma =  \frac{\text{sgn}(i j k \ell)}{3} (-1)^{i+j-1} 
\tbinom{3}{i}\tbinom{3}{j} B_{3-\ell, 3-k}, \quad \text{where $\{ijk\ell\}=\{0123\}$.}\]
Here $i<j$, $k <\ell$ and sgn$(ijk\ell)$ is the sign of the permutation that sends $0,1,2,3$ to $i,j,k,\ell$ in that order. This agrees with \eqref{eq:ast2}. It is clear from \eqref{eq:ast2} that $\ast$ is an involution on $\wedge^2 D_3V$ and has eigenvalues $\pm 1$. An element $\xi$ of $\wedge^2 D_3 V$ is called \emph{self-dual} if $\ast \xi=\xi$ and \emph{anti-self-dual} if $\ast \xi=-\xi$. We denote the linear subspaces of self-dual and anti-self-dual vectors of $\wedge^2 D_3 V$ by $\wedge^2_+ (D_3 V)$ and $\wedge^2_-(D_3 V)$ respectively.  It will be convenient to work with the following new basis of $\wedge^2 D_3 V$:
\[E_0=B_{01}, E_1 = 2 B_{02}, E_2 = B_{12}+ 3 B_{03}, E_3=2 B_{13},  E_4=B_{23}, E_5 =  3 B_{03} -B_{12}.\]
It follows from \eqref{eq:ast2} that $\wedge^2_+(D_3 V)$ is the $5$-dimensional subspace of $\wedge^2 D_3V$  spanned by   $E_0, \dots, E_4$, and $E_5$ spans the $1$-dimensional subspace of $\wedge^2_- (D_3V)$.
The coordinates $(z_0,\dots, z_5)$ with respect to the new basis are related to the Pl\"{u}cker coordinates $p_{ij}$ by:
\beq \label{eq:p_z} (p_{01}, p_{02}, p_{03}, p_{12}, p_{13}, p_{23})=(z_0, 2z_1,3(z_2+z_5), z_2-z_5, 2z_3,z_4).  \eeq
The linear isomorphism $\ast_2$ on $\wedge^2 D_3 V$ is given in the new coordinates by 
\[ \ast (z_0, \dots, z_4, z_5) =(z_0, \dots, z_4, -z_5). \]
Using the matrices \eqref{eq:g_3} and \eqref{eq:g_4} a direct calculation shows that: 
\beq \label{eq:wedge2g4} 
 \wedge^2 g^{[3]}  \cdot \bbsm z_0 \\ z_1 \\ z_2 \\ z_3 \\ z_4 \\z_5 \besm  = \det(g) \bbm g^{[4]} \bbsm z_0 \\z_1 \\z_2 \\z_3 \\z_4 \besm \\ \det(g)^2 z_5 \bem.
 \eeq
We have thus established the following lemma: 
\begin{lem} \label{dec_lem} The vector space $\wedge^2_-(D_3 V)$ of anti-self-dual elements is the one-dimensional space $\hat \mP_0$ spanned by $E_5$
and the vector space $\wedge^2_+ (D_3 V)$ of self-dual elements  is the five-dimensional space  $\hat \mH$ spanned by $E_0, \dots, E_4$. As representations of $GL_2(F)$ we have
\[ \wedge^2_- (D_3 V) \simeq \otimes^3 (\wedge^2 V), \qquad \wedge^2_+ (D_3 V) \simeq    (D_4 V)  \otimes \wedge^2 V. \]
Together with the  decomposition $\wedge^2 D_3 V   =\wedge^2_+ (D_3 V)  \oplus \wedge^2_-(D_3 V) $ we get the  isomorphism of $GL_2(F)$-modules
\beq \label{eq:pleth} \wedge^2 D_3 V \simeq  (\wedge^2(V) \otimes   D_4 V)\; \oplus\; (\otimes^3 (\wedge^2 V) ). \eeq
\end{lem}
 We note, in particular,  that the $\ast$ commutes with the action of $GL_2(F)$ on $\wedge^2 D_3 V$. 
\begin{rem} \label{rem_pleth} The isomorphism \eqref{eq:pleth}  is the analogue for $GL_2(F)$ of the well-known `plethysm'  isomorphism  of $SL_2(\bC)$ modules   (\cite[eqn. 11.12]{Fulton-Harris}) :
\[  \wedge^2\text{Sym}^3 W \to \text{Sym}^4(W) \oplus \text{Sym}^0(W),\qquad  \text{where $W = \bC^2$}.\]  

We also remark that there is also a $GL_2(F)$-equivariant isomorphism $\wedge^2 D_3 V \to (\wedge^2 V) \otimes D_2(D_2 V)$ where $D_2(D_2 V) = \text{Sym}^2( \text{Sym}^2(V^*))^*$ is the vector space of symmetric bilinear  forms on $\text{Sym}^2(V^*)$. This is the $GL_2(F)$ analogue of  the isomorphism of $SL_2(\bC)$ modules   $\wedge^2 \text{Sym}^3 W \to \text{Sym}^2( \text{Sym}^2 W)$ for $W = \bC^2$ (\cite[p.160]{Fulton-Harris}). Under this isomorphism, the element of $\wedge^2 D_3 V$ with coordinates $(z_0, \dots,z_5)$ goes to the symmetric bilinear form whose matrix with respect to the basis $X^2,XY,Y^2$ is 
\beq \bbm z_0 & z_1 & z_2 +z_5\\ z_1 & z_2-\tfrac{z_5}{2} & z_3 \\ z_2+ z_5 & z_3 & z_4 \bem. \eeq
\end{rem}

We consider the $GL_2(F)$-equivariant linear map $\Omega_4: D_4 V \to \text{Sym}^4(V^*)$  (see Definition  \ref{polardual}) represented by the matrix  
\[A_4  = \bbsm &&&&1\\ &&&-4&\\&&6&&\\&-4&&&\\1&&&& \besm \]
 of \eqref{eq:Am}. This matrix is invertible when char$(F) \neq 2, 3$, and hence gives  a 
$PGL_2(F)$-equivariant  isomorphism $\Omega_4: \bP(D_4 V) \to \bP(\text{Sym}^4(V^*))$:  The linear map $\Omega_4^{-1}$ can be regarded as a symmetric bilinear form on  the vector space  $\text{Sym}^4(V^*)$ of binary quartic forms. For $g \in GL_2(F)$, we have:
\beq \label{eq:A4_equivariance} \Omega_4^{-1} (g \cdot f) = \det(g)^{-4} \Omega_4^{-1}(f). \eeq
 Let $I(f)$ denote the quadratic form associated with the bilinear form $\Omega_4^{-1}$ :  If $f$ is represented as 
\[ f(X,Y)=z_0 Y^4-4 z_1 Y^3X+6 z_2 Y^2X^2 -4 z_3YX^3 + z_4X^4, \] 
then clearly $I(f)$ is given by 
\beq \label{eq:I_invariant} I(f) = (z_0 z_4 – 4 z_1 z_3 + 3 z_2^2)/3. \eeq
For a  binary form $f$,  by a $GL_2(F)$-invariant $\mathcal I(f)$ of weight $r$ and degree $d$ we mean a homogeneous function of degree $d$ in the coefficients of $f$ satisfying:
\[ \mathcal I( g \cdot f) = \det(g)^{-r} \mathcal I(f),\quad \text{for all } g \in GL_2(F) \text{ (see \cite[Definition 2.20]{Olver})}.\] 
For binary quartic forms, it is well known (see \cite[2.29-2.30]{Olver}) that there are two fundamental $GL_2(F)$-invariants: the  \emph{apolar invariant} $I(f)$ above of weight $4$ and degree $2$, and the \emph{catalecticant} $J(f)$ which is of weight $6$ and degree $3$:
\beq \label{eq:J_invariant}
J(f)=   \det \bbm z_0 & z_1 & z_2 \\ z_1 & z_2 & z_3 \\z_2 & z_3 & z_4 \bem. 
\eeq
Note that,
\beq \label{eq:IJmodular}
I(g \cdot f) = \det(g)^{-4} I(f), \qquad J( g\cdot f) = \det(g)^{-6} J(f). \eeq

Consider the  following $PGL_2(F)$-equivariant  maps \begin{enumerate}
\item the projection from $\mP_0$ to $\mH$ restricted to $\mQ$, represented by $(z_0, \dots, z_5) \mapsto (z_0, \dots, z_4)$ with respect to  the basis $E_0, \dots, E_5$ of $\wedge^2 D_3 V$ and $E_0, \dots, E_4$ of $\mH$.

\item the isomorphism $\mH \to \bP(D_4 V)$ of Lemma \ref{dec_lem} given by the identity map $(z_0, \dots, z_4) \mapsto (z_0, \dots, z_4)$ with respect to the bases $E_0, \dots, E_4$ on $\mH$ and $e_1^{[4-i]}e_2^{[i]}, i=0,\dots,4$ on $D_4 V$.
\item the isomorphism $\Omega_4: \bP(D_4 V)  \to  \bP(\text{Sym}^4(V^*))$  given by the matrix $A_4$ with respect to the bases $e_1^{[4-i]}e_2^{[i]}$ and $X^{4-i}Y^i$  for $i=0,\dots,4$ on $D_4 V$ and $\text{Sym}^4(V^*)$, respectively.

\end{enumerate}
Composing these three maps, we get a $PGL_2(F)$-equivariant map $\pi:\mQ \to \bP (\text{Sym}^4(V^*))$.  For a line $L$ of $\bP(D_3 V)$ with Pl\"{u}cker coordinates $p_{ij}$ (with respect to the basis $B_{ij}$) and coordinates $(z_0, \dots, z_5)$ with respect to the basis $E_0, \dots, E_5$, it follows that $\pi(L)$ is the quartic form 
\[ f_{L}=z_0 Y^4-4 z_1 Y^3X+6 z_2 Y^2X^2 -4 z_3YX^3 + z_4X^4. \] 
 We recall that  a line $L$  intersects  the tangent line to $C$ at $\nu_3(s,t)$ if and only if the Pl\"{u}cker coordinates $p_{ij}$ of $L$ satisfy \eqref{eq:incidence}. Rewriting this equation in the coordinates $z_0, \dots, z_5$ we get 
\begin{gather*}
 f_L(s,t)= z_0 t^4 -4z_1 t^3s +  6 z_2 t^2s^2  -4z_3 ts^3 + z_4 s^4 =0, \\  I(f_L)=(z_0z_4-4z_1z_3 +3 z_2^2)/3=  z_5^2.
\end{gather*}
We summarize the above discussion in the next lemma.
\begin{lem} \label{lem_pi}
 Let $\mF^+  \subset \bP(\text{Sym}^4(V^*))$ denote the subset of quartic forms $f$ with $I(f)$ a square in $F$. There is a $PGL_2(F)$-equivariant $2$-sheeted covering map $\pi:\mQ \to \mF^+$, where for $f \in \mF^+$ given by  
\[f=z_0 Y^4-4 z_1 Y^3X+6 z_2 Y^2X^2 -4 z_3YX^3 + z_4X^4, \] 
the inverse image $\pi^{-1}(f)$ consists of the  lines $\{L, L^{\perp}\}$ whose  coordinates with respect to the basis $E_0, \dots, E_5$ are given by $(z_0,\dots,z_4, \pm \sqrt{I(f)})$. \\
The following conditions are equivalent for a  line $L$ of $\bP(D_3 V)$ with $f = \pi(L)$ and $v_f=  \Omega_4^{-1}(f)$
\begin{enumerate}
\item[i)] $L$  intersects  the tangent line to $C$ at $\nu_3(s,t)$,
\item[ii)] $f(s,t)=0$,
\item[iii)] $v_f$ lies on the osculating hyperplane to $C_4$ at $\nu_4(s,t)$.
\end{enumerate}
\end{lem}
As an illustration of this lemma, we consider the  family of lines $\ell_\nu$ of $\bP(D_3 V)$ generated by  $(\nu B_1+B_3), B_0+B_2$ for $\nu \in \bF_q\setminus\{0, 1, 1/9\}$. This family of lines was treated in \cite{DMP2}. The lines $\ell_\nu$ have $f_{\ell_\nu}=\nu Y^4+(1-3 \nu) Y^2X^2+X^4$ for which $\sqrt{I(f)}=(3\nu+1)/6$ and 
\[ \sqrt{1 - \tfrac{1728}{\jmath(f)}} =  \tfrac{ (3\nu+1)^3}{9 (3\nu-1) (\nu-(2+\sqrt{3})^2)(\nu-(2-\sqrt{3})^2)}. \]

Figure \ref{maps} is a diagram showing  the sequence of maps involved, starting with the set of lines of $\bP(D_3(V))$ and ending with the set of binary quartic forms whose apolar invariant is a square.
\begin{figure}[h]
\centering
\begin{tikzpicture}[
  node distance=3cm and 3.5cm,
  every node/.style={align=center},
  box/.style={draw, rectangle, rounded corners, minimum width=3.5cm, minimum height=1cm},
  ->, >=stealth
]
\node[box]{All projective spaces carry an action of $PGL(V)$  where $V=F^2$ and all maps\\ between projective spaces are $PGL(V)$-equivariant.};
\end{tikzpicture}\\ \vspace{0.5cm}

\begin{tikzpicture}[
  node distance=3cm and 3.5cm,
  every node/.style={align=center},
  box/.style={draw, rectangle, rounded corners, minimum width=3.5cm, minimum height=1cm},
  ->, >=stealth
]
\node[box] (1) {Projective line $\bP(V)$};
\node[box, right=of 1] (2) {Twisted cubic $C$ of $\bP(D_3V)$};
\draw[->, thick] (1) --node[above] {Veronese map $\nu_3$} (2);
\end{tikzpicture}
\\ \vspace{0.5cm}

\begin{tikzpicture}[
  node distance=3cm and 3cm,
  every node/.style={align=center},
  box/.style={draw, rectangle, rounded corners, minimum width=3.5cm, minimum height=1cm},
  ->, >=stealth
]
\node[box] (line) {Lines of $\bP(D_3V)$ w.r.t \\ the twisted cubic $C$};
\node[box, right=of line] (Klein) {Klein quadric $\mQ$ in $\bP(\wedge^2 D_3V)$ \\containing the RNC   $C_4$ rep-\\resenting the tangent lines of $C$};
\node[box, below=of Klein] (projection) {Hyperplane $\mH$ spanned by $C_4$};
\node[box, below=of projection] (I) {$\bP(D_4 V)$};
\node[box, left=of I] (A4) {$\mF^+ \subset \bP(\text{Sym}^4(V^*))$ the subset of \\ binary quartic forms whose\\apolar invariant is a square};
\draw[->, thick] (Klein) -- node[right] {$\pi$ } node[left] {Projection of $\mQ$ from $\mP_0$ to  $\mH$\\
where $\mH=\bP(\wedge^2_+ (D_3V))$ and $\mP_0=\bP(\wedge^2_- (D_3V))$\\ and $\wedge^2 D_3V=\wedge^2_+ (D_3V) \oplus \wedge^2_- (D_3V)$ is the \\ 
decomposition into the selfdual and anti-selfdual\\ eigenspaces of the Hodge $\ast$ operator}(projection);
\draw[->, thick] (projection) -- node[right] {} node[left]{$PGL(V)$-equivariant  isomorphism  between \\ $\mH$ and $\bP(D_4V)$ coming from a $GL(V)$-\\equivariant isomorphism $\wedge^2_+ (D_3V) \simeq (D_4 V)  \otimes \wedge^2 V$ } (I);
\draw[->, thick] (I) -- node[above] {$\Omega_4$} node[below]{\\ \\  \\$PGL(V)$-equivariant  isomorphism  between \\ $\bP(D_4V)$ and $\bP(\text{Sym}^4(V^*))$ coming from the\\  
 bilinear form
(polarity) $\Omega_4$ induced by $C_4$}(A4);
\draw[->, thick] (line) -- node[above] {Pl\"ucker map } (Klein);
\draw[dashed,->,  thick] (line) to[out=225, in=135, looseness=1.5] (A4);
\end{tikzpicture}
\caption{Sequence of maps involved }\label{maps}
\end{figure}

 We end this section with the following geometric characterization of the vanishing of $J(f)$  (see \cite[Proposition 9.7]{Harris}): The catalecticant $J(f)$ vanishes if and only if $v_f=\Omega_4^{-1}(f)$ lies on the secant variety of the rational normal curve $C_4$. The points $v_f$ for which both  $I(f)=J(f)=0$ are the points on the tangent lines of $C_4$, and the points $v_f$ for which $J(f)=0$ but $I(f) \neq 0$ are the points not on $C_4$ lying on some secant line of $C_4$. The  points $v_f$ for which $I(f)=0$ but $J(f) \neq 0$ can also be characterized  geometrically: these points lie  on the intersection of osculating hyperplanes at $4$ distinct points $\nu_4(s_i,t_i),  i=1,2,3,4$ of $C_4$ with the property that the $\jmath$-invariant of 
 the $4$ points $(s_i,t_i)$ of $\bP^1$ is zero (the set of these $4$ points is `equianharmonic' in classical terminology).

 \subsection{Remarks on Atiyah's note \cite{Atiyah}} \label{Atiyah_note}
 As mentioned in \S \ref{introduction},  our geometric framework for studying the lines of $\bP^3$ w.r.t. the twisted cubic, is inherent in Atiyah's  one-page note \emph{A note on the tangents of a twisted cubic}, which is   divided into  the following $5$ terse points. In  particular, the $5$-th point uses the idea of studying a line $\ell$ of $\bP^3$ in terms of the 
 $\jmath$-invariant of the $4$ points $p_1, \dots, p_4$ of $\bP^1$ such that $\nu_3(p_i)=P_i$, where  $P_1, \dots, P_4$ are the $4$ points of $C_3$ whose tangent lines meet $\ell$.
  \begin{enumerate}
     \item It is  observed that the tangent lines to $C_3$ are represented by the curve $C_4$ lying on the Klein quadric $\mQ$.
     \item It is noted that $\mQ_0=\mQ \cap \mH$ (where  hyperplane $\mH$ is the hyperplane of $\bP^5$ spanned by $C_4$) must be the tangential variety of $C_4$, which is also the quadric threefold  associated  with the polarity on $\mH$ induced by $C_4$  [In terms of our coordinate system $(z_0, \dots, z_4)$, $\mQ_0$ is given by the locus $z_0z_4-4 z_1z_3 +3z_2^2=0$, and the quadratic form on $\mH$ associated to the polarity induced by $C_4$ is   $z_0z_4-4 z_1z_3 +3z_2^2$]. 
     \item It is observed that the chords of $C_3$ are represented by a Veronese surface $V$ contained in $\mQ$, and that the  projection from $\mP_0$ to $\mH$  carries the Veronese surface $V$ to the locus $U$ of points of intersection of osculating planes of $C_4$. [In our work this result can be seen as follows: a chord $L$  joining points $P, Q$ of $C_3$ intersects the tangent lines to $C_3$ at $P$ and $Q$ with multiplicity $2$, and does not intersect any other tangent line. So the quartic form $f(X,Y)=\pi(L)$ associated with $L$ has exactly two roots of multiplicity $2$ each. As shown in Lemma   \ref{lemDelta=0} (orbits 3) and 4)), such quartic forms correspond to points of intersection of osculating planes of $C_4$.]
     \item  A geometrical proof is given of the well-known result that  a necessary and  sufficient condition for four points on $C_4$ to be equianharmonic is that the pole of the $3$-space determined by them should lie on $\mQ_0$ [In our work this result can be seen as follows: in the next section we show that the $\jmath$-invariant of $4$  points $(s_i, t_i),  i=1,\dots,4$ of $\bP^1$ equals zero if and only if the apolar invariant $I(f)$ equals zero for the quartic form $f(X,Y)=\prod_{i=1}^4(Xt_i-Ys_i)=a_0 Y^4
     -4 a_1 Y^3X+6 a_2 Y^2X^2 -4 a_3YX^3 + a_4X^4$. Also, the pole of the $3$-space spanned by the points $P_i, i=1,\dots,4$ of $C_4$ given by 
     $(z_0, \dots, z_4)=(s_i^4,s_i^3t_i, s_i^2t_i^2, s_it_i^3,t_i^4)$ is precisely  $(a_0, \dots, a_4)$, and hence the result follows.] 
     \item 	The last result has the interpretation that   the tangents at four points of $C=C_3$ have a unique transversal if and only if, the four points are equianharmonic. [In our work this result can be seen (see Lemma \ref{lem_pi}) as follows: there is a unique line $L$ of $\bP^3$ which meets the tangent lines at given $4$ distinct points $(s_i^3,s_i^2t_i, s_it_i^2,t_i^3), i=1,\dots,4$ of $C_3$, if and only if $L=L^\perp$, i.e. $I(f_L)=0$ for the quartic form $f_L$ associated to $L$, or equivalently the $\jmath$-invariant of the $4$ roots $(s_i,t_i), i=1,\dots,4$ of $f_L$ is zero.]
 \end{enumerate}
We end this remark with a dictionary between our notation and Atiyah's:
  The projective space $\bP^n$ is denoted $S_n$ and the degree $n$ rational normal curve is denoted $C_n \hookrightarrow S_n$. The  Klein quadric $\mQ$ in $\bP^5$ representing lines of $\bP^3$ is denoted  $\Omega \hookrightarrow S_5$. The hyperplane $\mH$ spanned by $C_4$ is denoted  $\Sigma$. The intersection $\mQ_0=\mQ \cap \mH$ is denoted $\Omega_1=\Sigma \cap \Omega$,  and also as $I$. The point $\mP_0$ is denoted $L$ and defined to be the polar dual of the hyperplane $\Sigma$  with respect to the polarity on $S_5$ induced by the quadric $\Omega$.

\section{Binary quartic forms over $F$: invariants and $PGL_2(F)$ action} \label{sec4}
In this section, we obtain some results about  $PGL_2(F)$-orbits on binary quartic forms,  which will be used in the next section.  The form 
\[ f(X,Y)=z_0 Y^4-4 z_1 Y^3X+6 z_2 Y^2X^2 -4 z_3YX^3 + z_4X^4 \]
has repeated factors if and only if its discriminant $\Delta(f)=0$, where 
\beq \label{eq:disc} \Delta(f) =I(f)^3-J(f)^2, \eeq
(see \cite[p.29]{Olver}). It follows from \eqref{eq:IJmodular} that the rational function $\jmath(f)$ of degree $6$ in the coefficients of $f$  defined by
\beq \label{eq:jfdef}  
 \frac{\jmath(f)}{\jmath(f)-1728}=\frac{I(f)^3}{J(f)^2},  \eeq
being a ratio of invariants of weight $12$, has weight $0$, and hence is an  `absolute invariant' of $f$. We will show that $\jmath(f)$ equals the $\jmath$-invariant of the set of $4$ roots $(s_i,t_i)$ of 
\[ f(X,Y) = \prod_{i=1}^4(X t_i-Y s_i), \] 
(also see  \cite[Theorem 1.11.1]{Hirschfeld1}).  We use the standard identification of  $\bP^1(F)$  with the projective line $F \cup \{\infty \}$ where $(s,t) \in \bP^1(F)$ is identified with $t/s \in F \cup \{ \infty\}$.  We recall that  the cross-ratio $(P_1, P_2;P_3,P_4)$ of $4$ distinct and ordered points $(P_1, P_2, P_3, P_4)$ of $\bP^1(F)$ is
\beq \label{eq:crossratiodef} \lambda = (P_1, P_2;P_3,P_4)=  \frac{(P_3-P_1)(P_4-P_2)}{(P_3-P_2)(P_4-P_1)}. \eeq
 
The $\jmath$-invariant of an unordered set $\{P_1, P_2, P_3, P_4\}$ is given by the following two equivalent definitions:
\beq \label{eq:jdef}\frac{\jmath(\lambda)}{\jmath(\lambda) - 1728}= \frac{(\lambda^2-\lambda+1)^3}{\left((\lambda+1)(\lambda-2)(\lambda-\tfrac{1}{2}) \right)^2}, \qquad \jmath(\lambda) = \frac{256 (\lambda^2 - \lambda+1)^3}{\lambda^2 (\lambda-1)^2}. \eeq
We record the next two lemmas for easy reference (see, for example,  \cite[\S 6.1, \S 1.11]{Hirschfeld1}). First, we recall that the \emph{anharmonic group} $G_*$ is the subgroup of $PGL_2(F)$ given by 
\[G_*=\{ t \mapsto t^{\pm 1}, 1 - t^{\pm 1}, 1/(1-t^{\pm 1}) \}. \]
It is isomorphic to the symmetric group $S_3$ and is generated by the involution  $t \mapsto t^{-1}$ and the order $3$ element  $t \mapsto 1/(1-t)$.
\begin{rem}
    For the rest of the paper, by $\sigma_1, \dots, \sigma_4$ we mean the following permutations of the symmetric group $S_4$:
    \[ \sigma_1=(12)(34), \; \sigma_2=(13)(24),\;  \sigma_3=(12), \text{ and }\,  \sigma_4=(234).\]
\end{rem}
\begin{lem} \label{cross} The cross-ratio  $(P_1, P_2;P_3,P_4)$ gives a bijection between $F\setminus\{0,1\}$ and the set of  $PGL_2(F)$-orbits of ordered $4$-tuples $(P_1, \dots, P_4)$ of distinct points of the projective line over $F$. \\
The set of cross-ratios associated to the unordered set $\{P_1, P_2, P_3, P_4\}$  consists of the orbit of $\lambda = (P_1, P_2;P_3,P_4) $  under the  anharmonic group $G_*$. 
The homomorphism 
\[ \rho:S_4 \to G_*, \quad \rho(\sigma)( (P_1, P_2;P_3,P_4) )= (P_{\sigma(1)}, P_{\sigma(2)};P_{\sigma(3)},P_{\sigma(4)}), \]
 has kernel the normal subgroup $\langle \sigma_1, \sigma_2 \rangle \simeq \bZ/2 \bZ  \times \bZ/2 \bZ$. The quotient group $S_4/\mathrm{ker}(\rho)$  is isomorphic to $G_*$, and $\rho$ sends  the transposition $\sigma_3$ and the $3$-cycle $\sigma_4$  to  the involution  $\lambda \mapsto \lambda^{-1}$ and the order $3$ element  $\lambda \mapsto 1/(1-\lambda)$, respectively.
\end{lem}

\begin{lem} \label{jmath_lem}
The function $\jmath(\lambda)$ gives a bijection between the image $\jmath(F\setminus\{0,1\})$ and  the set of $G_*$-orbits on $F \setminus \{0,1\}$.  Each of these $G_*$ orbits has size $6$ with the following two exceptions:
  \beq \label{eq:j_0_1728} \jmath^{-1}(0) = \{\lambda: \lambda - \lambda^2=1\}, \quad  \jmath^{-1}(1728) = \{-1,2,1/2\}. \eeq
\end{lem}

The  cubic resolvent $R_f(X,Y)$  of $f$ is defined as follows. We assume $\Delta(f) \neq 0$: 
\beq \label{eq:resolvent} R_f(X,Y) =-4 Y^3 +3 I(f) YX^2 -J(f) X^3.  \eeq
It follows from the definitions of $I(f)$ and $J(f)$ that   $I(f) = I(XR_f)$ and $J(f)=J(X R_f)$.  Moreover, let $K \supset F$ be an extension field of $F$ such that $f$ has a linear factor $(Xt-Ys)$ over $K$.  We will show that there exists $g \in GL_2(K)$  such that $g \cdot f = \det(g)^{-2} X R_f(X,Y)$. If  $(s,t) \neq (0,1)$, then $X$ is a factor of $g_1 \cdot f$, where  $g_1=\bbsm  t/s &-1 \\ 1 &0 \besm \in SL_2(K)$. If  $(s,t)=(0,1)$ then $X$ is already a  factor of $f$ and we take $g_1 = \bbsm  1 &0 \\ 0 &1 \besm$. We consider $f_1 = g_1 \cdot f$. We can write  
\[ f_1(X,Y)=X(-4 \delta Y^3 +6 \alpha_1 XY^2 -4 \alpha_2  YX^2 +\alpha_3 X^3),\]
 for some $\alpha_1, \alpha_2, \alpha_3, \delta \in K$. Since $f$ and hence $f_1$ has no repeated factors, we know $\delta \neq 0$. For $g_2 = \bbsm  1 & 0 \\ -\frac{\alpha_1}{2\delta} & 1 \besm \in SL_2(K)$, we have  
\[ f_2 = g_2 \cdot f_1 = \delta X(-4 Y^3  +3 \zeta  YX^2 -\nu X^3), \] for some $\zeta, \nu \in K$. We note that $I(f_2) = \delta^2 \zeta$ and $J(f_2) =  \delta^3 \nu$. Since $g_1, g_2 \in SL_2(K)$, we get  (using  \eqref{eq:IJmodular}) that $I(f) = I(f_2)$ and $J(f) = J(f_2)$.  Thus $f_2 = \delta X(-4Y^3+ 3 I(f) \delta^{-2} X^2 Y- J(f) \delta^{-3} X^3)$. Finally, for $g_3=\bbsm \delta^{-1} & 0\\0&1 \besm$ we get
 $f_3=g_3 \cdot f_2 = \delta^2 X R_f(X,Y)$. Writing $g = g_3g_2g_1 \in GL_2(K)$ with $\det(g)=\delta^{-1}$,  we see
 \[ g \cdot f = \det(g)^{-2} X R_f(X,Y). \]

 If $u_1, u_2, u_3$ are the three roots of $R_f = -4 \prod_{i=1}^3(Y - u_i X)$, then the cross-ratio $\lambda=(\infty,u_1;u_2,u_3)$ equals $(u_3-u_1)/(u_2-u_1)$.  Using this in \eqref{eq:jdef} together with the fact that $\sum u_i=0$, $\sum u_iu_j= -3I(f)/4$ and $u_1u_2u_3=-J(f)/4$  gives 
\[ \frac{\jmath(\lambda) }{ \jmath(\lambda)-1728}  = \frac{I^3(f)}{J^2(f)}. \]  
Comparing this with \eqref{eq:jfdef} we see that $\jmath(f)  = \jmath(\lambda)$. Thus, $\jmath(f)$ is the same as  the  $\jmath$-invariant of the roots of $XR_f(X,Y)$, which in turn is the same as the $\jmath$-invariant of the roots of $f(X,Y)$.  (We remark that $\jmath(\lambda)$ is also the $\jmath$-invariant of the elliptic curve $Z^2X = R_f(X,Y)$, which can be rewritten in the standard form  $z^2=4 w^3 -\mathfrak{g}_2 w -\mathfrak{g}_3$ for $z = Z/X, w = -Y/X$ and $\mathfrak{g}_2=3 I(f), \mathfrak{g}_3=J(f)$). \\

 It will be useful to record here some more  subgroups of $S_4$:
\begin{align} \label{eq:g_i_def}
 \langle \sigma_1, \sigma_2, \sigma_4 \rangle   \simeq A_4, \quad   \langle \sigma_1, \sigma_2, \sigma_3  \rangle& \simeq  D_4, \quad \langle \sigma_1, \sigma_3 \rangle   \simeq \bZ/2 \bZ  \times \bZ/2 \bZ,  \\
\nonumber  \langle \sigma_3 \sigma_2 \rangle =\langle (1324)\rangle &\simeq \bZ /4 \bZ, \end{align}
 where $\bZ/m \bZ$ denotes the cyclic group of order $m$,  $D_4$ the dihedral group of size $8$, and $A_4$ the alternating group of size $12$.   For later use, we calculate the centralizer subgroup $Z_{S_4}(g)=\{h \in S_4: hg=gh\}$ of some elements of $S_4$:
\begin{align}  \label{eq:centralizer}
\nonumber Z_ {S_4}(\sigma_1) &=\langle \sigma_1, \sigma_2, \sigma_3  \rangle, \\
\nonumber Z_ {S_4}(\sigma_1 \sigma_3) &= \langle \sigma_1, \sigma_3  \rangle, \\
Z_ {S_4}(\sigma_4) &= \langle \sigma_4 \rangle, \\
\nonumber Z_ {S_4}(\sigma_3 \sigma_2)&= \langle \sigma_3 \sigma_2  \rangle.
\end{align}
We will now determine the $PGL_2(F)$-orbits of  sets $\{P_1, \dots, P_4\}$ of $4$ distinct unordered
points of the projective line over $F$. 
\begin{lem} \label{stabilizer_lemma} 
The $\jmath$-invariant gives a bijection between the set  $\jmath(F \setminus\{0,1\})\subset F$ and the $PGL_2(F)$-orbits of unordered $4$-tuples of distinct points of $F \cup \{\infty\}$. The stabilizer of an orbit is isomorphic to $A_4$, $D_4$ or $ \bZ/2 \bZ  \times \bZ /2 \bZ$,  according as the $\jmath$-invariant of the orbit equals $0, 1728$ or is different from $0, 1728$.\\
In more detail, let   $( (s_1,t_1),\dots, (s_4,t_4) )$ be  an ordered set of $4$  distinct points  of the projective line over $F$ with cross-ratio $\lambda_0$ and let 
$f(X,Y)=\prod_{i=1}^4(Ys_i - X t_i)$.  Let $\mathrm{Stab}_F(f)$ denote the stabilizer of $f(X,Y)$ in $PGL_2(F)$. The permutation representation of $\mathrm{ Stab}_F(f)$ on the roots $( (s_1,t_1),\dots, (s_4,t_4) )$ gives an injective homomorphism $\sigma: \mathrm{ Stab}_F(f) \to S_4$ mapping $g \mapsto \sigma_g$. The group $\mathrm{Stab}_F(f)$ is isomorphic via this homomorphism to: 
\begin{enumerate}
\item $ \langle \sigma_1, \sigma_2 \rangle \simeq  \bZ/2 \bZ  \times \bZ /2 \bZ$ if $\jmath(f) \neq 0, 1728$.
\item $ \langle \sigma_1, \sigma_2, \tilde \sigma_3  \rangle \simeq D_4$ if $\jmath(f) =1728$. Here, $\tilde \sigma_3$ equals $(12), (13), (14)$ according as $\lambda_0=-1,2,1/2$, respectively.
\item $\langle \sigma_1, \sigma_2, \sigma_4 \rangle   \simeq A_4$ if $\jmath(f) =0$.
\end{enumerate}
\end{lem}
\begin{proof}
By Lemma \ref{cross} we know that the cross-ratio gives a bijection between $F \setminus \{0,1\}$ and the $PGL_2(F)$-orbits of ordered $4$-tuples of distinct points of $F \cup \{\infty \}$. Two ordered $4$-tuples correspond to the same unordered set, if and only if their cross-ratios are in the same $G_*$-orbit. Thus, by Lemma \ref{jmath_lem}, the $\jmath$-invariant gives a bijection between $\jmath(F \setminus\{0,1\})$ and the set of $PGL_2(F)$-orbits of unordered $4$-tuples of distinct points of $F \cup \{\infty \}$.

Let $((s_1,t_1), \dots, (s_4,t_4))$ be the chosen ordering of the $4$ points, and let $\lambda_0$ denote the cross-ratio of this ordering. 
For $g \in  \text{Stab}_F(f)$, let $\sigma_g \in S_4$ be defined by  $g \cdot ((s_1,t_1), \dots, (s_4,t_4)) = ((s_{\sigma(1)},t_{\sigma(1)}), \dots, (s_{\sigma(4)},t_{\sigma(4)}))$. By Lemma \ref{cross}, the homomorphism  $\sigma:\text{Stab}_F(f) \to S_4$ is injective. 
The involutions  $\sigma_1=(12)(34)$ and $\sigma_2=(13)(24)$ are always in $\sigma(\text{Stab}_F(f))$.  To see this let $h_1$ be the unique element of $PGL_2(F)$ sending $(r_1, r_2, r_3)$ to $(r_2, r_1, r_4)$. The following  cross-ratios are equal
\[ \lambda_0 = (r_2, r_1; r_4, h_1(r_4)) =(r_1, r_2; h_1(r_4), r_4).\] 
Comparing the first and last tuples, we conclude $h_1(r_4)=r_3$. Thus, $\sigma_1=\sigma_{h_1}$.
 An identical argument shows that $\sigma_2=(13)(24)$ is in $\sigma(\text{Stab}_F(f))$. The  quotient group $S_4/\langle \sigma_1, \sigma_2 \rangle \simeq S_3$ has  $3$ elements of order $2$, namely the cosets represented by $(12), (13), (14)$ and two elements of order $3$ represented by $\sigma_4$ and $\sigma_4^2$.  
 
 Let $h_3$  be the unique element of $PGL_2(F)$ sending $(r_1, r_2, r_3)$ to $(r_2, r_1, r_3)$. By the properties of cross-ratio,  $\lambda_0=(r_2, r_1; r_3, h_3(r_4))$ whereas $(r_1, r_2; r_3, h_3(r_4))=\lambda_0^{-1}$. Therefore, $h_3(r_4)=r_4$ if and only if $\lambda_0=-1$. Thus, $\sigma_3=\sigma_{h_3}$ is in $\sigma(\text{Stab}_F(f))$ if and only if  $\lambda_0=-1$. An identical argument with $(12)$ replaced by $(13)$ and $(14)$, shows that $(13)  \in  \sigma(\text{Stab}_F(f))$ if and only if  $\lambda_0=2$, and 
$(14) \in \sigma(\text{Stab}_F(f))$ if and only if  $\lambda_0=1/2$. 

Let $h_4$  be the unique element of $PGL_2(F)$ sending $(r_2, r_3, r_4)$ to $(r_3, r_4, r_2)$. By the properties of cross-ratio,  
\[ \lambda_0 =(h_4(r_1), r_3; r_4, r_2), \quad (h_4(r_1), r_2; r_3, r_4) = 1 - \lambda_0^{-1}.\]
Therefore, $h_4(r_1)=r_1$ if and only if $\lambda_0 =1 - \lambda_0^{-1}$ which is equivalent to $\jmath(f)=0$. Thus $\sigma_4=\sigma_{h_4}$ is in $\sigma(\text{Stab}_F(f))$ if and only if  $\jmath(f)=0$. 
 \end{proof}
We remark that in \cite[Proposition 2.1]{Xiao}, a similar result as Lemma \ref{stabilizer_lemma} is obtained in the case  $F=\bC$, the field of complex numbers. We end this section with an expression for $I(f_1 f_2)$ where $f_1, f_2 \in \text{Sym}^2(V^*) $ are binary quadratic forms (which we will need in \S \ref{orbitsP3}). Let $f_1=\prod_{i=1}^2 (Ys_i-Xt_i) = u_0Y^2+u_1XY+u_2X^2$ and let 
$f_2=\prod_{i=1}^2 (Ys_i'-Xt_i')=v_0 Y^2+v_1 XY+v_2 X^2$. We recall that the homogeneous resultant of $f_1$ and $f_2$ is
\beq \label{eq:resultant} \text{Res}(f_1,f_2) = \textstyle \prod_{i=1}^2\prod_{j=1}^2 (t_is_j'-s_it_j') = \det \bbsm u_0&0&v_0&0\\ u_1&u_0&v_1&v_0\\u_2&u_1&v_2&v_1\\0&u_2&0&v_2 \besm. \eeq

\begin{lem} \label{I/res}
Let  $f_1=\prod_{i=1}^2 (Ys_i-Xt_i)$ and  $f_2 =\prod_{i=1}^2 (Ys_i'-Xt_i')$ be elements of  $\mathrm{Sym}^2(V^*)$. We assume $\mathrm{Res}(f_1,f_2) \neq 0$. Let  $\lambda$ denote the cross-ratio $((s_1,t_1),  (s_2,t_2); (s_1',t_1'), (s_2',t_2'))$. The expression $(\lambda+\lambda^{-1}-1)$ is independent of the orderings of each of the three sets $\{(s_1,t_1), (s_2,t_2) \}$,  $\{(s_1',t_1'), (s_2',t_2') \}$,  and $\{f_1,f_2\}$. We have:
\beq   \label{eq:I/res} I(f_1f_2) = \tfrac{1}{36} \, (\lambda+\lambda^{-1} - 1)\, \mathrm{Res}(f_1,f_2). \eeq
\end{lem}
\bep
Let $\lambda$ be the cross-ratio $((s_1,t_1),  (s_2,t_2); (s_1',t_1'), (s_2',t_2'))$.
Since the set   $\{\lambda, \lambda^{-1}\}$ is preserved by the subgroup $\langle \sigma_1, \sigma_2, \sigma_3\rangle$ of $S_4$, we see that $(\lambda+\lambda^{-1}-1)$ depends only on the unordered set $\{f_1,f_2\}$. For $g = \bbsm a & b\\ c& d \besm \in GL_2(F)$,   we have 
\begin{align*}
g \cdot f_1 &= \det(g)^{-2} \textstyle\prod_{i=1}^2( Y (as_i+bt_i) - X (cs_i+dt_i) ), \\
g \cdot f_2 &= \det(g)^{-2} \textstyle\prod_{i=1}^2( Y (as_i'+bt_i') - X (cs_i'+dt_i') ), \end{align*}
 and hence $\text{Res}(g \cdot f_1, g \cdot f_2)$ equals
\[  \det(g)^{-8} \textstyle\prod_{i=1}^2\prod_{j=1}^2 ((cs_i+dt_i) (as_j'+bt_j')-(as_i+bt_i)(cs_j'+dt_j'),\]
which simplifies to $\det(g)^{-4} \textstyle \prod_{i=1}^2\prod_{j=1}^2 (t_is_j'-s_it_j')$. Therefore,
\[ \text{Res}(g \cdot f_1, g \cdot f_2) = \det(g)^{-4} \, \text{Res}(f_1,f_2).\]
Since $I(g \cdot f_1f_2)=\det(g)^{-4} I( f_1f_2)$ by \eqref{eq:IJmodular}, we see that 
\[ \frac{I( f_1f_2)}{\text{Res}(f_1,f_2)}=\frac{I( g \cdot (f_1 f_2))}{\text{Res}(g \cdot f_1, g \cdot f_2)}.\]
Let $\tilde f_1 = XY$ and $\tilde f_2 = (Y-X)(Y-\lambda X)$. By Lemma \ref{cross} there exists a unique element of $PGL_2(F)$ which carries the 
ordered $4$-tuple of the projective line over $F$ represented by $((s_1,t_1),  (s_2,t_2), (s_1',t_1'), (s_2',t_2'))$ to the ordered $4$-tuple $(\infty, 0, 1, \lambda)$. Thus, there exists  $g \in GL_2(F)$ such that $g\cdot f_1= c_1 \tilde f_1 $ and $g \cdot f_2 = c_2  \tilde f_2$ for some $c_1, c_2 \in F^{\times}$.   Clearly, Res$(\tilde f_1, \tilde f_2) = \lambda$. We have $I(\tilde f_1 \tilde f_2)=(\lambda^2-\lambda+1)/36$ because 
\[ \tilde f_1 \tilde f_2 =  \tfrac{-1}{4} (-4Y^3X)+  \tfrac{1+\lambda}{6} (6 X^2Y^2) +  \tfrac{-\lambda}{4} (-4 YX^3).\] Therefore 
\[ \frac{I( f_1f_2)}{\text{Res}(f_1,f_2)}=\frac{I(  (c_1\tilde f_1)(c_2 \tilde f_2))}{\text{Res}(c_1 \tilde f_1, c_2 \tilde f_2)}=\frac{I( \tilde f_1 \tilde f_2)}{\text{Res}(\tilde f_1, \tilde f_2)} =  \frac{(\lambda+\lambda^{-1} - 1)}{36}. \]
\eep

\section{$PGL_2(q)$ orbits of binary quartic forms over $\bF_q$} \label{orbitsP4}
In this section, the field $F$ is the finite field $\bF_q$ with char$(\bF_q) \neq 2,3$. 
 We begin by fixing some notation:
\begin{description}
\item [$\overline{\bF_q}$] is  an algebraic closure of $\bF_q$,
\item[$G$] equals $PGL_2(q)$,
\item[$\phi$] is the Frobenius map $\phi(x)=x^q$ on $\overline{\bF_q}$,
\item[$\omega$] is a primitive cube root of unity in $\bF_{q^2}$,
\item[$\ep$] is a fixed quadratic non-residue in $\bF_q$,
\item[$\bF_q^\times$] is the multiplicative group of $\bF_q\setminus\{0\}$,
\item[$\gamma$] is  a generator of the cyclic group $\bF_q^\times$,
\item[$(\bF_q^\times)^2$] denotes the squares in $\bF_q^\times$ (and not a Cartesian product),
\item[$N_{q^2/q}(x)$] equals $x \phi(x)$ for $x \in \bF_{q^2}$,
\item[$N_{q^3/q}(x)$] equals $x \phi(x) \phi^2(x)$ for $x \in \bF_{q^3}$,
\item[$N_{q^4/q^2}(x)$] equals $x  \phi^2(x)$ for $x \in \bF_{q^4}$,
\item [$\mu \in \{\pm1\}$] is defined by $q \equiv \mu \mod 3$.
\end{description}
\subsection{$G$-orbits of binary cubic forms} Before going on to classify the $G$-orbits of binary quartic forms, we shall briefly discuss the $G$-orbits of binary cubic forms over $\bF_q$. Let $f$ be a binary cubic form over $\bF_q$ given by \[ f(X,Y)=y_0X^3-3y_1X^2Y+3y_2XY^2-y_3Y^3.\] Since $\Omega_3:\bP(D_3V) \to \bP(\text{Sym}^3(V^*))$ is a $G$-equivariant isomorphism, the $G$-orbit classification of binary cubic forms over $\bF_q$ is equivalent to $G$-orbit classification of points of $\bP(D_3V)$. The latter classification is given in 
 of \cite[Corollary 5]{Hirschfeld3}.  In the next lemma, we record  the  $G$-orbits of binary cubic forms and also the corresponding $G$-orbits of points in $\bP(D_3V)$.
\begin{lem}\label{cubic} 
The projective space of binary cubic forms over $\bF_q$ of size $q^3+q^2+q+1$ can be decomposed into the following five $G$-orbits. We also indicate in parentheses, the description of the corresponding orbit of points in $\bP(D_3V)$.
\begin{enumerate}
    \item $G \cdot X^3$ of size $(q+1)$ (corresponding to points of $C(\bF_q)$).
    \item $G \cdot X^2Y$ of size $q(q+1)$ (corresponding to points not on $C(\bF_q)$ but on some tangent line of $C(\bF_q)$).
    \item $G \cdot XY(X-Y)$ of size $(q^3-q)/6$ (corresponding to intersection points of the osculating planes at three distinct points of $C(\bF_q)$).
    \item $G \cdot X(X^2-\epsilon Y^2)$ of size $q(q^2-1)/2$ (corresponding to intersection points of the osculating planes to $C$ at $P,Q,R$, where $P$ is a point of $C(\bF_q)$ and $Q,R$ are two Galois conjugate points of $C(\bF_{q^2})$).
    \item $G \cdot (X-\theta Y)(X-\phi(\theta) Y)(X-\phi^2(\theta) Y)$, where $\theta \in \bF_{q^3}\setminus \bF_{q^2}$, of size $(q^3-q)/3$ (corresponding to intersection points of the osculating planes to $C$ at $P,Q,R$, where $P,Q,R$ are three Galois conjugate points of $C(\bF_{q^3})$).
\end{enumerate}
\end{lem}
We remark that the classification of binary cubic forms over $\bF_q$ into $GL_2(\bF_q)$ orbits has been studied in the number theory literature also, for example in the works \cite{Taniguchi3}, \cite{Taniguchi} and \cite{Taniguchi2} by   T. Taniguchi and F. Thorne. Using the orbit classification, they further determined the Fourier transform of the characteristic function of any $GL_2(\mathbb{F}_q)$-invariant subset of binary cubic forms in \cite{Taniguchi}. In \cite{Taniguchi2}, they exploited these Fourier transform formulas to yield `level of distribution' results for cubic and quartic fields. The authors obtain results not only on binary cubic forms, but also on some related pre-homogeneous vector spaces (vector spaces with an action of an algebraic group, such that there is a  dense open orbit).

\subsection{$G$-orbits of binary quartic  forms with discriminant zero}
 The isomorphism  $\Omega_4: \bP(D_4 V) \to \bP(\text{Sym}^4(V^*))$ is $G$-equivariant, and hence the decomposition of binary quartic forms over $\bF_q$ into $G$-orbits is equivalent to the decomposition of $\bP(D_4 V)$ into $G$-orbits.  Given a binary quartic form
\[f(X,Y)=z_0 Y^4-4 z_1 Y^3X+6 z_2 Y^2X^2 -4 z_3YX^3 + z_4X^4, \] 
let  $v_f=\Omega_4^{-1}(f)=\sum_{i=0}^4 z_i E_i$. We recall from Lemma \ref{lem_pi} that $(Xt-Ys)$ is a factor of $f(X,Y)$ if and only if $v_f$ lies in the osculating hyperplane to $C_4$ at $\nu_4(s,t)$.

\begin{definition} \label{Fi_def} Given $f(X,Y) \in \bP(\mathrm{Sym}^4(V^*))$, let $F \supset \bF_q$ denote the splitting field of $f(X,Y)$. We decompose  $\bP(\mathrm{Sym}^4(V^*))$ as $\mF^0 \cup \mF_1 \cup \mF_2 \cup \mF_2' \cup \mF_4 \cup \mF_4'$ as below. The set  $\mF^0$ consists of forms with $\Delta(f)=0$. The sets $\mF_i$ for $i=1,2,4$ consist of forms $f$ with $\Delta(f) \neq 0$ and having  exactly $i$ linear factors defined over $\bF_q$. The sets $\mF_4'$ and $\mF_2'$ consist of forms $f$ with $\Delta(f) \neq 0$,  having no linear factors over $\bF_q$, and whose splitting field $F$ is $\bF_{q^2}$ and $\bF_{q^4}$, respectively. Then
 \begin{enumerate} 
\item   $|\mF_{4}|=\tbinom{q+1}{4}=\tfrac{(q-2)|G|}{24}$ is the number of ways of picking $4$ distinct linear forms  over $\bF_q$.
\item $|\mF_4'|=\tbinom{ (q^2-q)/2 }{2}= \tfrac{(q-2)|G|}{8}$ is the number of ways to pick two distinct irreducible quadratic forms over $\bF_q$.
\item  $|\mF_{2}|=\tbinom{q+1}{2}\tfrac{q^2-q}{2}=\tfrac{q|G|}{4}$ is the number of ways of picking $2$ distinct linear forms  and an irreducible quadratic form over $\bF_q$.
\item $|\mF_2'|=\tfrac{q^4-q^2}{4}= \tfrac{q |G|}{4}$ is the number of irreducible quartic forms over $\bF_q$.
\item  $|\mF_{1}|=(q+1)\tfrac{q^3-q}{3}=\tfrac{(q+1)|G|}{3}$ is the number of ways of picking a linear form and an irreducible cubic form over $\bF_q$.
\item  From the above parts, we see that   $|\mF_1|+|\mF_2|+|\mF_2'|+|\mF_4|+|\mF_4'|$ equals $(q^4-q^2)$. Therefore, $|\mF^0|=(q^3+2q^2+q+1)$.
\end{enumerate}
\end{definition}
We will determine the $G$-orbits of $\mF^0, \mF_1, \mF_2, \mF_4$ and $\mF_2', \mF_4'$. We start with  the set $\mF^0$. It follows from the expression $\Delta(f) =I(f)^3-J(f)^2$ that for $f \in \mF^0$, either $I(f), J(f)$  both vanish,  or both of them are non-zero.
\begin{lem}  \label{lemDelta=0} The set $\mF^0$ of size $(q^3+2q^2+q+1)$ consists of  the following $G$-orbits. The first two orbits below have both $I$ and $J$-invariants zero. The remaining orbits have non-zero values for both $I$ and $J$-invariants. We also indicate in parentheses, the description of the corresponding orbit in $\bP(D_4V)$.
\begin{enumerate}
\item $G \cdot X^4$  of size $(q+1)$ (corresponding to points of $C_4(\bF_q)$).  
\item $G \cdot X^3Y$ of size $q(q+1)$ (corresponding to  points not on $C_4$ but  on some  tangent line  of $C_4(\bF_q)$).  
\item $G \cdot X^2Y^2$  of size $(q^2+q)/2$ (corresponding to intersection points of the osculating planes of $C_4$ at two  distinct points $P, Q$ of $C_4(\bF_q)$).
\item $G \cdot (X^2-\epsilon Y^2)^2$ of size $(q^2-q)/2$  (corresponding to the intersection points of the osculating planes of $C_4$ at  a pair of Galois conjugate points $P, Q$ of $C_4(\bF_{q^2})$). 
\item $G \cdot X^2Y(Y-X)$ of size $(q^3-q)/2$ (corresponding to intersection points of the  osculating plane to  $C_4$ at $P$ and  the osculating hyperplanes of $C_4$ at $Q$ and $R$   where $P, Q, R$ are $3$ distinct points of $C_4(\bF_q)$).
\item $G \cdot X^2(X^2-\epsilon Y^2)$ of size $(q^3-q)/2$ (corresponding to intersection points of the  osculating plane to  $C_4$ at $P$ and  the osculating hyperplanes of $C_4$ at $Q$ and $R$   where $P$ is a point of $C_4(\bF_q)$, and $Q, R$ are Galois conjugate  points of $C_4(\bF_{q^2})$).
\end{enumerate}
\end{lem}
\begin{proof}
We decompose $\mF^0$ by considering the  partitions of $4$, corresponding to the multiset of rational roots of $f$ and their multiplicities. The $(q+1)$ forms  having  a single factor of multiplicity $4$ form an orbit $G \cdot X^4$  whose stabilizer is the group $\{ t \mapsto c+dt\}$ fixing $\infty$. The $(q+1)q$ forms having   two distinct factors  of  multiplicities $3$ and $1$ form an orbit $G \cdot X^3Y$ whose stabilizer is the group $\{ t \mapsto dt\}$ fixing both $0$ and $\infty$. The $\tbinom{q+1}{2}$ forms having  two distinct rational factors  of  multiplicity $2$ each form an orbit $G \cdot X^2Y^2$  with stabilizer group $\{t \mapsto d t^{\pm 1}\}$ preserving the set $\{0, \infty\}$. Similarly, the $\tbinom{q}{2}$ forms having  two distinct Galois conjugate factors  of  multiplicity $2$ each form an orbit $G \cdot (X^2- \epsilon Y^2)^2$ with stabilizer group of size $2(q+1)$ preserving  $\{ \pm \sqrt \ep\}$. The $(q+1)\tbinom{q}{2}$ forms having  three distinct rational factors  of  multiplicity $2, 1$ and $1$ form an orbit $G \cdot X^2Y(Y-X)$ whose stabilizer is the group of order $2$ generated by the involution $t \mapsto(1-t)$. Similarly, the  $(q+1)\tbinom{q}{2}$ forms having a rational factor  of  multiplicity $2$ and two Galois conjugate factors of multiplicity $1$  each, form an orbit
$G \cdot X^2(X^2- \epsilon Y^2)$ whose stabilizer is the group of order $2$ generated by the involution $t \mapsto \pm t$. 
\end{proof}
\subsection{$G$-orbits of binary quartic  forms with nonzero discriminant}
Let $f(X,Y)$ be a quartic form over $\bF_q$ with $\Delta(f) \neq 0$.  If $F$ denotes the splitting field of $f(X,Y)$, then Stab$_F(f)$ has been determined in Lemma \ref{stabilizer_lemma}.  We will determine the group Stab$(f) = \text{Stab}_{F}(f) \cap PGL_2(q)$. 

\begin{definition} \label{restricted_ordering}
For $f  \in \mF_4', \mF_2', \mF_{2}, \mF_{1}$, by a \emph{restricted ordering} of the roots of $f$, we mean orderings of the following form:
 \begin{description}
 \item  [$f\in\mF_{4}:$ ]\emph{ no restriction on the ordering $(r_1, r_2, r_3, r_4)$}.
\item[$f\in \mF_4':$ ]  $(r_1, r_2, r_3, r_4)=(r_1, \phi(r_1), r_3, \phi(r_3))$.
\item  [$f\in \mF_{2}:$ ]  $(r_1, r_2, r_3, r_4)=(r_1,r_2, r_3, \phi(r_3))$\emph{ with }$r_1, r_2 \in PG(1,q)$.
\item[$f\in \mF_2':$ ]   $(r_1, r_2, r_3, r_4)=(r_1, \phi^2(r_1), \phi(r_1), \phi^3(r_1))$.
\item  [$f\in\mF_1:$ ]   $(r_1, r_2, r_3, r_4)=(r_1,r_2,\phi(r_2),\phi^2(r_2))$ \emph{ with } $r_1 \in PG(1,q)$.
\end{description}
We note that the action of  $G=PGL_2(q)$ carries a restricted ordering of points to another restricted ordering of points of the same type. \\
 For $f$ with $\Delta(f) \neq 0$,   let $\text{ord}(f)$ denote the set of all restricted orderings of the roots of $f$. 
  \end{definition}

  The Frobenius map $\phi$ acts on $\text{ord}(f)$ by $\phi(r_1, r_2, r_3, r_4)=(\phi(r_1), \dots, \phi(r_4))$.
  \begin{lem}  \label{ZS4}
 The action of the Frobenius map $\phi$ on $\mathrm{ord}(f)$ is realized  by the permutation  $\sigma_{\phi} \in S_4$:    $\phi(r_1,\dots,r_4)=(r_{\sigma_\phi(1)}, \dots, r_{\sigma_\phi(4)})$ defined  by 
 \[ \sigma_{\phi} = \begin{cases}  \mathrm{trivial} &\text{if $f \in \mF_4$},\\ 
 \sigma_1=(12)(34) &\text{ if $f \in \mF_4'$,}\\
   \sigma_1\sigma_3=(34) &\text{ if $f \in \mF_2$,}\\
      \sigma_3\sigma_2=(1324) &\text{ if $f \in \mF_2'$,}\\
   \sigma_4=(234) &\text{ if $f \in \mF_1$.} \end{cases}\]
   
    The set  $\mathrm{ord}(f)$ forms a single orbit  $Z_{S_4}(\sigma_\phi) \cdot (r_1, r_2, r_3, r_4)$ under the  centralizer subgroup of $\sigma_\phi$ in $S_4$ given by 
    \beq \label{eq:ZS4} Z_{S_4}(\sigma_\phi)=\begin{cases} 
    S_4 &\text{ if $f \in \mF_4$,} \\
\langle \sigma_1, \sigma_2, \sigma_3  \rangle \simeq D_4 &\text{ if $f \in \mF_4'$,} \\
\langle \sigma_1,  \sigma_3  \rangle  \simeq \bZ/2 \bZ  \times \bZ/2 \bZ   &\text{ if $f \in \mF_2$,} \\

\langle \sigma_3 \sigma_2 \rangle = \langle (1324)  \rangle\simeq \bZ/4\bZ   &\text{ if $f \in \mF_2'$,} \\
\langle \sigma_4 \rangle = \langle (234)  \rangle \simeq \bZ/3 \bZ &\text{ if $f \in \mF_1$.}
\end{cases}\eeq

\end{lem}
\begin{proof}
In each of the cases  $f  \in \mF_4, \mF_4', \mF_2', \mF_{2}, \mF_{1}$, it is clear that the permutation $\sigma_\phi$ is as asserted, for example, if $f \in \mF_2$ and $(r_1, r_2, r_3, r_4)=(r_1,r_2, r_3, \phi(r_3))$ then $\sigma_{\phi}=(34)$.  Further, if  $(r_1, \dots, r_4)$ and $(r_1', \dots, r_4')$  are in $ \text{ord}(f)$, let $\sigma \in S_4$ such that  $(r_1',\dots,r_4') = (r_{\sigma(1)}, \dots, r_{\sigma(4)})$.
It follows that 
\[ r_{\sigma_\phi \sigma(i)}= \phi(r_{\sigma(i)})  =\phi(r_i')= r_{\sigma_{\phi}(i)}'
= r_{\sigma \sigma_{\phi}(i)}
,\]
and hence $\sigma \in Z_{S_4}(\sigma_\phi)$. Finally, \eqref{eq:ZS4} immediately follows from
\eqref{eq:centralizer}.\\

\end{proof} 
   
\begin{definition}
Let $\rho:S_4 \to G_*$ be the homomorphism defined in Lemma \ref{cross}. 
The subgroup  $\rho(Z_{S_4}(\sigma_\phi))$ of the anharmonic group $G_*$ will be denoted $H_i$ in the case of $\mF_i, i=1,2,4$  and $H_i'$ in the case $\mF_i', i=2,4$. 
\end{definition}
\begin{prop}
The groups $H_i, H_i'$ in the definition above equal:
\begin{align}
\nonumber H_1&=\langle \lambda \mapsto  1/(1-\lambda) \rangle, \\
H_2&=\langle  \lambda \mapsto  1/\lambda\rangle, \\
\nonumber H_4&=G_*, \\
\nonumber H_4'&=H_2'=H_2.
\end{align}
\end{prop}
\begin{proof} The groups $H_i, H_i'$ are calculated as follows:  From Lemma \ref{cross}, we know  that  i) ker$(\rho) = \langle \sigma_1, \sigma_2 \rangle$, ii)$\rho(\sigma_3)$ is the map $t \mapsto t^{-1}$, and iii) $\rho(\sigma_4)$ is the map $t \mapsto 1/(1-t)$. Using this in \eqref{eq:ZS4} we get $\rho(Z_{S_4}(\sigma_\phi))$ equals i) $G_*$ in case of $\mF_4$, ii) $H_1$ in case of $\mF_1$, and iii) $H_2$ in case of $\mF_4', \mF_2, \mF_2'$.\end{proof} 

We define some subsets $ \tilde \mN_i = \mN_i \cup \mN_i^{1728} \cup \mN_i^{0}$  of $\overline{\bF_q}$. We will prove in Proposition \ref{Lambda_prop} below that the cross-ratio of a restricted ordering of the roots of $f \in \mF_i, \mF_i'$ will lie in $\tilde \mN_i$ where
\begin{align} \label{eq:Ni_def}
\nonumber  \tilde \mN_4&=  \bF_q\setminus\{0,1\}, \\
\tilde \mN_2&= \{\lambda \in \bF_{q^2}\setminus\{1\} : N_{q^2/q}(\lambda)=1 \},\\
\nonumber \tilde \mN_1&=\{\lambda \in \bF_{q^3} : \lambda^{q+1}- \lambda^q+1=0 \},\\
\nonumber \mN_i^{1728} &=\tilde \mN_i \cap \{-1, \tfrac{1}{2}, 2\},\\
\nonumber \mN_i^{0} &=\tilde \mN_i \cap \{ -\omega, -\omega^2\}, \\
\nonumber \mN_i&=\tilde \mN_i \setminus\{ -1,  \tfrac{1}{2}, 2,-\omega, -\omega^2\}. 
\end{align}
We note that $|\tilde \mN_i|=(q+2-i)$: this is obvious for $i=2, 4$, and for $i=1$, we observe that for $\bF_{q^3}=\bF_q[\theta]$, the $(q+1)$ values $\lambda_x= (x, \theta; \phi(\theta), \phi^2(\theta) )$ for  $x \in \bF_q \cup \{\infty\}$ are in $\mN_1$ because, 
\[ (\lambda_x)^q =(x,\phi(\theta);  \phi^2(\theta), \theta) =\frac{1}{1-\lambda_x},\] 
and the condition $\lambda^q = 1/(1-\lambda)$ is the same as $\lambda^{q+1} - \lambda^q+1=0$. Clearly, 
\begin{align*}
|\mN_4^{1728}|=|\{-1,1/2,2\}|=3,\; |\mN_2^{1728}|=|\{-1\}|=1, \;|\mN_1^{1728}|=0, \\
|\mN_4^0|=|\mN_1^0|=\mu+1,\quad |\mN_2^0|=1-\mu. \end{align*}
Therefore, $|\mN_i|=|\tilde\mN_i|-|\mN_i^{1728}|-|\mN_i^0|$ equals
\[ |\mN_4| =(q-6-\mu), \quad  |\mN_2| = (q-2+\mu), \quad   |\mN_1| = (q-\mu).\]
The notation reflects the fact that $\mN_i^{1728} \subset \jmath^{-1}(1728)$ and 
$\mN_i^{0} \subset \jmath^{-1}(0)$.  We note that the groups $H_i$ preserve the sets $\mN_i, \mN_i^{1728}, \mN_i^{0}$  and $\tilde \mN_i$.
\begin{prop} \label{Lambda_prop}
Given $f \in \mF_i$ (resp. $\mF_i'$) the set $\Lambda_f$ of cross-ratios of restricted orderings $\mathrm{ord}(f)$ of the roots of $f$,  forms a single $H_i$-orbit  (resp $H_i'$-orbit)  in $\tilde \mN_i$. \\
Let $\mF_i/G$  (resp. $\mF_i'/G$) denote the set of $G$-orbits on $\mF_i$  (resp. $\mF_i'$).  The map that takes $f \in \mF_i$ (resp $f \in \mF_i'$) to   $\Lambda_f$   in  $\tilde \mN_i$ induces a bijective map
\begin{align*} 
\mF_i/G  &\leftrightarrow \tilde\mN_i/H_i, \quad i=1,2,4.\\
\mF_i'/G  &\leftrightarrow \tilde\mN_i/H_i',  \quad i=2,4.
  \end{align*}
\end{prop}
\bep
Let $f \in \mF_i$ (resp. $\mF_i'$). First, we show that $\Lambda_f$ forms a single $H_i$-orbit  (resp $H_i'$-orbit)  in $\overline{\bF_q}$. We consider the map ord$(f) \to \Lambda_f$ that takes a restricted ordering $(r_1, r_2, r_3, r_4)$ of the roots of $f$ to its cross-ratio $\lambda_f=(r_1, r_2; r_3, r_4)$. As shown in Lemma \ref{ZS4}, the set ord$(f)$ is the orbit  $Z_{S_4}(\sigma_\phi) \cdot (r_1, r_2, r_3, r_4)$ and $\rho(Z_{S_4}(\sigma_\phi))$ equals $H_i$ if $f \in \mF_i$ and $H_i'$ if $f \in \mF_i'$. Thus, $\Lambda_f = H_i \cdot \lambda_f$ if $f \in \mF_i$ and $\Lambda_f = H_i' \cdot \lambda_f$  if $f \in \mF_i'$.\\
 Next we show that $\lambda_f \in \tilde \mN_i$ for $f \in \mF_i \cup \mF_i'$.
\begin{enumerate}
\item If $f \in \mF_4$ then clearly  $\lambda_f \in \bF_q \setminus\{0,1\}=\tilde \mN_4$.
\item  If  $f\in \mF_4'$ then   $\lambda_f=N_{q^2/q}(\beta)$ where $\beta = \tfrac{r_3-r_1}{r_3-\phi(r_1)}$. Hence, $\lambda_f \in  \tilde \mN_4$. 
\item If  $f\in \mF_{2}$  then  $\lambda_f=\tfrac{\beta}{\phi(\beta)}$ where $\beta=(r_3-r_1)( \phi(r_3)- r_2)$. Hence, $\lambda_f\in  \tilde \mN_2$. 
\item If  $f\in \mF_2'$  then $\lambda_f=\tfrac{\beta}{\phi(\beta)}$ where $\beta=N_{q^4/q^2}(\phi(r_1)-r_1)$. Hence, $\lambda_f\in  \tilde \mN_2$. 
\item If  $ f \in\mF_{1}$  then  $\lambda_f=\tfrac{-\beta}{\phi(\beta)}$ where $\beta=(\phi(r_2)-r_1)(\phi^2(r_2)-r_2)$ satisfies 
\[\lambda_f^q(1-\lambda_f)=\tfrac{\beta+\phi(\beta)}{-\phi^2(\beta)}=\tfrac{r_2(\phi^2(r_2)-\phi(r_2)) -r_1(\phi^2(r_2)-\phi(r_2))}{r_2(\phi^2(r_2)-\phi(r_2)) -r_1(\phi^2(r_2)-\phi(r_2))}=1.\]
Thus, $\lambda_f \in  \tilde \mN_1$.
\end{enumerate}

Next we show that for $f \in \mF_i$, the map $f \mapsto H_i \cdot \lambda_f$ induces an injective map  $\mF_i/G\to \tilde \mN_i/H_i$ and similarly for $f \in \mF_i'$, the map $f \mapsto H_i' \cdot \lambda_f$ induces an injective map  $\mF_i'/G\to \tilde \mN_i/H_i'$. Suppose   $f, \tilde f \in \mF_i$ (resp. $\mF_i'$) with $\Lambda_f = \Lambda_{\tilde f}=H_i \cdot \lambda $ (resp $H_i' \cdot \lambda$) then there exist $(r_1, \dots, r_4) \in \text{ord}(f)$ and $(\tilde r_1, \dots, \tilde r_4) \in \text{ord}(\tilde f)$  such that $\lambda = (r_1,r_2; r_3, r_4)=(\tilde r_1,\tilde r_2;\tilde r_3, \tilde r_4)$. By Lemma \ref{cross} there exists a unique   $g \in PGL_2(\overline{\bF_q})$ which carries $(r_1,r_2,r_3,r_4)$ to $(\tilde r_1,\tilde r_2,\tilde r_3,\tilde r_4)$. However, from the definition of a restricted ordering, it follows that $\phi(g)$ also has  the same property, and hence $g \in PGL_2(q)$ by the uniqueness of $g$.  This shows that $\tilde f \in G \cdot f$. Thus, we have shown that the map $G \cdot f \mapsto \Lambda_f$ gives  injective maps   $\mF_i/G  \rightarrow \tilde \mN_i/H_i$ and  $\mF_i'/G \rightarrow \tilde \mN_i/H_2$. \\
It now remains to show that the maps $\mF_i/G \to \tilde \mN_i/H_i$ and $\mF_i'/G \to \tilde \mN_i/H_i'$ are surjective. In other words, each orbit $H_i \cdot \lambda$ in $\tilde \mN_i$ is realized as $H_i \cdot \lambda_f$ for some $f \in \mF_i$  and each orbit in $H_i' \cdot \lambda$ in $\tilde \mN_i$ is realized as $H_i' \cdot \lambda_f$ for some $f \in \mF_i'$. We do this case by case: \\

{\bf \noindent  Case of $\mF_4$ } \hfill \\  Given $\lambda \in \bF_q \setminus\{0,1\} = \tilde \mN_4$, let
\beq \label{eq:psilambda}  \psi_\lambda = XY(Y-X)(Y-\lambda X).
 \eeq
Then the form $f =  \psi_\lambda$ has $H_4 \cdot \lambda_f = H_4 \cdot \lambda$.

{\bf \noindent Case of $\mF_4'$} \hfill \\
The map $g(t) = \frac{t-\sqrt\ep}{t+\sqrt\ep}$ maps $\bF_{q^2}\setminus \{ \bF_q \cup  \{\pm \sqrt\ep\}\}$ bijectively to $\{x \in \bF_{q^2}^\times: N_{q^2/q}(x) \neq 1\}$, which in turn is mapped onto $\bF_q\setminus\{0,1\} = \tilde \mN_4$ by the map $N_{q^2/q}:\bF_{q^2}^{\times} \to \bF_q^{\times}$. Taking  $(r_1, r_2, r_3, r_4)=(\sqrt \ep, -\sqrt \ep, \alpha, \phi(\alpha))$ for $\alpha \in \bF_{q^2}\setminus \{ \bF_q \cup  \{\pm \sqrt\ep\}\}$ we get  $(r_1, r_2; r_3, r_4) = N_{q^2/q}(\tfrac{\alpha - \sqrt\ep}{\alpha+\sqrt\ep})$.  Let
\beq \label{eq:psialpha}  \psi'_\alpha = (Y^2- \ep X^2) ( Y - \alpha X)(Y - \phi(\alpha)X) , \quad  \alpha \in \bF_{q^2}\setminus\{\bF_q \cup  \{\pm \sqrt\ep\}\}. \eeq
It follows that given $\lambda \in \tilde \mN_4$,   there exists $\alpha \in \bF_{q^2}\setminus \{ \bF_q \cup  \{\pm \sqrt\ep\}\}$ such that $H_4 \cdot \psi'_\alpha = H_4 \cdot \lambda$.

{\bf \noindent  Case of $\mF_1$} \hfill \\
As noted above,  for $\bF_{q^3}=\bF_q[\theta]$, we have 
\[ \tilde \mN_1 = \{ (x, \theta; \phi(\theta), \phi^2(\theta) ) \colon  x \in \bF_q \cup \{\infty\} \}. \]
Hence, for each $\lambda \in \tilde \mN_1$, the orbit $H_1 \cdot \lambda$ can be realized as $H_1 \cdot \lambda_f$ as $f$ runs over 
\beq \label{eq:etar}  \eta_{(s,t)} = (Ys-tX) (Y -\theta X)(Y - \phi(\theta)X)(Y - \phi^2(\theta)X) , \qquad  (s,t)  \in PG(1,q). \eeq

{\bf \noindent  Case of $\mF_2$} \hfill \\
Taking  $(r_1, r_2, r_3, r_4)$ to be $(\infty, 0, r-\sqrt\ep,r+\sqrt\ep)$ for $r \in  \bF_q$, we get  $(r_1, r_2; r_3, r_4) = \tfrac{r + \sqrt\ep}{r-\sqrt\ep}$. The map $g(t) = \tfrac{t+\sqrt\ep}{t-\sqrt\ep}$ maps $\bF_q$ bijectively to $\tilde \mN_2$. 
Thus, each $H_2 \cdot \lambda$ in $\tilde \mN_2$ can be realized as 
as  $H_2 \cdot \lambda_f$  as $f$ runs over 
\beq \label{eq:upsilonr}  \upsilon_r = XY ( (Y -rX)^2 -\ep X^2), \qquad  r  \in \bF_q. \eeq

{\bf \noindent  Case of $\mF_2'$ for $q \equiv 1 \mod 4$} \hfill \\
 If $q \equiv 1 \mod 4$, let $\bF_{q^4} = \bF_q[\theta]$ where $\theta^4=\gamma$ and $\gamma$ is a generator of the cyclic group $\bF_q^{\times}$. We note that $\phi(\theta) = \imath \theta$ where $\imath \in \bF_q$ is a square root of $-1$. The map  $g(t) = \tfrac{t -  \imath \theta^2/2}{t + \imath \theta^2/2}$  maps $\bF_q$ bijectively to $\tilde \mN_2$. Therefore, it suffices to exhibit $f \in \mF_2'$ for which $\lambda_f$ is of the form $\tfrac{r -  \imath \theta^2/2}{r + \imath \theta^2/2}$.
 We claim $f = \upsilon'_r$ defined below, has this property: 
\beq \label{eq:upsilonr1mod4}   \upsilon'_r= \begin{cases}
 (Y^2- r  X^2)^2 - \gamma X^4 -  4 \sqrt{\gamma r}  X^3Y &\text{if $\ep r$ is a square in $\bF_q$,} \\
  (Y^2-\gamma r X^2)^2 - \gamma^3 X^4 -  4 \gamma^2 \sqrt{r}   X^3Y &\text{if $r \in (\bF_q^\times)^2$}.
\end{cases} \eeq

To see this, we consider $(r_1, \phi^2(r_1),\phi(r_1), \phi^3(r_1))$ with

\begin{enumerate}
\item $r_1=\theta+  t \theta^2$ where $t \in \bF_q$.
Here  $\phi(r_1) - r_1=(\imath-1)\theta - 2 t \theta^2$ and hence $\beta=N_{q^4/q^2}(\phi(r_1)-r_1)$ equals $2 \theta^2 (\imath + 2 t^2 \theta^2)$. Therefore, $\lambda_f=\tfrac{\beta}{\phi(\beta)}$ equals $\tfrac{\gamma t^2 + \imath \theta^2/2}{\gamma t^2 - \imath \theta^2/2}=\tfrac{\gamma (\imath t)^2 - \imath \theta^2/2}{\gamma (\imath t)^2 + \imath \theta^2/2}$. 
The corresponding form in $\mF_2'$ is:
\[  \upsilon'_{\ep t^2}=(Y^2-\gamma t^2 X^2)^2 - \gamma X^4 -  4 \gamma t  X^3Y.\]

\item   $r_1=t \theta^2+\theta^3$  where $t \in \bF_q^{\times}$.
Here,  we have $\phi(r_1) - r_1=-\theta^2( 2t+ (\imath+1)\theta )$ and hence $\beta=N_{q^4/q^2}(\phi(r_1)-r_1)$ equals $4 \gamma (t^2 - \imath \theta^2/2)$. Therefore, $\lambda_f=\tfrac{\beta}{\phi(\beta)}$ equals $\tfrac{ t^2 - \imath \theta^2/2}{ t^2 + \imath \theta^2/2}$. The corresponding  forms in $\mF_2'$ is:
\[ \upsilon'_{t^2}= (Y^2-\gamma t^2 X^2)^2 - \gamma^3 X^4 -  4 \gamma^2 t  X^3Y.\]

\end{enumerate}

{\bf \noindent Case of $\mF_2'$ for $q \equiv 3 \mod 4$} \hfill \\
If $q \equiv 3$ mod $4$,  let $\bF_{q^2}=\bF_q[\imath]$ where $\imath^2=-1$ and let $\bF_{q^4} = \bF_{q^2}[\theta]$ where $N_{q^2/q}(\theta^2)=-1$. Replacing $\theta^2=\theta_0 + \imath \theta_1$ with $-\theta^2$ if necessary, we may assume $\theta_0$ is a non-square in $\bF_q$. We note that $\phi(\theta) = \imath/\theta$. 
 The map  $g(r) =  \tfrac{2r +  \theta_0  +  \imath}{ 2r +  \theta_0 -  \imath }$  maps $\bF_q$ bijectively to $\tilde \mN_2$. Therefore, it suffices to exhibit $f \in \mF_2'$ for which $\lambda_f$ is of the form
$\tfrac{2r +  \theta_0 +  \imath}{ 2r +  \theta_0 -  \imath }$ for each $r \in \bF_q$.
We claim $f = \upsilon'_r$ defined below, has this property: 
\beq \label{eq:upsilonr3mod4}  \upsilon'_r=\begin{cases}
 (Y^2- r  X^2)^2 -  X^4  + 2 X^2(  (Y^2+ r  X^2)   \theta_0  + 2XY \sqrt{-r}  \theta_1)   &\!\!\text{if $r \in \{-x^2: x \in \bF_q\}$}, \\
  (Y^2+ r X^2)^2 -X^4 -  2 X^2(  (Y^2- r  X^2)  \theta_0 - 2 XY \sqrt{r}  \theta_1)  &\!\!\!\!\text{if $r \in (\bF_q^\times)^2$.}
\end{cases} \eeq
 
To see this, we consider $(r_1, \phi^2(r_1),\phi(r_1), \phi^3(r_1))$ with
\begin{enumerate}
\item $r_1=\imath(\theta+  t)$ where $t \in \bF_q$.  \\Here  $\phi(r_1) - r_1= (-2 \imath t - \imath \theta +1/\theta)$ and hence $\beta=N_{q^4/q^2}(\phi(r_1)-r_1)$ equals $-4 t^2+   2 \theta_0  +2  \imath$.  Therefore, $\lambda_f=\tfrac{\beta}{\phi(\beta)}$ equals
$\tfrac{-2 t^2 +  \theta_0  +  \imath}{ -2 t^2 +  \theta_0  -  \imath }$. The corresponding form in $\mF_2'$ is:
\[  \upsilon'_{- t^2}= (Y^2+ t^2  X^2)^2 -  X^4  + 2X^2(  (Y^2- t^2  X^2)  \theta_0  + 2 tXY   \theta_1  ). \]

\item $r_1=(-\imath t -\imath/ \theta)$ for $ t \in  \bF_q^{\times}$. Here  $\phi(r_1) - r_1= (2 \imath t + \imath/ \theta + \theta)$ and hence $\beta=N_{q^4/q^2}(\phi(r_1)-r_1)$ equals $-4 t^2- 2 \theta_0 -2  \imath$.  Therefore, $\lambda_f=\tfrac{\beta}{\phi(\beta)}$ equals
$\tfrac{2 t^2 +  \theta_0 +  \imath}{ 2 t^2 +  \theta_0 -  \imath }$. The corresponding form in $\mF_2'$ is:
\[  \upsilon'_{t^2}= (Y^2+ t^2  X^2)^2 -  X^4  - 2 X^2(  (Y^2- t^2  X^2) \theta_0 - 2t XY  \theta_1).   \]

\end{enumerate}
\eep

  We will now determine the $H_i$-orbits  on $\tilde \mN_i$. 
\begin{lem} \label{Ji_def}
The function $\jmath(\lambda)$ induces injective maps $\bar \jmath:  \mN_i/H_i \to \bF_q$. 
 The sets   $J_i = \jmath(\mN_i), i=1,2,4$ partition $\bF_q \setminus \{0, 1728\}$ and have sizes:
\beq  \label{eq:Ji_def} |J_4|=\tfrac{(q-6-\mu)}{6}, \quad |J_2|=\tfrac{q-2+\mu}{2}, \quad  |J_1|=\tfrac{q-\mu}{3},  \eeq
where  $\mu \in \{ \pm 1\}$ with  $q \equiv \mu \mod 3$.

\begin{enumerate}
\item The  map $\bar \jmath:  \mN_i/H_i \to J_i$ gives a bijection between $H_i$-orbits on $\mN_i$ and $J_i$.
\item  The  map $\bar \jmath:  \mN_4/H_2 \to J_4$ is surjective and $3$-to-$1$, and hence the number of $H_4'=H_2$-orbits on $\mN_4$ is $3 |J_4|$.
\item The $H_i, H_i'$ orbits on $\mN_i^{1728}$ are:\begin{enumerate}
		\item $\mN_1^{1728} = \emptyset$.
		\item   $\mN_4^{1728}/H_4$ consists of a single orbit $\{-1, 1/2, 2\}$.
		\item  $\mN_4^{1728}/H_2$ consists of  two orbits $\{-1\}$ and $\{1/2, 2\}$.
		\item  $\mN_2^{1728}/H_2$ consists of a single orbit $\{-1\}$. 
	 \end{enumerate}
\item   Let $q \equiv \mu \mod 3$ with $\mu \in \{\pm 1\}$. 
 The number of orbits in  $\mN_4^0/H_4$, $\mN_4^0/H_2$, $\mN_1^0/H_1$ and   $\mN_2^0/H_2$ is $(1+\mu)/2$,  $(1+\mu)/2$,  $(1+\mu)$ and  $(1-\mu)/2$. More explicitly, if $\mu=1$ then  the $H_i, H_i'$ orbits on $\mN_i^{0}$ are:
		\begin{enumerate}
		\item  $\mN_4^0/H_4$ and  $\mN_4^0/H_2$ consist of a single orbit $\{-\omega, -\omega^2\}$.
		\item  $\mN_1^0/H_1$ consists of two orbits $\{-\omega\}$ and $\{-\omega^2\}$. 
		\item  $\mN_2^0 = \emptyset$.
		\end{enumerate}
	 If $\mu=-1$ then 
		\begin{enumerate}
		\item  $\mN_4^0 = \mN_1^0 = \emptyset$. 
		\item  $\mN_2^0/H_2$  consists of a single orbit $\{-\omega, -\omega^2\}$.
		\end{enumerate}
	 \end{enumerate}
\end{lem}
\begin{proof}
Since $\jmath(\lambda)$ is constant on $G_*$-orbits and $H_i \subset G_*$,  we get induced maps $\bar \jmath:  \mN_i/H_i \to \overline{\bF_q}$. We first show that $J_1, J_2, J_4$ are disjoint subsets of $\bF_q$:
 We note that the sets $\mN_1 \subset \bF_{q^3}\setminus \bF_q$, $\mN_2 \subset \bF_{q^2}\setminus \bF_q$ and $\mN_4 \subset  \bF_q$. Therefore, if  $\lambda_i \in \mN_i$, $\lambda_j \in \mN_j$ with $i \neq j$,     the $G_*$-orbits  $G_*(\lambda_i)$ and $G_*(\lambda_j)$ are distinct. Thus, $J_1, J_2, J_4$ are disjoint sets of $\overline{\bF_q}$. Clearly $\mN_4 \subset \bF_q$ hence $J_4 \subset \bF_q\setminus \{0, 1728\}$.
For $\lambda \in \mN_1$ we can write the expression for $\jmath(\lambda)$ as 
$\jmath(\lambda)=256\,  N_{q^3/q}(\lambda-\phi(\lambda) )
$  which shows that $J_1 \subset \bF_q \setminus \{0,1728\}$. For $\lambda \in \mN_2$ we can write the expression for $\jmath(\lambda)$ as 
\[ \jmath(\lambda)=256\tfrac{(\phi(\lambda) +\lambda-1)^3}{\phi(\lambda)+\lambda -2}
,\]
which shows that $J_2 \subset \bF_q \setminus \{0,1728\}$.\\

Next we show that $\bar \jmath:\mN_i/H_i \to \bF_q$  is injective. If $i=4$ then $H_4=G_*$ and hence $\bar \jmath$ is injective. We now assume $\{i,j\}=\{1,2\}$. 
Suppose  $g \in G_*$ and $\lambda \in \mN_i$ such that $g(\lambda) \in \mN_i$, then we must show 
$g(\lambda)=h(\lambda)$ for some $h \in H_i$.  Since $G_* = H_iH_j$, we may assume $g \in H_j$. If $g$ is the identity element of $H_j$ then we take $h$ to be the identity element of $H_i$.  So we assume $g$ is a nontrivial element of $H_j$. If $i=1$ then $g(\lambda)=1/\lambda$ does not satisfy  $x^{q+1}- x^q+1=0$  unless $q \equiv 1 \mod 3$ and $\lambda \in  \{-\omega, -\omega^2\}$ which are not in $\mN_1$. If $i=2$ then $g(\lambda)$ is  $1/(1-\lambda)$ or $\lambda/(\lambda-1)$ which do not satisfy the equation $N_{q^2/q}(x)=1$ 
 unless $q \equiv 2 \mod 3$ and $\lambda \in  \{-\omega, -\omega^2\}$ which are not in $\mN_2$.\\

Finally, we show that $J_1,J_2,J_4$ partition $\bF_q \setminus \{0,1728\}$. As shown above, the function  $\jmath(\lambda)$ is  injective on $\mN_i/H_i$. Therefore $|J_i|=|\mN_i/H_i|$.  Since $\mN_4$ has size $(q-6-\mu)$ and each $H_4=G_*$ orbit has size $6$ we see that $ |J_4|=\tfrac{(q-6-\mu)}{6}$. Similarly, $|\mN_2| = (q-2+\mu)$,   $|\mN_1| = (q-\mu)$,  each $H_2$ orbit on $\mN_2$ has size $2$,  and each $H_1$-orbit on  $\mN_1$ orbit has size $3$. Therefore, $|J_2|=|\mN_2/H_2| =(q-2+\mu)/2$ and   $|J_1|=|\mN_1/H_1| =(q-\mu)/3$.
 Since $|J_4|+|J_2|+|J_1|=(q-2)$ we see that the sets $J_4, J_2,J_1$ partition  $\bF_q \setminus \{0,1728\}$. This completes the proof of part (1) of the lemma. \\
 Since $\mN_4/H_2 \to \mN_4/H_4$ is a $3$-to-$1$ map, it follows that $\mN_4/H_2 \to J_4$ is a $3$-to-$1$ surjective map. This proves part (2).\\
As for parts $(3)-(4)$, the $H_i, H_i'$-orbits of $\mN_i^{1728}$ and $\mN_i^0$ readily follow from the fact that  $\mN_i^{1728}$ equals i) $\{-1,1/2,2\}$ if $i=4$, ii) $\{-1\}$ if $i=2$, and $\emptyset$ if $i=1$. Similarly,   $\mN_4^0, \mN_2^0, \mN_1^0$ equal $ \{-\omega, -\omega^2\}$, $\emptyset$,  and $\{-\omega, -\omega^2\}$ if $q \equiv 1 \mod 3$ and they equal $\emptyset$, $\{-\omega, -\omega^2\}$ , $\emptyset$ if $q \equiv 2 \mod 3$.
\end{proof}

\begin{lem}\label{Ji_lem} The  $(q-2)=|J_4|+|J_2| +|J_1|$ orbits in $\mF_4 \cup \mF_2 \cup \mF_1$ are represented by $G \cdot \mathcal E_r$ for $r \in J_4 \cup J_2 \cup J_1=\bF_q\setminus\{0,1728\}$, where
\beq \label{eq:E_r} \mathcal E_r(X,Y)= X(\tfrac{256}{27}(1-\tfrac{1728}{r})Y^3-4 Y X^2 + X^3), \qquad r \in \bF_q\setminus\{0,1728\}.  \eeq
\end{lem}
\begin{proof}
We recall from \eqref{eq:resolvent} that, for a  quartic form $f(X,Y)$ over $\bF_q$ and $K \supset \bF_q$ any field over which $f$ has a linear factor, there is a $g \in PGL_2(K)$ such that $f = g \cdot  X R_f$ where  $R_f(X,Y) =-4 Y^3 +3 I(f) YX^2 -J(f) X^3$. For $f \in \mF_4 \cup \mF_2 \cup \mF_1$,   $f$ has a linear factor over $\bF_q$ 
and hence we can take $K=\bF_q$. If $\jmath(f) \neq 0, 1728$, equivalently if $I(f)$ and  $J(f)$ are  non-zero, then  replacing $R_f$ with $h \cdot R_f$ where $h(t)=\tfrac{3I(f)}{4J(f)} t$ we get 
\[ G \cdot f = G \cdot X \left(\tfrac{256 J(f)^2}{27 I(f)^3}Y^3-4 Y X^2 + X^3\right). \]
Further assuming $\Delta(f)  = I^3(f)-J^2(f) \neq 0$ we can rewrite the above equation as 
\[ G \cdot f = G \cdot X \left(\tfrac{256}{27}(1-\tfrac{1728}{\jmath(f)})Y^3-4 Y X^2 + X^3\right), \qquad \jmath(f) \neq 0, 1728, \infty. \]
In terms of $Z=X/Y$, it is easy to see that the number of values of $c \in \bF_q \setminus\{0, -\tfrac{256}{27}\}$  for which the equation  $Z^3 -4  Z^2=c$  has $(i-1)$  roots in $\bF_q$ where $i=4,2$ and $1$,  equals   $(q-6-\mu)/6, (q-2+\mu)/2$ and $(q-\mu)/3$, respectively. In other words, of these $(q-2)$ orbits, there are $|J_i|$ orbits in $\mF_i$ for $i=1,2,4$.
\end{proof}

\begin{prop} \label{stabilizer_prop} Given $f$ with $\Delta(f) \neq 0$, let $(r_1,r_2,r_3,r_4)$ be a fixed restricted ordering of the roots of $f$, and let 
 $\lambda_f$ denote the cross-ratio of this ordering. Let $F$ be a splitting field of $f$ over $\bF_q$. We recall from  Lemma \ref{stabilizer_lemma} the injective homomorphism $\sigma:\text{Stab}_F(f) \to  S_4$ and the transposition  $\tilde \sigma_3 \in S_4$ 
 given by  $(12), (13), (14)$ according as $\lambda_f=-1,2,1/2$, respectively. Let $\sigma_{\phi} \in S_4$ be as in Lemma \ref{ZS4}. We identify Stab$_F(f)$  with its image in $S_4$ under the monomorphism $\sigma:\text{Stab}_F(f) \hookrightarrow S_4$.  The  stabilizer group Stab$(f)$ equals the centralizer of $\sigma_\phi$ in $\mathrm{Stab}_F(f)$
\[ \mathrm{Stab}(f)=Z_{\mathrm{Stab}_F(f)}(\sigma_\phi), \] 
and is isomorphic to:\\

In case $\jmath(f) \neq 0, 1728$:
\begin{enumerate}
\item[(1)]$ \langle \sigma_1, \sigma_2 \rangle \simeq \bZ/2 \bZ \times \bZ/2 \bZ$ if  $f \in \mF_4$,
\item[(2)]$ \langle \sigma_1, \sigma_2 \rangle \simeq \bZ/2 \bZ \times \bZ/2 \bZ$ if  $f \in   \mF_4'$,
\item[(3)]  $\langle \sigma_1\rangle   \simeq \bZ/2 \bZ$ if  $f \in \mF_2$,
\item[(4)]  $\langle \sigma_1\rangle   \simeq \bZ/2 \bZ$ if  $f \in  \mF_2'$,
\item[(5)] \emph{the trivial group}  if  $f \in \mF_1$.\\
\end{enumerate}
In case  $\jmath(f) =1728$:
\begin{enumerate}
\item[(i)] $ \langle \sigma_1, \sigma_2, \tilde\sigma_3  \rangle \simeq D_4$ if $f \in \mF_4$,
\item[(ii)] $ \langle \sigma_1, \sigma_2, \tilde\sigma_3  \rangle \simeq D_4$ if $f \in \mF_4'$ and $\lambda_f=-1$,
\item[(iii)]$ \langle \sigma_1, \sigma_2 \rangle \simeq \bZ/2 \bZ \times \bZ/2 \bZ$ if 
$f \in \mF_4'$ and $\lambda_f \in \{2,1/2\}$,
\item[(iv)]   $\langle \sigma_1, \sigma_3 \rangle   \simeq \bZ/2 \bZ \times \bZ/2 \bZ$ if  $f \in \mF_2$,
\item[(v)]  $\langle \sigma_3  \sigma_2 \rangle \simeq \bZ/4 \bZ$ if  $f \in \mF_2'$.\\
\end{enumerate}
In case  $\jmath(f) =0$: 
\begin{enumerate}
\item[(a)] $\langle \sigma_1, \sigma_2, \sigma_4 \rangle   \simeq A_4$ if 
$f \in \mF_4$,
\item[(b)] $\langle \sigma_1, \sigma_2 \rangle \simeq \bZ/2 \bZ \times \bZ/2 \bZ$ if  $f \in \mF_4'$,
\item[(c)] $\langle \sigma_1 \rangle   \simeq \bZ/2 \bZ$  if  $f \in \mF_2$,
\item[(d)] $\langle \sigma_1 \rangle   \simeq \bZ/2 \bZ$  if  $f \in \mF_2'$,
\item[(e)] $\langle \sigma_4 \rangle   \simeq \bZ/3 \bZ$  if  $f \in \mF_1$.
\end{enumerate}
\end{prop}

\begin{proof}
For $g \in \text{Stab}_F(f)$ to be in $PGL_2(q)$ we must have $\phi(g)=g$ which is equivalent to $\sigma_{\phi(g)}=\sigma_g$. Since $\sigma_{\phi(g)} = \sigma_\phi \sigma_g (\sigma_\phi)^{-1}$ we get
\[ \text{Stab}(f) = \{ g \in \text{Stab}_F(f) : \sigma_\phi \sigma_g (\sigma_\phi)^{-1}= \sigma_g\} =Z_{\text{Stab}_F(f)}(\sigma_\phi). \]
If $f \in \mF_4$ then $F = \bF_q$ and hence $\text{Stab}(f) = \text{Stab}_F(f)$ is as given in  Lemma \ref{stabilizer_lemma}:
\begin{description}
\item[--]  $ \langle \sigma_1, \sigma_2 \rangle$ if $\jmath(f) \neq 0,1728$,
\item[--]  $ \langle \sigma_1, \sigma_2, \tilde\sigma_3  \rangle$  if $\jmath(f)=1728$,
\item[--]  $ \langle \sigma_1, \sigma_2, \sigma_4  \rangle$ if  $\jmath(f)=0$.
\end{description}

This proves the  assertions about $\mF_4$ in $1), i), a)$. Using the above calculation of $ \text{Stab}_F(f)$  and the calculation of  $Z_{S_4}(\sigma_\phi)$ in \eqref{eq:ZS4}, we can immediately determine  Stab$(f) = \text{Stab}_F(f) \cap Z_{S_4}(\sigma_\phi)$.  If $f \in \mF_4'$ then Stab$(f)=\text{Stab}_F(f) \cap \langle \sigma_1, \sigma_2,  \sigma_3 \rangle$. Thus, Stab$(f)$ equals
\begin{description}
\item[--]  $ \langle \sigma_1, \sigma_2, \sigma_3  \rangle$  if $\lambda_f=-1$,
\item[--]  $ \langle \sigma_1, \sigma_2 \rangle$ if $\lambda_f \in \{2,1/2\}$,
\item[--]  $ \langle \sigma_1, \sigma_2 \rangle$ if  $\jmath(f)=0$,
\item[--]  $ \langle \sigma_1, \sigma_2 \rangle$ if $\jmath(f) \neq 0,1728$.
\end{description}
This proves assertions 2), ii)-iii), b). If $f\in \mF_2$ then Stab$(f)=\text{Stab}_F(f) \cap \langle \sigma_1,  \sigma_3 \rangle$. Thus, Stab$(f)$ equals
\begin{description}
\item[--]  $ \langle \sigma_1 \rangle$ if $\jmath(f)=0$, 
\item[--]  $ \langle \sigma_1, \sigma_3 \rangle$ if  $\jmath(f)=1728$,
\item[--]  $ \langle \sigma_1 \rangle$ if $ \jmath(f) \neq 0,1728$.
\end{description}
This proves assertions 3), iv), c). If $f\in \mF_2'$ then Stab$(f)=\text{Stab}_F(f) \cap 
\langle \sigma_3 \sigma_2 \rangle$ (where  $\sigma_3 \sigma_2= \langle (1324)  \rangle$).  Thus, Stab$(f)$ equals
\begin{description}
\item[--]  $ \langle \sigma_1 \rangle$ if $\jmath(f)=0$, 
\item[--]  $\langle \sigma_3 \sigma_2 \rangle$  if  $\jmath(f)=1728$,
\item[--]  $ \langle \sigma_1 \rangle$ if $ \jmath(f) \neq 0,1728$.
\end{description}
This proves assertions 4), v), d). If $f\in \mF_1$ then Stab$(f)=\text{Stab}_F(f) \cap 
\langle \sigma_4  \rangle$ (where  $\sigma_4 = \langle (234)  \rangle$).  Thus Stab$(f)$ equals
\begin{description}
\item[--]  $ \langle \sigma_4 \rangle$ if $\jmath(f)=0$, 
\item[--]  trivial  if $ \jmath(f) \neq 0,1728$.
\end{description}
This proves assertions 5), e).
 \end{proof}

\subsection{ Proof of Theorem \ref{result2} }
The forms $f$ with $\Delta(f) \neq 0$ have been partitioned  into the sets  $\mF_1 \cup \mF_2 \cup \mF_4 \cup \mF_2' \cup \mF_4'$.  The size of each orbit $G \cdot f$ is $|G|/|\text{Stab}(f)|$ where  $|\text{Stab}(f)|$  is given in Proposition \ref{stabilizer_prop}. The number of orbits in each $\mF_i$, $\mF_i'$ can be determined from Proposition \ref{Lambda_prop}: The $G$-orbits in $\mF_i$ (resp. $\mF_i'$) are in bijective correspondence with the $H_i$-orbits (resp. $H_i'$-orbits) in $\tilde \mN_i$.
The $H_i$-orbits (resp. $H_i'$-orbits) in $\tilde \mN_i$ have in turn been calculated in  Lemma \ref{Ji_def}. Putting these together, we get:
\begin{enumerate}
\item The $G$-orbits in $\mF_4$ are as follows:\\

\begin{tabular}{c| *{6}{c}}
$\jmath$   & $|G|$ & $|G|/2$ & $|G|/3$ & $|G|/4$ & $|G|/12$   &  $|G|/8$  \\
&&&&&&\\
$\jmath \in J_4$   & $0$ & $0$ & $0$  & $1$ &  $0$ &   $0$  \\ 
$\jmath=1728$ & $0$ & $0$ & $0$ & $0$ & $0$ &  $1$  \\
$\jmath=0$  & $0$ & $0$ & $0$ & $0$ &  $\tfrac{1+\mu}{2}$  &  $0$ 
\end{tabular}
\[ \]
\item  The $G$-orbits in $\mF_4'$ are as follows:\\

\begin{tabular}{c| *{6}{c}}
$\jmath$   & $|G|$ & $|G|/2$ & $|G|/3$ & $|G|/4$ & $|G|/12$   &  $|G|/8$  \\
&&&&&&\\ 
$\jmath \in J_4$   & $0$ & $0$ & $0$  & $3$ &  $0$ &   $0$  \\ 
$\jmath=1728$ & $0$ & $0$ & $0$ & $1$ & $0$ &  $1$  \\
$\jmath=0$  & $0$ & $0$ & $0$ & $\tfrac{1+\mu}{2}$ &  $0$  &  $0$  \\ [1ex]
\end{tabular}
\[ \]
\item  The $G$-orbits in $\mF_2$ are as follows:\\

\begin{tabular}{c| *{6}{c}}
$\jmath$   & $|G|$ & $|G|/2$ & $|G|/3$ & $|G|/4$ & $|G|/12$  &  $|G|/8$  \\
&&&&&&\\
$\jmath \in J_2$   & $0$ & $1$ & $0$  & $0$ &  $0$ &   $0$ \\ 
$\jmath=1728$ & $0$ & $0$ & $0$ & $1$ & $0$ &  $0$  \\
$\jmath=0$  & $0$ & $\tfrac{1-\mu}{2}$ & $0$ & $0$ &  $0$  &  $0$ 
\end{tabular}
\[ \]
\item  The $G$-orbits in $\mF_2'$ are as follows:\\

\begin{tabular}{c| *{6}{c}} \label{tableF4}
$\jmath$   & $|G|$ & $|G|/2$ & $|G|/3$ & $|G|/4$ & $|G|/12$  &  $|G|/8$  \\
&&&&&&\\ 
$\jmath \in J_2$   & $0$ & $1$ & $0$  & $0$ &  $0$ &   $0$ \\ 
$\jmath=1728$ & $0$ & $0$ & $0$ & $1$ & $0$ &  $0$  \\
$\jmath=0$  & $0$ & $\tfrac{1-\mu}{2}$ & $0$ & $0$ &  $0$  &  $0$ 
\end{tabular}
\[ \]
\item  The $G$-orbits in $\mF_1$ are as follows:\\

\begin{tabular}{c| *{6}{c}}
$\jmath$   & $|G|$ & $|G|/2$ & $|G|/3$ & $|G|/4$ & $|G|/12$  &  $|G|/8$  \\
&&&&&&\\
$\jmath \in J_1$ & $1$  & $0$ & $0$  & $0$ &  $0$ &   $0$ \\ 
$\jmath=1728$ & $0$ & $0$ & $0$ & $0$ & $0$ &  $0$  \\
$\jmath=0$  & $0$ & $0$& $1+\mu$  & $0$ &  $0$  &  $0$ 
\end{tabular} 
\[\]
\end{enumerate}
Combining the tables of orbits for $\mF_{4}, \mF_4', \mF_{2}, \mF_2'$ and $\mF_{1}$ above, we get the table below of orbits for all $f$ with $\Delta(f) \neq 0$, which is exactly Table \ref{table:1}.\\

\begin{tabular}{c| *{6}{c}}
   & $|G|$ & $|G|/2$ & $|G|/3$ & $|G|/4$ & $|G|/12$   &  $|G|/8$  \\
&&&&&&\\  \hline 
&&&&&&\\
$\jmath(f)\in J_4$ & $0$ & $0$ & $0$ & $4$ & $0$ &  $0$  \\
$\jmath(f)\in J_2$ & $0$ & $2$ & $0$ & $0$ & $0$ &  $0$  \\
$\jmath(f)\in J_1$ & $1$ & $0$ & $0$ & $0$ & $0$ &  $0$  \\
$\jmath(f)=1728$ & $0$ & $0$ & $0$ & $3$ & $0$ &  $2$  \\
$\jmath(f)=0$  & $0$ & $(1-\mu)$ & $(1+\mu)$ & $\tfrac{1+\mu}{2}$ &  $\tfrac{1+\mu}{2}$  &  $0$\\
&&&&&&\\
\hline
\end{tabular}
\[\]

In Table \ref{table:3} we list  representatives for each of the $(2q+2+\mu)$ orbits of binary quartic forms with non-zero discriminant, and  the isomorphism class of the stabilizer of each orbit (obtained in  Proposition \ref{stabilizer_prop}).  We determine the orbit representatives by refining the calculations in Proposition \ref{Lambda_prop} and Lemma \ref{Ji_lem}.

\begin{table}[b]
\begin{tabular}{||c|c|c|c||}
\hline
Representative  of $\mO$ & $\jmath(\mO)$ & Type & Stabilizer    \\
&&&\\  
\hline  
&&&\\  
$\mathcal E_r, \; r \in J_4$ & $r \in J_4$ & $\mF_4$&  $\bZ/2\bZ \times \bZ/2\bZ$ \\
&&&\\

$\mathcal E_r, \;  r \in J_2$& $r \in J_2$ & $\mF_2$&  $\bZ/2\bZ$ \\
&&&\\  
$\mathcal E_r, \; r \in J_1$& $r \in J_1$ &$ \mF_1$&  trivial \\
&&&\\  
\hline 
&&&\\  
$XY(Y-X)(Y+X) $ & $1728$ & $\mF_4$ & $D_4$ \\
$X( Y^3  -X^3), \; q \equiv 1 \mod 3$ & $0$ & $\mF_4$ & $A_4$\\
$XY(Y^2-\ep X^2)$ & $1728$ & $\mF_2$ & $\bZ/2\bZ \times \bZ/2\bZ$ \\
$X( Y^3  -X^3), \; q \equiv -1 \mod 3 $ & $0$ & $\mF_2$ & $\bZ/2\bZ$\\
$X( Y^3  -\gamma X^3), \; q \equiv 1 \mod 3 $ & $0$ & $\mF_1$ & $\bZ/3\bZ$\\
$X( Y^3  -\gamma^2 X^3), \; q \equiv 1 \mod 3$ & $0$ & $\mF_1$ & $\bZ/3\bZ$\\
&&&\\  
\hline 
&&&\\  
$ \psi'_{\alpha_{-1}}$  &  $1728$ & $\mF_4'$ & $D_4$  \\
$ \psi'_{\alpha_{\{2, 1/2\}}} $  &  $1728$ & $\mF_4'$ & $\bZ/2\bZ \times \bZ/2\bZ$  \\
$ \psi'_{\alpha_{\{-\omega, -\omega^2\}}}, \; q \equiv 1 \mod 3$  &  $0$ & $\mF_4'$ & $\bZ/2\bZ \times \bZ/2\bZ$  \\
&&&\\  
$ \psi'_{\alpha_{\{\lambda, \lambda^{-1}\}}}, \; \{\lambda, \lambda^{-1}\} \in \mN_4/H_2$   &  $\jmath(\lambda) \in J_4$ & $\mF_4'$ & $\bZ/2\bZ \times \bZ/2\bZ$ \\  
these are $3|J_4|$ orbits&&&\\  
\hline
&&&\\ 
$\upsilon'_{t_{-1}}, \quad q \equiv 1 \mod 4$ & $1728$ &  $\mF_2'$ &  $\bZ/4\bZ$\\
&&&\\  
$\upsilon'_{t_{\{-\omega, -\omega^2\}}},  \quad q \equiv 5 \mod 12$  & $0$& $\mF_2'$ &$\bZ/2\bZ$\\
&&&\\  
$\upsilon'_{t_{\{\lambda, \lambda^{-1}\}}}, \quad \{\lambda, \lambda^{-1}\} \in \mN_2/H_2$
& $ \jmath(\lambda) \in J_2$& $\mF_2'$ &$\bZ/2\bZ$\\
$q \equiv 1 \mod 4$ these are $|J_2|$ orbits 
&&&\\
\hline  
&&&\\
$\upsilon'_{t_{-1}},  \quad  q \equiv 3 \mod 4$ & $1728$ &  $\mF_2'$ &  $\bZ/4\bZ$\\
&&&\\  
$\upsilon'_{t_{\{-\omega, -\omega^2\}}}, \quad q \equiv 11 \mod 12$  & $0$& $\mF_2'$ &$\bZ/2\bZ$\\
&&&\\  
$\upsilon'_{t_{\{\lambda, \lambda^{-1}\}}}, \quad \{\lambda, \lambda^{-1}\} \in \mN_2/H_2$  & $ \jmath(\lambda)\in J_2$& $\mF_2'$ &$\bZ/2\bZ$\\
$q \equiv 3 \mod 4$ these are $|J_2|$ orbits &&&\\
\hline  
\end{tabular}
\[\]
\caption{Representatives and stabilizers of $G$-orbits of quartic forms with nonzero discriminant: }
 \label{table:3}

\end{table}

 We begin with the orbits in $\mF_4 \cup \mF_2 \cup \mF_1$. As shown in Lemma \ref{Ji_lem}, each such orbit $G \cdot f$ is represented by 
\[ X R_f(X,Y)=X(-4 Y^3 +3 I(f) YX^2 -J(f) X^3).\]
 
The  $(q-2)=|J_4|+|J_2| +|J_1|$ orbits  in $\mF_4 \cup \mF_2 \cup \mF_1$  with $\jmath(\mO) \notin \{0, 1728\}$ are represented by $G \cdot \mathcal E_r$, where 
\[  \mathcal E_r(X,Y)= X(\tfrac{256}{27}(1-\tfrac{1728}{r})Y^3-4 Y X^2 + X^3), \qquad r \in \bF_q\setminus\{0,1728\}.  \]
If $\jmath(f)=1728$, i.e., $J(f)=0$ then $XR_f(X,Y)=XY(-4 Y^2+ 3 I(f) X^2)$. We get two orbits corresponding to the cases $3I(f)/4$ being a quadratic residue or non-residue in $\bF_q$, and  these orbits are in $\mF_4$ and $\mF_2$, respectively. 

If $\jmath(f)=0$, i.e. $I(f)=0$ then $XR_f(X,Y)=-X(4 Y^3 +  J(f) X^3)$. If $q \equiv 2 \mod 3$ we get a single orbit $G \cdot X(Y^3-X^3)$ in $\mF_2$. If $q \equiv 1 \mod 3$ we get three  orbits $G \cdot X(Y^3- X^3)$ in $\mF_4$, and $G \cdot X(Y^3- \gamma^i X^3), i=1,2$ in $\mF_1$.\\

Turning to the orbits in $\mF_4'$,  we recall from Proposition \ref{Lambda_prop},  that there are $(q-1)/2$ orbits in $\mF_4'$ which correspond to the $(q-1)/2$ orbits of the group $H_2=\{t \mapsto t^{\pm 1} \}$ on $\tilde \mN_4=\bF_q\setminus\{0,1\}$. These orbits are  given by $\{-1\}$ and the $(q-3)/2$ classes $\{\lambda, \lambda^{-1}\}$ for $\lambda \in \bF_q\setminus\{0, \pm 1\}$. Let $q \equiv \mu \mod 3$ with $\mu \in \{\pm 1\}$. If $\mu=1$, one of these classes is $\{-\omega, -\omega^2\}$, and there are $(q-4-\mu)/2$ classes in $\mN_4/H_2$ where $\mN_4 = \tilde \mN_4 \setminus \{-\omega, -\omega^2, -1,2, \tfrac{1}{2}\}$. The map 
\[\chi_1:\bF_{q^2}\setminus\{\bF_q \cup \{\pm \sqrt \ep\}\} \to \tilde \mN_4 , \qquad 
\chi_1(\alpha)=  N_{q^2/q}(\tfrac{\alpha - \sqrt\ep}{\alpha+\sqrt\ep}) \]
is surjective. We pick some
$\alpha_{-1}$  and $\alpha_{\{\lambda, \lambda^{-1}\}}$ in $\bF_{q^2}$ 
 satisfying $\chi_1(\alpha_{-1})=-1$ and   $\chi_1(\alpha_{\{\lambda, \lambda^{-1}\}}) \in \{\lambda, \lambda^{-1}\}$.  We will consider the forms  $\psi'_{\alpha}$ for these $(q-1)/2$ values of $\alpha$, where  $\psi'_{\alpha}$ is the form given in \eqref{eq:psialpha}:
\[ \psi'_\alpha = (Y^2- \ep X^2) ( Y - \alpha X)(Y - \phi(\alpha)X) , \quad  \alpha \in \bF_{q^2}\setminus\{\bF_q \cup  \{\pm \sqrt\ep\}\}. \]

Next, we turn to the orbits in $\mF_2'$. We recall from Proposition \ref{Lambda_prop}, that  there are $(q-1)/2$ orbits in $\mF_2'$ which correspond to the $(q-1)/2$ orbits of the group $H_2=\{t \mapsto t^{\pm1} \} \subset G_*$  acting on $\tilde \mN_2=\{ \lambda \in \bF_{q^2}\setminus\{1\}: N_{q^2/q}(\lambda)=1\}$. These orbits are  given by $\{-1\}$ and the $(q-3)/2$ classes $\{\lambda, \lambda^{-1}\}$ for $\lambda \in \tilde \mN_2\setminus\{ -1\}$. 

If $q \equiv 1 \mod 4$, we recall the discussion around \eqref{eq:upsilonr1mod4}: $\bF_{q^4}=\bF_q[\theta]$ where $\theta^4=\gamma$, and  $\gamma$ denotes a generator of the cyclic group $\bF_q^\times$. The map 
\[\chi_2: \bF_q   \to \tilde \mN_2 , \qquad \chi_2(t)=\tfrac{t -  \imath \theta^2/2}{t + \imath \theta^2/2} \]
is bijective, and moreover maps the classes  $\{\pm t\}$ bijectively to the classes $\{\lambda, \lambda^{-1}\}$.  The class $t=0$ maps to the class $\lambda = -1$.   We denote the class $t=0$ by $t_{-1}$. We pick some  $t_{\{\lambda, \lambda^{-1}\}} \in \bF_q$ satisfying the property that 
$\chi_2(t_{\{\lambda, \lambda^{-1}\}}) \in \{\lambda, \lambda^{-1}\}$. For example,  
\[ t_{\{\lambda, \lambda^{-1}\}} =\tfrac{\imath \theta^2 (1+\lambda)}{2(1-\lambda)}. \]
 We will consider the forms  
$\upsilon'_t$   for these $(q-1)/2$ values of $t$, where  $\upsilon'_t$ 
is the form given in \eqref{eq:upsilonr1mod4}:
\[ \upsilon'_t= \begin{cases}
 (Y^2- t  X^2)^2 - \gamma X^4 -  4 \sqrt{\gamma t}  X^3Y &\text{if $\ep t$ is a square in $\bF_q$,} \\
  (Y^2-\gamma t X^2)^2 - \gamma^3 X^4 -  4 \gamma^2 \sqrt{t}   X^3Y &\text{if $t \in (\bF_q^\times)^2$.}
\end{cases} \]

 In case  $q \equiv 3 \mod 4$,  we recall the discussion around \eqref{eq:upsilonr3mod4}: 
 $\bF_{q^2}=\bF_q[\imath]$ where $\imath = \sqrt{-1}$, and $\bF_{q^4} = \bF_{q^2}[\theta]$ where $\theta^2 = \theta_0  + \imath \theta_1$ with $\theta_0^2+\theta_1^2=-1$ and 
 $\theta_0$ is a non-square in $\bF_q$.   The map 
\[ \chi_3: \bF_q   \to \tilde \mN_2 , \qquad \chi_3(t)=\tfrac{2r +  \theta_0  +  \imath}{ 2r +  \theta_0 -  \imath }\]
is bijective, and moreover maps the classes  $\{t, -\theta_0-t\}$ bijectively to the classes $\{\lambda, \lambda^{-1}\}$.  The class $t=-\theta_0/2$ maps to the class $\lambda = -1$. 
We again pick  $t_{\{\lambda, \lambda^{-1}\}} \in \bF_q$ which satisfies the property that 
$\chi_3(t_{\{\lambda, \lambda^{-1}\}}) \in \{\lambda, \lambda^{-1}\}$, for example,
\[ t_{\{\lambda, \lambda^{-1}\}} =-\tfrac{\theta_0}{2}+\tfrac{\imath  (\lambda+1)}{2(\lambda-1)}. \]  
We denote the class $t=-\theta_0/2$ by $t_{-1}$.  We will consider the forms  
$\upsilon'_t$   for these $(q-1)/2$ values of $t$, where  $\upsilon'_t$ 
is the form given in \eqref{eq:upsilonr3mod4}:
\[\upsilon'_t= \begin{cases}
 (Y^2- t  X^2)^2 -  X^4  + 2 X^2(  (Y^2+ t  X^2)   \theta_0  + 2XY \sqrt{-t}  \theta_1)   &\!\!\text{if $t \in \{-x^2: x \in \bF_q\}$}, \\
  (Y^2+ t X^2)^2 -X^4 -  2 X^2(  (Y^2- t  X^2)  \theta_0 - 2 XY \sqrt{t}  \theta_1)  &\!\!\!\!\text{if $t \in (\bF_q^\times)^2$.}
\end{cases} \]
\begin{rem} \label{NT_remark}
     The problem of classifying $GL_2$-orbits of binary quartic forms has been studied in number theory literature. In the  work \cite{BhargavaShankar}, M. Bhargava and A. Shankar obtained asymptotic results on the number of $GL_2(\bZ)$-equivalence classes of irreducible binary quartic forms $f(X,Y)$ over $\bZ$, having bounded height $H(f)=\text{max}\{I(f)^3, J(f)^2\}$.  Using a connection between $2$-Selmer groups of elliptic curves and binary quartic forms, the authors proved a very important result: when ordered by height, the average rank of elliptic curves over rationals is bounded. \\

    In the case when $q$ is a prime number, D. Kamenetsky in his master's thesis \cite[Chapter 3]{Kamenetsky}  also provides a description of the stabilizer groups of binary quartic forms $f$ with nonzero discriminant, which agrees with Proposition  \ref{stabilizer_prop} above, and determines the number of orbits and sizes of the orbits (the group acting on nonzero  binary quartic forms is $GL_1(\bF_p) \times GL_2(\bF_p)$ which essentially corresponds to $PGL_2(\bF_p)$ acting on the projective space of quartic forms). The author also explores the application of  these orbits to  Fourier transforms of certain functions on the space of binary quartic forms.\\

     In the work \cite{ishitsuka} by Y. Ishitsuka, T. Taniguchi, F. Thorne, S. Y. Xiao, the authors determined the Fourier transform of the characteristic function of singular (discriminant zero) binary quartic forms over $\mathbb{F}_p$. Using this result, the authors  obtain further results on $2$-Selmer groups of Elliptic curves, and on the asymptotic distribution of the  number of $GL_2(\bZ)$-equivalence classes of irreducible binary quartic forms $f(X,Y)$ over $\bZ$, having bounded height and whose discriminant is square-free and has at most $4$ prime factors.

\end{rem}

\section{Classification of lines of $PG(3,q)$} \label{orbitsP3}
Let $\mF^+$ denote the subset of forms $f$ in $\bP( \text{Sym}^4 (V^*))$ for which $I(f)$ is a square in $\bF_q$. We recall from Lemma \ref{lem_pi} that $\pi:\mQ \to \mF^+$ is a $G$-equivariant  $2$-sheeted covering branched over the forms with $I(f)=0$. For $f \in \mF^+$ the set $\pi^{-1}(f)$ is of the form  $\{L , L^{\perp}\}$ where $L^{\perp}$ is the polar dual of $L$.   If $\mO = G \cdot f$ with $I(f)=0$, then $\pi^{-1}(\mO)$ forms a self-dual orbit $\fO$ of size $|\mO|$. If  $\mO = G \cdot f$ with $I(f) \neq 0$, then $\pi^{-1}(\mO)$ either forms a single self-dual orbit $\fO$ of size $2 |\mO|$, or a pair of orbits $\fO, \fO^{\perp}$ each of size $|\mO|$. 
 Let $\mQ^0$  denote the subset of $\mQ$ consisting of the non-generic lines of $PG(3,q)$.  We begin with the classification  of $G$-orbits of $\mQ^0$. It is shown in \cite {BH} and \cite[p.236]{Hirschfeld3}, that $\mQ^0$ can be decomposed into $8$ classes (in case char$(\bF_q) \neq 2, 3$)  $\fO_1, \dots,\fO_5, \fO_1^\perp, \fO_3^\perp, \fO_5^\perp$. The decomposition of these classes into $G$-orbits is well-known in literature (\cite[Table 1]{GL}, \cite[Theorem 3.1]{DMP1}, \cite[Theorem 8.1]{BPS}): the six classes other that $\fO_5, \fO_5^\perp$ consist of a single orbit, whereas $\fO_5$  and $\fO_5^\perp$ each consist of $2$ orbits $\fO_5 = \fO_{51} \cup \fO_{52}$ and $\fO_5^\perp = \fO_{51}^\perp \cup \fO_{52}^\perp$.
 
 We recall that a line $L$ is non-generic if it is either contained in some osculating plane of the twisted cubic $C$ or intersects $C$. If $P=\nu_3(s,t)$ is a point of $C$, then the set of lines through $P$  form a  `Latin plane' $\Pi$ in the Klein quadric  $\mQ$ (\cite[Table 15.10]{Hirschfeld3}). The  polar dual of these lines is the set of  lines contained in the osculating plane at $P$ represented in $\mQ$ by  a `Greek plane'  $\Pi^{\perp}$. The intersection of $\Pi$ and $\Pi^{\perp}$ represents the lines through $P$ and contained in the osculating plane at $P$. Under the projection $\pi: \mQ \to \bP( \text{Sym}^4(V^*))$ we have $\pi(\Pi)=\pi(\Pi^\perp)$  is the set of forms $f(X,Y)$ having $(Xt-Ys)$ as a repeated root. For example, if $P = \nu_3(0,1)$, then the Latin plane $\Pi$ is given by $z_0=z_1=z_2-z_5=0$ whereas  the Greek plane $\Pi^{\perp}$ is given by $z_0=z_1=z_2+z_5=0$. The plane $\pi(\Pi)=\pi(\Pi^{\perp})$ consists of forms having $X^2$ as a factor. The classification of forms $\mF^0$ having discriminant  $\Delta(f)=0$ is given in Lemma \ref{lemDelta=0}. Since $\Delta(f)=I^3(f)-J^2(f)=0$ we see that $I(f)$ is always a square in $\bF_q$. It is equal to zero for the first two orbits  below, and non-zero for the remaining $4$ orbit pairs. Here, an axis is the line of intersection of two osculating planes of $C$,  a unisecant of $C$ is a non-tangent line of $C$ which meets $C$ in exactly one point, and an external line of $C$ is a line which is not an axis of $C$, and which does not meet $C(\overline{\bF_q})$.
  (see \cite{BH}, \cite[Ch.21]{Hirschfeld3} for more details).
\begin{lem} The set of special lines decompose into the following $10$ orbits:
\begin{enumerate}
\item The orbit $\fO_2 = \pi^{-1}(G \cdot X^4)$ consisting of the  $(q+1)$ tangent lines of $C$. 
\item The orbit $\fO_4 = \pi^{-1}(G \cdot X^3Y)$ consisting of the
$q(q + 1)$  non-tangent unisecants in osculating planes.
\item Orbits $\fO_1$ and $\fO_1^{\perp}$ from $\pi^{-1}(G \cdot  X^2Y^2)$ of size $(q^2+q)/2$ each  and consisting of the secant lines, and the real axes of $C$, respectively.
 
\item Orbits  $\fO_3$ and $\fO_3^{\perp}$ from $\pi^{-1}(G \cdot (X^2-\epsilon Y^2)^2$  of size $(q^2-q)/2$ each  and consisting of the imaginary secant lines, and the imaginary axes of $C$, respectively.

\item Orbits  $\fO_{51}$ and $\fO_{51}^{\perp}$ from $\pi^{-1}(G \cdot X^2Y(Y-X))$  of size $(q^3-q)/2$ each. 

\item Orbits  $\fO_{52}$ and $\fO_{52}^{\perp}$ from $\pi^{-1}(G \cdot X^2(X^2-\epsilon Y^2))$ of size $(q^3-q)/2$ each.  
\end{enumerate}
The set of external lines in osculating planes consists of two orbits $\fO_5^{\perp}= \fO_{51}^{\perp} \cup \fO_{52}^{\perp}$. The set of unisecants not  in osculating planes consists of two orbits 
$\fO_5 = \fO_{51} \cup \fO_{52}$.

\end{lem}

We now begin the proof of Theorem \ref{main}. For an orbit $\fO$ on $\mQ\setminus\mQ^0$ let $\jmath(\fO)=\jmath(\mO)$ where $\mO=\pi(\fO)$, and where $\jmath(\mO)=\jmath(f)$ for any quartic form $f$ representing $\mO$. 
\begin{prop} \label{J+prop} \hfill  \begin{enumerate}
\item All the $(3+\mu)$ orbits having $\jmath(\mO)=0$ are in $\mF^{+}$.
\item  Let  $\jmath(G \cdot f) = 1728$. If  $f \in \mF_4 \cup \mF_4'$ then $G \cdot f \in \mF^+$ if $q \equiv \pm 1 \mod 12$.  If $f \in \mF_2 \cup \mF_2'$ then $G \cdot f \in \mF^+$ if $q \equiv \pm 5 \mod 12$.
\item  if $\jmath(G \cdot f) \neq 0, 1728$  and $\jmath(f) \in J_i$ for $i=1,2,4$ then 
$G \cdot f \in \mF^+$  if and only if $\jmath(f) \in J_i^+$ where 
 \[ J_i^+ = \{r \in J_i : r/(r-1728) \text{ is a square in $\bF_q$}\}. \]
 The sets $J_i^+$ have sizes 
 \begin{enumerate}
\item $|J_1^+|=|J_1|/2 =(q-\mu)/6$ where $q \equiv \mu \mod 3$  and $\mu \in \{ \pm 1\}$.
\item $|J_4^+|= (q-\ell)/12$ where  $q \equiv \ell \mod 12$ and $\ell \in \{5,7,11,13\}$.
\item $|J_2^+|=\begin{cases} (q-1)/4 &\text{ if $q \equiv 1 \mod 12$,}\\
   (q-5)/4 &\text{ if $q \equiv 5 \mod 12$,}\\
 (q-3)/4  &\text{ if $q \equiv -1, -5 \mod 12$.}\end{cases}$
\end{enumerate}
\end{enumerate}
\end{prop}
\bep
Clearly $\jmath(f)=0$ if and only if $I(f)=0$ and hence all orbits with $\jmath(\mO)=0$ are in $\mF^{+}$. Now, suppose $\jmath(f)=1728$. We write  $f = f_1 f_2$ where $f_1 = \prod_{i=1}^2(Ys_i-Xt_i)$ and    $f_2 = \prod_{i=3}^4(Ys_i-Xt_i)$ for a restricted ordering $((s_1,t_1), \dots, (s_4,t_4)) \in \text{ord}(f)$. 
We note from \eqref{eq:resultant} and \eqref{eq:crossratiodef} that 
\begin{multline*} \text{Res}(f_1,f_2) = \alpha \beta, \quad \text{ and } \lambda = \tfrac{\alpha}{\beta}, \quad \text{where }\\  \alpha = (t_4 s_2-t_2s_4) (t_3 s_1-t_1s_3), \quad \beta = (t_4 s_1-t_1s_4) (t_3 s_2-t_2s_3). \end{multline*}
Therefore, by Lemma \ref{I/res}, we have  
\[ 36 I(f) = \beta^2 (\lambda^2-\lambda + 1).\]
If $f \in \mF_4$, then clearly $\beta \in \bF_q$. If  $f \in \mF_4'$, we have $(s_2,t_2) = (\phi(s_1), \phi(t_1))$, $(s_4,t_4) = (\phi(s_3), \phi(t_3))$, and hence $\beta = N_{q^2/q}(\phi(t_3)s_1-\phi(s_3)t_1) \in \bF_q$. Thus for $f \in \mF_4 \cup \mF_4'$, the quantity  $I(f)$ is a square in $\bF_q$ if and only if $(\lambda^2-\lambda + 1)$ is a square for $\lambda \in \{-1,1/2,2\}$,  which is equivalent to  the existence of $\sqrt{3}$ in $\bF_q$, i.e. $q \equiv  \pm 1 \mod 12$.  If $f \in \mF_2$, then $(s_4,t_4) = (\phi(s_3), \phi(t_3))$ whereas  $s_1,t_1, s_2,t_2$ are in $\bF_q$, therefore  $\alpha = \phi(\beta)$. Similarly, if $f \in \mF_2'$ then
\[ (s_4,t_4) = (\phi^3(s_1), \phi^3(t_1)), \;  (s_3,t_3) = (\phi(s_1), \phi(t_1)), \; (s_2,t_2) = (\phi^2(s_1), \phi^2(t_1)), \] and therefore 
$\alpha = \phi(\beta)$ again.  Since $\mN_2^{1728}= \{-1\}$, we see that for $f \in \mF_2 \cup \mF_2'$ we have $\phi(\beta)/\beta = \lambda = -1$, or in other words $ \phi(\beta)=-\beta$. Thus, $\beta= y \sqrt\ep$ for some $y \in \bF_q$,  and hence $\beta^2 = \ep y^2$ is a non-square in $\bF_q$.  Thus, for $f \in \mF_2 \cup \mF_2'$, the quantity  $I(f)$ is a square in $\bF_q$ if and only if $3$ is a non-square in $\bF_q$   which is equivalent to $q \equiv  \pm 5 \mod 12$. \\
Now suppose $\mO = G \cdot f$ with $\jmath(f) \neq 0, 1728$. It follows from the identity  \eqref{eq:jfdef}  
   $\jmath(f)/(\jmath(f)-1728)=I(f)^3/J(f)^2$,  that $I(f)$ is a square in $\bF_q$ if and only if $\jmath(f)/(\jmath(f)-1728)$ is a square in $\bF_q$.  
   For $\lambda \in \mN_4$ we recall from \eqref{eq:jdef} that
\[ \tfrac{\jmath(\lambda)}{\jmath(\lambda)-1728} =  \frac{(\lambda^2-\lambda+1)^3}{\left((\lambda+1)(\lambda-2)(\lambda-\tfrac{1}{2}) \right)^2} \] 
is a square in $\bF_q$ if and only if $(\lambda^2-\lambda+1)$ is a square. The number of solutions $(\lambda,s)$ of  the equation $\lambda^2-\lambda+1=s^2$ is $(q-1)$, of which there are $4$ solutions $\{0,1\} \times \{\pm1\}$ with $\lambda \in \{0,1\}$. There are $2$ solutions $\{-\omega^{\pm1}\} \times \{0\}$ if $q \equiv 1,7 \mod 12$ and there are $6$ solutions $\{-1,1/2,2\} \times \{ \pm \sqrt 3\}$ if $q \equiv 1,11 \mod 12$.  Thus, the number of   $\lambda \in \mN_4$ with $ \jmath(\lambda)/(\jmath(\lambda)-1728)$ being a square in $\bF_q$, is  $(q-\ell)/2$ where $q \equiv \ell \mod 12$ and  $\ell \in \{5,7,11,13\}$.  These make up  $(q-\ell)/12$ orbits of $H_4=G_*$ in $\mN_4$. Thus $|J_4^+|=(q-\ell)/12$.\\
For $\lambda \in \mN_2$,  let $\tau = \lambda+\phi(\lambda)$. Since $\lambda \notin \bF_q$ and $\lambda \phi(\lambda)=1$, the quantity $\tau^2-4 = (\phi(\lambda)-\lambda)^2$ is a non-square in $\bF_q$. We can write \eqref{eq:jdef} as
\[ \tfrac{\jmath(\lambda)}{\jmath(\lambda)-1728} =  \frac{4(\tau-1)^3}{(\tau+2)(2\tau-5)^2}. \]
Therefore,  this quantity is a square in $\bF_q$ if and only if $(\tau-1)/(\tau+2)$ and $\ep(\tau^2-4)$  are  squares in $\bF_q^{\times}$.
Setting $(\tau-1)/(\tau+2)=t^2$, i.e. $\tfrac{\tau}{2}=\tfrac{t^2+1/2}{1-t^2}$, we see that $4|J_2^+|$ is 
the number of rational solutions $(t,s)$ with $t\notin\{ \pm1/2, 0\}$  of the equation   $s^2= \tfrac{3\ep (4t^2-1)}{(t^2-1)^2}$. We may replace this equation with  $3 \ep   (4t^2-1) = ((t^2-1)s)^2$ as the latter equation has no rational solutions for which $t^2-1=0$.
 The number of solutions $(t,s)$ of this equation is $(q+1)$ or $(q-1)$ according as  $3$ is a square or non-square in $\bF_q$, i.e. $q \equiv \pm 1 \mod 12$ or $q \equiv \pm 5 \mod 12$. There are $2$ solutions $(t,s) = (\pm 1/2, 0)$ with $t = \pm 1/2$,   and if $q \equiv 5, 11 \mod 12$ there are $2$ solutions $(t,s) = (0, \pm \sqrt{-3 \ep})$.  Thus  $4|J_2^+|$ equals $(q-1), (q-5), (q-3), (q-3)$ according as $q \equiv 1,5,7,11 \mod 12$.\\
 Finally, the number of $r \in \bF_q \setminus \{0,1728\} = J_4 \cup J_2 \cup J_1$ such that  
 $r/(r-1728)$ is a square in $\bF_q$ is $(q-3)/2$. Therefore, $|J_1^+| = (q-3)/2 - |J_2^+| - |J_4^+|$
equals  $(q-\mu)/6$ where $q \equiv \mu \mod 3$  and $\mu \in \{ \pm 1\}$.
\eep

\begin{lem} \label{lem_sgn}
Let $\{L , L^{\perp}\}$ in $\mQ \setminus \mQ^0$ with $L \neq L^{\perp}$ and let $f = \pi(L)$. Let $(r_1, \dots,r_4) \in \mathrm{ord}(f)$ and  let $F$ be a splitting field of $f$. Let $\sigma: \mathrm{Stab}_F(f) \to S_4$ be the homomorphism $g \mapsto \sigma_g$  of Lemma  \ref{stabilizer_lemma}   given by the permutation action of $\mathrm{Stab}_F(f)$ on $(r_1, \dots, r_4)$. Let 
\[ \widetilde{\mathrm{sgn}}: \mathrm{Stab}_F(f) \to \bZ/2\bZ,\]
  be the homomorphism defined by $\widetilde{\mathrm{sgn}}(g)=-1$ if $g$ interchanges $L$ and $L^{\perp}$  and $\widetilde{\mathrm{sgn}}(g)=1$ if $g$ fixes $L$ and $L^{\perp}$. We have 
  \[ \widetilde{\mathrm{sgn}}(g) = \mathrm{sgn}(\sigma_g), \quad \text{where } \mathrm{sgn}:S_4 \to \bZ/2\bZ \text{ is the sign homomorphism}.\]
\end{lem}
\begin{proof}
Let $g_0 \in PGL_2(F)$ with $g_0 \cdot (r_1,r_2, r_3,r_4) = (\infty,0,1, \lambda)$. Let $\psi_\lambda = XY(Y-X)(Y-\lambda X)$ and let $\sigma:\text{Stab}_F(\psi_\lambda) \to S_4$ be the homomorphism $g \mapsto \sigma_g$  of Lemma  \ref{stabilizer_lemma}   given by the permutation action of $\text{Stab}_F(\psi_\lambda)$ on $(\infty,0,1,\lambda)$. If $g \in \text{Stab}_F(f)$ and $h = g_0 g g_0^{-1} \in \text{Stab}_F(\psi_\lambda)$, then writing $(\infty, 0, 1, \lambda)$ as $(w_1, \dots, w_4)$ we have
\[ w_{\sigma_h(i)}=g_0g g_0^{-1} w_i =  g_0g r_i = g_0 r_{\sigma_g(i)} = w_{\sigma_g(i)},\]
which shows that $\sigma_g = \sigma_{h}$.  By  Lemma  \ref{stabilizer_lemma} we know that Stab$_F(\psi_\lambda)$ is isomorphic via the homomorphism $\mathrm{Stab}_F(\psi_\lambda) \to S_4$ to: 
\begin{enumerate}
\item $ \langle \sigma_1, \sigma_2 \rangle \simeq  \bZ/2 \bZ  \times \bZ /2 \bZ$ if $\lambda \notin \{-1,1/2,2\}$,
\item $ \langle \sigma_1, \sigma_2, \tilde\sigma_3  \rangle \simeq D_4$ if $\lambda \in \{-1,1/2,2\}$,
\end{enumerate}
 where  $\tilde\sigma_3$ equals  the transposition  $(12), (13)$ or  $(14)$ according as $\lambda = -1, 2$ or $1/2$, respectively.
  It is easy to check that $\sigma_1=\sigma_{g_1}$ for  $g_1(t)=\tfrac{\lambda}{t}$ and  $\sigma_2=\sigma_{g_2}$ for  $g_2(t)=\tfrac{t-\lambda}{t-1}$. The transposition  $\tilde\sigma_3$ is realized by $g_3(t) = \tfrac{1}{t},\tfrac{t}{t-1}$ or $ \tfrac{t}{2t-1}$ according as $\lambda = -1, 2$ or $1/2$.
  We also note that $\sigma(\text{Stab}(\psi_\lambda))$ is contained in the alternating group $A_4$ if $\lambda \notin \{-1,1/2,2\}$, and in case $\lambda \in \{-1,1/2,2\}$  we have sgn$(\sigma_g)=1$ for the subgroup  $\langle g_1, g_2 \rangle \simeq \bZ/2 \bZ  \times \bZ /2 \bZ$ and sgn$(\sigma_g)=-1$ for the coset $g_3 \langle g_1, g_2 \rangle$.   \\
By Lemma \ref{lem_pi}, the lines $\{g_0\cdot L, g_0 \cdot L^{\perp}\}$ have coordinates
\beq \label{eq:L_lambda} (0,1,\tfrac{2}{3} (1+\lambda), \lambda, 0, \pm \tfrac{2}{3} \sqrt{\lambda^2-\lambda+1}). \eeq  We recall  from \eqref {eq:wedge2g4} that the action of $g \in PGL_2$ on $\mQ$ is given in coordinates by
\[  g   \cdot \bbsm y_0 \\ y_1 \\ y_2 \\ y_3 \\ y_4 \\y_5 \besm  = \det(g) \bbm g^{[4]} \bbsm y_0 \\y_1 \\y_2 \\y_3 \\y_4 \besm \\ \det(g)^2 y_5 \bem. \]
Taking $g= g_1, g_2, g_3$, we see  that $g_1$  and $g_2$ fix $g_0\cdot L$ and $g_0\cdot L^\perp$, whereas $g_3$ interchanges them. Thus, $h \in \text{Stab}_F(\psi_\lambda)$ interchanges the  lines $\{g_0\cdot L, g_0 \cdot L^{\perp}\}$ if and only if $\text{sgn}(\sigma_h)=-1$. 
Returning to $g \in \text{Stab}_F(f)$ and $h =g_0 g g_0^{-1}   \in \text{Stab}_F(\psi_\lambda)$, we see from $g \cdot L =  g_0^{-1} h g_0 \cdot L$ that $h(g_0 \cdot L) = (g_0 \cdot L)^{\perp}$ if and only if 
$g \cdot  L = L^{\perp}$ (we note that $(g_0 \cdot L)^\perp = g_0 \cdot L^\perp$ because $L^{\perp} = \ast L$ and $\ast$ commutes with the action of $PGL_2$). Therefore, $g$ interchanges $L$ and $L^\perp$ if and only if $h$ interchanges $g_0 \cdot L$ and $g_0 \cdot L^\perp$. We have shown that the latter condition is equivalent to $\text{sgn}(\sigma_h)=-1$, and we have already shown that $\sigma_g = \sigma_h$. This completes the proof that $\widetilde{\text{sgn}}(g)=\text{sgn}(\sigma_g)$.

\end{proof}

\begin{prop} \label{selfdualj1728}
For an orbit $\mO \in \mF^+$:
\begin{enumerate}
\item if $\jmath(\mO)=0$, then $\mO$  lifts to a single orbit $\fO$ of the same size as $\mO$ and with  Stab$(\fO)$ isomorphic to Stab$(\mO)$.
\item if  $\jmath(\mO) \notin \{0, 1728\}$ and $\mO \subset \mF^+$, then $\mO$  lifts to a pair of (distinct) orbits  $\fO, \fO^\perp$ each of the same size as $\mO$ and with  Stab$(\fO)$ isomorphic to Stab$(\mO)$.
\item if $\jmath(\mO)=1728$ then of the $5$ orbits $G \cdot f$ in Table \ref{table:1}: 
\begin{enumerate}
\item [(i)] $G \cdot f \in \mF_4$ with $\Lambda_f = \{-1, 1/2, 2\}$.
\item [(ii)]  $G \cdot f \in \mF_4'$ with $\Lambda_f = \{-1\}$.
\item [(iii)] $G \cdot f \in \mF_4'$ with $\Lambda_f = \{1/2, 2\}$.
\item [(iv)] $G \cdot f \in \mF_2$ with $\Lambda_f = \{-1\}$.
\item [(v)] $G \cdot f \in \mF_2'$ with $\Lambda_f = \{-1\}$.
\end{enumerate}
The orbits (i)-(iii) are in $\mF^+$ if $q \equiv \pm 1 \mod 12$,  and the orbits (iv)-(v) are in $\mF^+$ if $q \equiv \pm 5 \mod 12$.  The orbit in $(iii)$ lifts to a pair of distinct orbits $\fO, \fO^\perp$  each of the same size as $\mO$ and with  Stab$(\fO)$ isomorphic to Stab$(\mO)$.
The orbit in each of the cases  (i),(ii),(iv),(v) lifts to a single self-dual orbit $\fO$ of double the size as $\mO$, and Stab$(\fO)$ is isomorphic to the index $2$ subgroup of Stab$(\mO)$ consisting of elements $g$ with $\text{sgn}(\sigma_g)=1$. 
\end{enumerate}

\end{prop}

\begin{proof}
Each orbit  $\mO$ with $\jmath(\mO)=0$, by definition lifts to a single orbit $\fO$ with $\pi:\fO \to \mO$ being bijective. This proves (1) above. 

Let $\fO = G \cdot L$ with $L \neq L^\perp$, and let $\pi(\fO)=G \cdot f$. We have seen in Lemma \ref{lem_sgn}, that there exists $g \in \text{Stab}(f)$  with $g \cdot L = L^{\perp}$ if and only if sgn$(\sigma_g)=-1$. In other words, if $\sigma(\text{Stab}(f)) \subset A_4$, then $\fO$ and $\fO^\perp$ are distinct orbits and $\text{Stab}(\fO),  \text{Stab}(\fO^\perp)$ are isomorphic to $\text{Stab}(f)$.
On the other hand, if $\sigma(\text{Stab}(f)) \subsetneq A_4$, then $G \cdot f$ lifts to  a single orbit $\fO$ with Stab$(\fO)$ is isomorphic to the index $2$ subgroup of Stab$(\mO)$ consisting of elements $g$ with $\text{sgn}(\sigma_g)=1$.

We have also shown that in case $\jmath(f) \neq 0,1728$, we have sgn$(\sigma_g)=1$ for all $g \in \text{Stab}_F(f)$ (where $F$ is a splitting field of $f$). This proves (2) above. \\

Of the five orbits (i)-(v) above with $j(\mO)=1728$,   there exists $g \in \text{Stab}(f)$ with $\text{sgn}(\sigma_g)=-1$  only in  the cases (i), (ii), (iv), (v).  This proves (3) above.
By part (2) of Proposition \ref{J+prop}, the orbits (i)-(iii) are in $\mF^+$ if $q \equiv \pm 1 \mod 12$, and the orbits (iv)-(v) are in $\mF^+$ if $q\equiv \pm 5 \mod 12$.
\end{proof} 
In the Proposition above, we have determined the isomorphism class of the stabilizer subgroup of each line orbit $\fO$ in terms of the stabilizer subgroup of $\mO=\pi(\fO)$. It will be useful to record this:
\begin{cor} \label{cor_stab}
$\text{Stab}(\fO)$  is isomorphic to $\text{Stab}(\mO)$ where $\mO = \pi(\fO)$  for all orbits except the orbits listed as (i), (ii), (iv), (v) in Proposition \ref{selfdualj1728}. For these four orbits we have 
\beq \label{eq:stab_self_dual}   
\text{Stab}(\fO) \cong
\begin{cases} 
\langle \sigma_1, \sigma_2 \rangle \simeq \bZ/2 \bZ \times \bZ/2 \bZ &\text{ for orbits (i)-(ii),  here $q \equiv \pm 1 \mod 12$.}\\
\langle \sigma_1 \rangle \simeq \bZ/2 \bZ &\text{ for orbit (iv),  here $q \equiv \pm 5 \mod 12$.}\\
\langle (\sigma_3  \sigma_2)^2 =\sigma_1 \rangle \simeq \bZ/2 \bZ &\text{ for orbit (v), here $q \equiv \pm 5 \mod 12$.}\end{cases} \eeq
\end{cor}
\bep
We only need to prove \eqref{eq:stab_self_dual}. For these  orbits (i), (ii), (iv), (v), the group $\sigma(\text{Stab}(f))$ is given in cases (i), (ii), (iv), (v),  respectively, of 
Proposition \ref{stabilizer_prop}. Since $A_4=\langle \sigma_1, \sigma_2, \sigma_4 \rangle$, the index $2$ subgroup $\sigma(\text{Stab}(f)) \cap A_4$ of $\sigma(\text{Stab}(f))$ is as listed in \eqref{eq:stab_self_dual}.
\eep

\subsection*{Proof of Theorem  \ref{thm_main} }
We are now ready to prove our main result, Theorem \ref{thm_main}. We recall from Table \ref{table:1} that $G$  orbits of various sizes  in $\mF\setminus\mF^0$ are:\\

\begin{tabular}{c| *{6}{c}}
   & $|G|$ & $|G|/2$ & $|G|/3$ & $|G|/4$ & $|G|/12$   &  $|G|/8$  \\
&&&&&&\\  \hline \\
$\jmath(f)\in J_4$ & $0$ & $0$ & $0$ & $4$ & $0$ &  $0$  \\
$\jmath(f)\in J_2$ & $0$ & $2$ & $0$ & $0$ & $0$ &  $0$  \\
$\jmath(f)\in J_1$ & $1$ & $0$ & $0$ & $0$ & $0$ &  $0$  \\
$\jmath(f)=1728$ & $0$ & $0$ & $0$ & $3$ & $0$ &  $2$  \\
$\jmath(f)=0$  & $0$ & $(1-\mu)$ & $(1+\mu)$ & $\tfrac{1+\mu}{2}$ &  $\tfrac{1+\mu}{2}$  &  $0$
  \\ [1ex]
\hline  \\ [1ex]
\end{tabular}
\[\]

We will show that the number of orbits $\fO$  in $\mQ \setminus \mQ^0$, with $\jmath(\fO) =0, 1728$ and $\jmath(\fO) \in \bF_q \setminus \{0, 1728\}$ and their sizes are as given in the Table \ref{table:2} which we reproduce here:\\
\noindent \begin{tabular}{c| *{5}{c}}
$\jmath(\fO)$   & $|G|$ & $\tfrac{|G|}{2}$ & $\tfrac{|G|}{3}$ & $\tfrac{|G|}{4}$ & $\tfrac{|G|}{12}$     \\
&&&&&\\  \hline \\
$\jmath\in J_4^+$ & $0$ & $0$ & $0$ & $8$ & $0$   \\
$\jmath\in J_2^+$ & $0$ & $4$ & $0$ & $0$ & $0$   \\
$\jmath\in J_1^+$ & $2$ & $0$ & $0$ & $0$ & $0$   \\
$\jmath=1728$ & $0$ &  \small{$\begin{cases}  0 &\text{ if $q \equiv \pm 1 \mod 12$} \\
2 &\text{ if $q \equiv \pm 5 \mod 12$} \end{cases}$}  & $0$ & 
 \small{$\begin{cases} 4 &\text{ if $q \equiv \pm 1 \mod 12$} \\
0 &\text{ if $q \equiv \pm 5 \mod 12$} \end{cases}$} & $0$   \\
$\jmath=0$  & $0$ &  \small{$(1-\mu)$} &  \small{$(1+\mu)$} & $\tfrac{1+\mu}{2}$ &  $\tfrac{1+\mu}{2}$  
  \\ [1ex]
\hline
\\
total & $\tfrac{q-\mu}{3}$ &  \small{$q-1$} &   \small{$(1+\mu)$} & $\tfrac{2q-10 -(1+\mu)/2}{3}$ &  $\tfrac{1+\mu}{2}$. \\ [1ex]
\end{tabular}
\[\]
First, we determine the size of each orbit $\fO$. We have shown in Proposition \ref{selfdualj1728}, that 
Stab$(\fO)$ is isomorphic to Stab$(\mO)$ where $\mO=\pi(\fO)$, in all cases except the orbits listed as (i),(ii),(iv),(v) therein.
 In these four cases, Stab$(\fO)$ is isomorphic to the index $2$ subgroup of Stab$(\mO)$ given in \eqref{eq:stab_self_dual}.  In particular, the size of each orbit $\fO$ is the same as $\mO=\pi(\fO)$ except in these four cases, where $|\fO|=2 |\pi(\fO)|$.

 \emph{Orbits $\fO$ with $\jmath(\fO)=0$}: By part 1) of Proposition \ref{selfdualj1728},  there is one orbit $\fO=\pi^{-1}(\mO)$ for each orbit $\mO$ with $\jmath(\mO)=0$. Therefore,  the last  row of Table \ref{table:2} is identical with the  third row of Table \ref{table:1}.

\emph{Orbits $\fO$ with $\jmath(\fO) \notin \{0, 1728\}$}:   By part 2) of Proposition \ref{selfdualj1728}, the orbits $\mO$ in each of the first three rows of Table \ref{table:1} lift to  distinct orbits $\fO, \fO^\perp$  provided $\mO \subset \mF^+$.  Thus, the first three rows of Table \ref{table:2} are obtained from  the corresponding  rows of Table \ref{table:1}, by multiplying the entries by $2$, and taking only rows representing  orbits for which $\jmath(\mO) \in J_4^+ \cup J_2^+\cup J_1^+$.

\emph{Orbits $\fO$ with $\jmath(\fO) =1728$}: 
By part 3) of Proposition \ref{selfdualj1728}, in case  $q \equiv \pm 1 \mod 12$ only the $3$ orbits (i)- (iii) are in $\mF^+$, whereas if $q \equiv \pm 5 \mod 12$, then only the orbits (iv)-(v) are in $\mF^+$. 
\qed
\[\]

We have determined the $(2q-3+\mu)$ orbits of generic lines in the Klein representation in $PG(5,q)$. The usual way to  representation  lines of $PG(3,q)$  is to give a pair of generators in $PG(3,q)$ for the line (as has been done in the works \cite{DMP1}, \cite{DMP2} and \cite{GL} for some line orbits).  We now tabulate the  generators of each of the $(2q-3+\mu)$  orbits of generic lines.  Given a representative $(z_0, \dots, z_5) \in \mQ$ of a line $L$, the  Pl\"{u}cker coordinates $p_{ij}$ of $L$ can be found using \eqref{eq:p_z}. If some Pl\"{u}cker $p_{ij}(L)$ is non-zero, say if $p_{01}(L) \neq 0$, then there is a unique pair of  generators of $L$ of the form $(1,0,x_2,x_3), (0,1,y_2,y_3)$. We can represent these generators by the rows of a $2 \times 4$ matrix having $\bbsm 1 & 0\\ 0& 1 \besm$ as a submatrix. This is the method we use to obtain the generator matrices below for all the line orbits.
As an example, we consider the orbit $G \cdot \psi_\lambda$ from \eqref{eq:psilambda}:
\[ G \cdot \psi_\lambda, \,  \psi_\lambda= XY(Y-X)(Y-\lambda X), \, \lambda \in \{x\in\bF_q:x^2-x+1 \text{ is a square in $\bF_q\setminus\{1\}$}\}.\]
 We have shown in \eqref{eq:L_lambda} that the coordinates $(z_0, \dots, z_5)$ of the lines $L, L^{\perp}$ with $\pi(L)=\psi_\lambda$ are  $(0,1,2(1+\lambda)/3, \lambda,0, \pm  \tfrac{2}{3} \sqrt{\lambda^2-\lambda+1})$.   Using \eqref{eq:p_z}, the  Pl\"{u}cker coordinates of these lines are:
\[p_{01}=p_{23}=0, p_{02}=1, p_{03}=1+\lambda + s, p_{12}=\tfrac{1+\lambda -s}{3}, p_{13}=\lambda \text{ where }  s=\pm \sqrt{\lambda^2-\lambda+1}. \]
Since $p_{02}=1$ there exists a unique pair of generators of the form $(1,x_1,0,x_3), (0,y_1,1,y_3)$. 
Since $p_{01}=y_1, p_{23}=-x_3$ we see that $y_1=x_3=0$. Since $p_{03}=y_3$ and $p_{12} = x_1$, we get $y_3=1+\lambda + s$ and $x_1=(1+\lambda -s)/3$:
\[ M^{\pm}(\psi_{\lambda}) = \bbsm 1 & (1+\lambda -s)/3& 0 & 0\\0&0&1& 1+\lambda+s \besm, \quad s = \pm \sqrt{\lambda^2-\lambda+1}. \]
If $\lambda \in \{-\omega, -\omega^2\}$ (when $q \equiv 1 \mod 3$) then $L = L^{\perp}$ is represented by 
\beq \label{eq:psi_omega} M(\psi_{-\omega}) = \bbsm 1 & (1-\omega)/3& 0 & 0\\0&0&1& 1-\omega \besm.\eeq
If the superscript $\pm$ appears on a generator matrix, it is understood that it stands for two  generator matrices representing a pair of lines $L, L^\perp$.  In Table \ref{table:4} below, we use the notation of Table \ref{table:3}.  The last column of the Table \ref{table:4} gives the isomorphism class of Stab$(\fO)$ for each orbit $\fO$, which we have already determined in Corollary \ref{cor_stab}. We begin with orbits $\fO$ with $\pi(\fO) \in \mF_4 \cup \mF_2 \cup \mF_1$. The orbits $\fO, \fO^{\perp}$ with $\jmath(\fO) \in J_4^+ \cup J_2^+ \cup J_1^+$ are represented by:
\beq M_r^\pm(\mathcal E) = \bbm \frac{16s}{3} & 2 & 1 &0\\
\frac{128s^2}{27}&\frac{16s}{9}&0&1 \bem, \quad s=\pm \sqrt{1 - \tfrac{1728}{r}}. \eeq

The orbit $\mO=G \cdot XY(Y-X)(Y+X)$ in $\mF_4$ with $\jmath(\mO)=1728$ equals $G \cdot \psi_{-1}$. It lifts to a single self-dual  orbit $\fO$ represented by 
\beq M(\psi_{-1})= \bbsm 1 & -1/s & 0 & 0\\0&0&1& s \besm, \quad s =  \sqrt 3, \, q \equiv \pm 1 \mod 12. \eeq

The orbit $\mO=G \cdot  X( Y^3  -X^3)$  in $\mF_4$ with $\jmath(\mO)=0$ can also be represented as $G \cdot \psi_{-\omega}$. This orbit   lifts to a single orbit $\fO$ represented by the matrix $M(\psi_{-\omega})$ in \eqref{eq:psi_omega} above.

The orbit $\mO=G \cdot XY(Y^2-\ep X^2)$ in $\mF_2$ with $\jmath(\mO)=1728$ 
lifts to a single self-dual orbit $\fO$ represented by 
\beq M(\upsilon_{0})= \bbsm 1 & -s/3 & 0 & 0\\0&0&1& s \besm, \quad s =  \sqrt{ 3 \ep}, \, q \equiv \pm 5 \mod 12. \eeq

The orbit $\mO=G \cdot  X( Y^3  -r X^3)$  with $\jmath(\mO)=0$  lifts to a single orbit $\fO$
represented by
\beq M_{X(Y^3-rX^3)}= \bbsm 1 & 0 & 0 & -2r\\0&0&1& 0 \besm. \eeq
Thus, the orbits (a)  $G \cdot X(Y^3  - X^3)$ for $q \equiv 2 \mod 3$, (b) $G \cdot X( Y^3  -\gamma X^3)$ for $q \equiv 1 \mod 3$, and (c) $G \cdot X(Y^3  -\gamma^2 X^3)$ for $q \equiv 1 \mod 3$, each lift to a single orbit generated by (a) $M_{X(Y^3-X^3)}$, (b) $M_{X(Y^3-\gamma X^3)}$, (c) $M_{X(Y^3-\gamma^2 X^3)}$, respectively.\\

We recall from Table \ref{table:3} the quantities  $\alpha_{-1}$, $\alpha_{\{-\omega,-\omega^2\}}$ and $\alpha_{\{\lambda, \lambda^{-1}\}}$ which are mapped to $-1, -\omega$ and $\lambda \in   \bF_q\setminus \{0, \pm 1, -\omega,-\omega^2\}$, respectively, under the map $\chi_1(\alpha) = N_{q^2/q}(\tfrac{\alpha - \sqrt\ep}{\alpha+\sqrt\ep})$. The generator matrices for the orbits $\fO$ for which $\pi(\fO)$ is  in $\mF_4'$ has these representatives are:
\beq
M_{\alpha_{-1}}(\psi') = \bbsm  1& 0& \ep y/\sqrt3& \ep x\\
0&1& x&  \ep y \sqrt3 \besm, \quad \alpha_{-1} = x + \sqrt\ep y. \eeq
\beq   M_{\alpha_{\{-\omega,-\omega^2\}}}(\psi') = \bbsm 
1& 0& \ep y/\sqrt{-3}&  \ep x\\
0&1& x&  \ep y \sqrt{-3}\besm, \quad \alpha_{\{-\omega,-\omega^2\}} = x+\sqrt\ep y. \eeq
\begin{table}[b]
\begin{tabular}{||c|c|c|c||}
\hline
Generators of $\fO$ & $\jmath(\fO)$ & Type of $\pi(\fO)$ & Stabilizer    \\
&&&\\  
\hline  
&&&\\  
$M^\pm_r, \;  r \in J_4^+$ \; ($2|J_4^+|$ orbits)& $r \in J_4^+$ & $\mF_4$&  $\bZ/2\bZ \times \bZ/2\bZ$ \\
&&&\\  
$M^\pm_r, \;  r \in J_2^+$  \; ($2|J_2^+|$ orbits)& $r\in J_2^+$ & $\mF_2$&  $\bZ/2\bZ$ \\
&&&\\  
$M^\pm_r, \;  r \in J_1^+$  \; ($2|J_1^+|$ orbits)& $r\in J_1^+$ &$ \mF_1$&  trivial \\
&&&\\  
\hline 
&&&\\  
$M(\psi_{-1}), \;q \equiv \pm 1 \mod 12$ & $1728$ & $\mF_4$ & $\bZ/2\bZ \times \bZ/2\bZ$ \\
$ M(\psi_{-\omega}),\; q \equiv 1 \mod 3$ & $0$ & $\mF_4$ & $A_4$\\
$M(\upsilon_{0}),  \;q \equiv \pm 5 \mod 12$ & $1728$ & $\mF_2$ & $\bZ/2\bZ$ \\
$M_{X(Y^3- X^3)} \; q \equiv -1 \mod 3 $ & $0$ & $\mF_2$ & $\bZ/2\bZ$\\
$M_{X(Y^3-\gamma X^3)}  \; q \equiv 1 \mod 3 $ & $0$ & $\mF_1$ & $\bZ/3\bZ$\\
$M_{X(Y^3-\gamma^2 X^3)} \; q \equiv 1 \mod 3$ & $0$ & $\mF_1$ & $\bZ/3\bZ$\\
&&&\\  
\hline 
&&&\\  
$M_{\alpha_{-1}}(\psi'),  \;q \equiv \pm 1 \mod 12$  &  $1728$ & $\mF_4'$ & $\bZ/2\bZ \times \bZ/2\bZ$  \\

$M_{\alpha_{\{2, 1/2\}}}^\pm(\psi'), \; q \equiv \pm 1 \mod 12 $  &  $1728$ & $\mF_4'$ & $\bZ/2\bZ \times \bZ/2\bZ$  \\
$ M_{\alpha_{\{-\omega, -\omega^2 \}}}(\psi'),\;  q \equiv 1 \mod 3$  &  $0$ & $\mF_4'$ & $\bZ/2\bZ \times \bZ/2\bZ$  \\
&&&\\  
$ M_{\alpha_{\{\lambda, \lambda^{-1}\}}}^\pm(\psi'), \;  \lambda \in \mN_4\,$  (total $6|J_4^+|$ orbits)    &  $\jmath(\lambda)\in J_4^+$ & $\mF_4'$ & $\bZ/2\bZ \times \bZ/2\bZ$ \\  
&&&\\  
\hline
&&&\\ 
$M_{t_{-1}}(\upsilon', 1),  \quad q \equiv  5 \mod 12$ & $1728$ &  $\mF_2'$ &  $\bZ/2\bZ$\\
&&&\\  
$M_{t_{\{-\omega,-\omega^2\}}}(\upsilon', 1), \quad q \equiv 5 \mod 12$  & $0$& $\mF_2'$ &$\bZ/2\bZ$\\
&&&\\  
$M_{t_{\{\lambda, \lambda^{-1}\}}}^{\pm}(\upsilon', 1), \,\lambda \in \mN_2\,$  (total $2|J_2^+|$ orbits)& $ \jmath(\lambda)\in J_2^+$& $\mF_2'$ &$\bZ/2\bZ$\\
here $q \equiv 1 \mod 4$. &&&\\
\hline  
&&&\\
$M_{t_{-1}}(\upsilon', 3), \quad  q \equiv -5 \mod 12$ & $1728$ &  $\mF_2'$ &  $\bZ/2\bZ$\\
&&&\\  
$M_{t_{\{-\omega,-\omega^2\}}}(\upsilon', 3), \quad q \equiv -1 \mod 12$  & $0$& $\mF_2'$ &$\bZ/2\bZ$\\
&&&\\  
$M_{t_{\{\lambda, \lambda^{-1}\}}}^{\pm}(\upsilon', 3),\,  \lambda \in \mN_2\,$, 
(total $2|J_2^+|$ orbits)& $ \jmath(\lambda)\in J_2^+$& $\mF_2'$ &$\bZ/2\bZ$\\
here $q \equiv 3 \mod 4$.
&&&\\
\hline  
\end{tabular}
\[\]
\caption{Generators  and stabilizers of $G$-orbits of generic lines of $PG(3,q)$}
 \label{table:4}

\end{table}
\begin{multline} M_{\alpha_{\{\lambda, \lambda^{-1}\}}}^\pm(\psi') = \bbm 
1& 0& \frac{- \ep y (1+\lambda)}{3(1 - \lambda)}  +s&  \ep x\\
0&1& x&   \tfrac{\ep y (1+\lambda)}{1 - \lambda} + 3 s\bem, \\ \alpha_{\{\lambda, \lambda^{-1}\}} = x + \sqrt\ep y, \; 
 s  = \pm\frac{2\ep y \sqrt{ \lambda^2 - \lambda +1}}{3(1-\lambda)}. \end{multline}

We recall from Table \ref{table:3}, the representatives $\upsilon'_{-1}$, $\upsilon'_{\{-\omega,-\omega^2\}}$ and $\upsilon'_{\{\lambda, \lambda^{-1}\}}$ for  $\lambda \in \mN_2$, for the orbits in $\mF_2'$.  The generator matrices for the orbits $\fO$ for which $\pi(\fO)$ has these representatives are as below.  The cases $q \equiv 1 \mod 4$ and $q \equiv 3 \mod 4$ are treated separately.
The notation $M_t(\upsilon',1)$ and  $M_t(\upsilon',3)$ refer to these two cases.

\beq  M_{t_{-1}}(\upsilon', 1) = \bbsm 1& 0& \tfrac{s}{3}& 0\\
0&1& 0&  s\besm, \quad s =  \sqrt{-3\gamma}.  \eeq

\beq  M_{t_{\{-\omega,-\omega^2\}}}(\upsilon', 1) = \begin{cases} 
 \bbsm 1& 0& t/3& -2 \sqrt{t\gamma}\\ 
0&1& 0&  -t\besm  &\text{if $\frac{-\sqrt{3 \gamma}}{2}  \notin (\bF_q^\times)^2$} \\
 \bbsm 1& 0& \gamma t/3& -2 \gamma^2 \sqrt{t}\\
0&1& 0&  -\gamma t\besm  &\text{if $\frac{-\sqrt{3 \gamma}}{2}  \in (\bF_q^\times)^2$} \end{cases},\quad t = \frac{-\sqrt{3 \gamma}}{2}.  \eeq
\beq  M_{t_{\{\lambda, \lambda^{-1}\}}}^{\pm}(\upsilon', 1) = \begin{cases} 
 \bbsm 1& 0& \tfrac{t+s}{3}& -2 \sqrt{t\gamma}\\
0&1& 0&  s-t\besm  &\text{if $\ep t \in (\bF_q^\times)^2$} \\
 \bbsm 1& 0& \tfrac{\gamma(t+s)}{3}& -2 \gamma^2 \sqrt{t}\\
0&1& 0&  \gamma(s-t)\besm &\text{if $t \in (\bF_q^\times)^2$} \end{cases},  \,  s = \pm \tfrac{2 \imath \theta^2  \sqrt{\lambda^2-\lambda+1}}{1-\lambda},\, t=\tfrac{\imath \theta^2 (1+\lambda)}{2(1-\lambda)}. \eeq

\beq  M_{t_{-1}}(\upsilon', 3)  = \begin{cases} 
 \bbsm 1& 0& \tfrac{3\theta_0+2s}{6}&  \theta_1 \sqrt{2\theta_0}\\
0&1& 0&  s+\tfrac{3\theta_0}{2} \besm  &\text{if $q \equiv 11 \mod 24$} \\
 \bbsm 1& 0& \tfrac{3\theta_0+2s}{6} &  \theta_1 \sqrt{-2\theta_0}\\
0&1& 0&   s-\tfrac{3\theta_0}{2} \besm &\text{if $q \equiv -1 \mod 24$} \end{cases},  
\quad  s =  \sqrt{-3}, \, t=-\tfrac{\theta_0}{2}. \eeq

\beq  M_{t_{\{-\omega,-\omega^2\}}}(\upsilon', 3) =\begin{cases} 
 \bbsm 1& 0& \tfrac{t-\theta_0}{3}& 2 \theta_1 \sqrt{-t}\\
0&1& 0&  \theta_0-t\besm  &\text{if $-t \in (\bF_q^\times)^2$} \\
 \bbsm 1& 0& \tfrac{-t+\theta_0}{3} & 2 \theta_1 \sqrt{t}\\
0&1& 0&   -\theta_0+t \besm &\text{if $t \in (\bF_q^\times)^2$} \end{cases},  
\quad  \, t=\tfrac{\sqrt 3 - \theta_0}{2}. \eeq

\beq  M_{t_{\{\lambda, \lambda^{-1}\}}}^{\pm}(\upsilon', 3) = \begin{cases} 
 \bbsm 1& 0& \tfrac{t-\theta_0+s}{3}& 2 \theta_1 \sqrt{-t}\\
0&1& 0&  s+\theta_0-t\besm  &\text{if $-t \in (\bF_q^\times)^2$} \\
 \bbsm 1& 0& \tfrac{-t+\theta_0+s}{3} & 2 \theta_1 \sqrt{t}\\
0&1& 0&   s-\theta_0+t \besm &\text{if $t \in (\bF_q^\times)^2$} \end{cases},  
\, s = \pm\tfrac{2  \sqrt{-\lambda^2+\lambda-1}}{1-\lambda}, \, t=
-\tfrac{\theta_0}{2}+\tfrac{\imath  (\lambda+1)}{2(\lambda-1)}. \eeq

\appendix
\section{} \label{appendix}
We continue to use the notation of Section \S \ref{setup}:
\begin{description}
    \item[$V=F^2$] is a two-dimensional vector space over a field $F$.
    \item [$D_mV$, $\mathrm{Sym}^m V$  $\mathrm{and}$  $\mathrm{Sym}^m(V^*)$] are associated tensor spaces.
    \item [$C_m$] is the degree $m$ rational normal curve in $\bP(D_m V)$.
    \item [$\mO_{m-1}(P)$] is the osculating hyperplane to $C_m$ at a point $P$.
    \item[$\Omega_m$] is the bilinear form on $D_mV$ induced by the map $P \mapsto \mO_{m-1}(P)$.
\end{description}
We refer the reader to \S \ref{setup} for the detailed definitions. In this section, we address the questions of  when  the  $GL_2(F)$-modules  $D_m V, \text{Sym}^m V$ and $\text{Sym}^m(V^*)$ are i) irreducible, and ii) isomorphic. 
 We work over an arbitrary field $F$ but for one assumption: 
\begin{assumption} \label{assumption} $m + 2 \leq |F|$ in case $F$ is finite. 
\end{assumption}
\begin{thm} \label{lcm_cond} The following conditions are equivalent:
\begin{enumerate}
\item The binomial coefficients $\binom{m}{0}, \dots, \binom{m}{m}$ are all nonzero in $F$. 
\item Either char$(F)=0$ or $\text{char}(F)=p$ and $b < p$ where $m+1=p^a b$ is the unique factorization of $m+1$ with $b$ relatively prime to $p$.
\item The intersection of all osculating hyperplanes of $C_m(F)$ is empty.
\item  $D_m V$ is an irreducible representation of $GL_2(F)$. 
\item  The spaces  $D_m V, \text{Sym}^m V$ and $\text{Sym}^m(V^*) \otimes  (\otimes^m \wedge^2 V )$ are isomorphic as $GL_2(F)$-modules.
\end{enumerate}
\end{thm}
We first show the equivalence of (1), (2) and (3). Let $A_m$ be the matrix in \eqref{eq:Am} and let $T_m$ be the matrix 
\beq T_m = \bbsm \binom{m}{0} & & &\\ &\binom{m}{1} && \\ && \ddots& \\ &&&\binom{m}{m} \besm, \qquad A_m =  \bbsm & & && \binom{m}{0} \\
&& &-\binom{m}{1}&  \\
&    &\iddots& &  \\
&(-1)^{m-1} \binom{m}{m-1}&&  & \\ 
(-1)^m \binom{m}{m}&& &  &  
 \besm. \eeq
We recall from Definition \ref{polardual} the bilinear form $\Omega_m$ on $D_m V$ whose matrix with respect to the basis $e_1^{[m+1-i]}e_2^{[i-1]}, i=1\dots m+1$ of $D_mV$  is $A_m$.
The equivalence of the first  two conditions is obvious if char$(F)=0$. In case char$(F)=p$, we use the identity (see \cite{Farhi}) 
\[ \text{lcm}\{ \tbinom{m}{0}, \dots, \tbinom{m}{m}\} = \frac{\text{lcm}\{1, \dots, m+1\}}{m+1}.\]
Let $m+1=p^a b$ be the unique factorization of $m+1$ with $b$ coprime to $p$. The largest power of $p$ that divides $\text{lcm}\{1, \dots, m+1\}$ is $a+ \lfloor \log_p b \rfloor$, and the largest power of $p$ that divides $m+1$ is just $a$. Therefore, $\tbinom{m}{0}, \dots, \tbinom{m}{m}$ are all relatively prime to $p$ if and only if $\lfloor \log_p b \rfloor=0$, i.e. $b<p$.\\

  The osculating hyperplane $\mO_{m-1}(P)$ to $C_m$ at $P=\nu_m(s,t)$ consists of all points $Q$ such that $\Omega_m(\nu_m(s,t), Q)=0$. Since any $(m+1)$ points of $C_m$ are linearly independent and there are at least $(m+1)$-points on $C_m(F)$ (by Assumption  \ref{assumption}) we see that the intersection of all osculating hyperplanes of $C_m(F)$ is the radical of the bilinear form $\Omega_m$, which is also the kernel of the matrix $A_m$.
 Clearly, $A_m$ is nonsingular if and only condition (1) holds. \\

We now show  the equivalence of  (4) and (1):  If $(1)$ does not hold, then $W_0=\text{ker}(A_m)$  is a proper $GL_2(F)$-submodule of $D_m V$. Conversely, we will show that any  proper $GL_2(F)$-submodule  $W \subset D_m V$ is contained in  $W_0$ :  let  $r=\text{dim}(W)$ with $1 \leq r \leq m$, and let $M$ be a $(m+1) \times r$ matrix in reduced column echelon form whose columns form a basis for $W$. Let $1  \leq \alpha_1 < \dots < \alpha_r  \leq m+1$ denote the pivot rows of $M$. Thus, the submatrix of $M$ on rows $\alpha_1, \dots, \alpha_r$ is the $r \times r$ identity matrix. For any $g \in GL_2(F)$, let $\tilde g$ be the $r \times r$ submatrix of $g^{[m]}$ on the rows and columns indexed by $\alpha_1, \dots, \alpha_r$. The $GL_2(F)$-invariance of $W$ implies $g^{[m]} M = M \tilde g$. Taking $g$ to be represented by the matrix  $\bbsm 1 & \\ & \lambda \besm$, we see that $M_{ij} (\lambda^{i-\alpha_j}-1)=0$ for all  $\lambda \in F^{\times}$.   If $M_{ij} \neq 0$ and $i \neq \alpha_j$  then we must have  $\lambda^{i-\alpha_j}=1$ for all  $\lambda \in F^{\times}$.   If $F$ is infinite, this is impossible. If $|F|=q$ is finite, then $(q-1)$ must divide $i-\alpha_j$.
However,  by Assumption  \ref{assumption}, we have 
$0 \leq |i-\alpha_j| \leq m  \leq q-2$ and hence $(q-1)$ does not divide $|i-\alpha_j|$. 
  Therefore, we conclude that all entries except the $r$ pivot entries (having value $1$) are zero, in other words $M_{ij} = \delta_{i, \alpha_j}$. 
Either $\alpha_1=1$ or $\alpha_1 >1$.  If $\alpha_1=1$, then the first column of $M$ is $(1,0, \dots,0)$ and hence $\nu_m(1,0) \in W$. Since $W$ is $GL_2(F)$-invariant, we see that $C_m(F)$ is contained in $W$. However, the linear  span of $C_m(F)$ is all of $D_m V$ which contradicts the fact that $W$ is a proper subspace of $D_m V$. If $\alpha_1 >1$, the first row of $M$ consists of zeros, and hence $W$ is contained in the osculating hyperplane $\mO_{m-1}(\nu_m(0,1))$. Since $W$ is $GL_2(F)$-invariant, we see that $W$ is contained in $\mO_{m-1}(\nu_m(s,t))$ for all $(s,t) \in \bP^1(F)$, and hence in $W_0$. \\

Finally, we  show the equivalence of  (5) and (1):
It follows from Proposition \ref{equivs} below (parts (1) and (3))  that any $GL_2(F)$ equivariant linear map between  $D_m V$ and $\text{Sym}^m(V^*) \otimes  (\otimes^m \wedge^2 V)$  is a scalar multiple of $A_m$, and 
any $GL_2(F)$ equivariant linear map between  $D_m V$ and $\mathrm{Sym}^m V$ is a scalar multiple of $T_m$. Thus, the condition (1) that the binomial coefficients $\binom{m}{0}, \dots, \binom{m}{m}$ are all nonzero in $F$, is equivalent to  
the condition that the spaces  $D_m V, \text{Sym}^m V$ and $\text{Sym}^m(V^*) \otimes  (\otimes^m \wedge^2 V )$ are mutually isomorphic as $GL_2(F)$-modules.

\begin{prop} \label{equivs} Under Assumption \ref{assumption} we have:
\begin{enumerate}
\item \hfil$\displaystyle \mathrm{Hom}_{GL_2(F)}\left( D_m V,  \mathrm{Sym}^m(V^*) \otimes  (\otimes^m \wedge^2 V)\right)$  is the one-dimensional space generated by $A_m $. 
\item There is a $PGL_2(F)$-equivariant linear isomorphism $\bP( D_m V)  \to \bP(\text{Sym}^m(V^*))$  if and only if $\binom{m}{0}, \dots, \binom{m}{m}$ are all nonzero in $F$. In this case $A_m$ is the unique such isomorphism.

\item  \hfil$\displaystyle  \mathrm{Hom}_{GL_2(F)}\left(D_m V, \mathrm{Sym}^m V\right)$   is the one-dimensional space generated by $T_m$.
\item There is a $PGL_2(F)$-equivariant linear isomorphism $\bP(D_m V)  \to \bP( \text{Sym}^m V) $ if and only if $\binom{m}{0}, \dots, \binom{m}{m}$ are all nonzero in $F$. In this case, $T_m$ is the unique such isomorphism.
\end{enumerate}
\end{prop}

\begin{proof} Proof of (1): Let $A \in \text{Hom}_{GL_2(F)}\left( D_m V, \text{Sym}^m(V^*) \otimes  (\otimes^m \wedge^2 V)\right)$, i.e.  
\[ (g^{[m]})^t A   g^{[m]}  =  \text{det}(g)^m   A, \quad \text{  for all $g \in GL(V)$}.\]
 Taking $g$ to be represented by the matrix  $\bbsm 1 & \\ & \lambda \besm$, we note that  $g^{[m]}$ is represented by the matrix 
\[ \bbsm 1 & & &\\ & \lambda & &\\ &&\ddots &\\ &&&\lambda^m \besm.\]
Therefore, $ (\lambda^{i+j-2}-\lambda^m) A_{ij}=0$.  Suppose $i+j \neq m+2$ and $A_{ij} \neq 0$, then we must have $\lambda^{m+2-i-j}=1$ for all $\lambda \in F^{\times}$.
If $F$ is infinite, this is impossible. If $|F|=q$ is finite, then by  Assumption \ref{assumption}  we have $|m+2-i-j| \leq m \leq q-2$ which implies that $(q-1)$ does not divide $(m+2-i-j)$ and hence the condition that $\lambda^{m+2-i-j}=1$ for all $\lambda \in F^{\times}$ is impossible. Therefore, we conclude that all non-zero entries of $A$ are on the anti-diagonal. Let  \[ A = \bbsm &&&a_1\\&&a_2&\\&\iddots&&\\a_{m+1}&&&\besm.\] Taking $g = \bbsm &1\\1&\besm$ we get $a_{i} = a_{m+2-i} (-1)^m$. 
Next, taking $g=\bbsm 1 & 0\\ 1& 1 \besm$, we note that  $g^{[m]}$ is represented by the lower triangular matrix whose $ij$-th entries are $\binom{i-1}{j-1}$ and hence:
\[ a_i    \tbinom{m+1-i}{j-1} =         \tbinom{m+1-j}{i-1} (-1)^{m+2-j-i} a_{m+2-j} =  \tbinom{m+1-j}{i-1}         (-1)^{i+j} a_j. \]
Putting $i=1$, we get 
\[ a_j= (-1)^{j-1} \tbinom{m}{j-1}  a_1,\]
 that is,  $A=a_1 A_m$ as required.\\

Proof of (3): Let $T \in \text{Hom}_{GL_2(F)}\left( D_m V,\text{Sym}^m V\right)$, i.e.  
\[ T   g^{[m]}  = g^m   T, \quad \text{  for all $g \in GL(V)$}.\]
 Taking $g$ to be represented by the matrix  $\bbsm 1 & \\ & \lambda \besm$, we note that  both $g^{[m]}$ and $g^m$ are represented by the matrix 
\[ \bbsm 1 & & &\\ & \lambda & &\\ &&\ddots &\\ &&&\lambda^m \besm.\]
Therefore, $ (\lambda^{i-j}-1) T_{ij}=0$.  Suppose $i \neq j$ and $T_{ij} \neq 0$, then we must have $\lambda^{i-j}=1$ for all $\lambda \in F^{\times}$.
If $F$ is infinite, this is impossible. If $|F|=q$ is finite, then by Assumption \ref{assumption},   we have $|i-j| \leq m \leq q-2$ and hence  $(q-1)$ does not divide $|i-j|$. 
Therefore, we conclude that $T$ is a diagonal matrix. Let $t_i$ denote $T_{ii}$. Taking $g = \bbsm &1\\1&\besm$ we get $t_{i} = t_{m+2-i}$. Next, taking $g=\bbsm 1 & 0\\ 1& 1 \besm$, we note that  $(g^{[m]})_{ij}=\binom{i-1}{j-1}$  as above and $(g^m)_{ij} = \binom{m+1-j}{m+1-i}$ we get: 
\[ a_i \tbinom{i-1}{j-1} = \tbinom{m+1-j}{m+1-i} a_j,\]
and hence $a_i = \binom{m}{i-1} a_1$, i.e. $T = a_1 T_m$.\\

Proofs of (2) and (4):  If  $\binom{m}{0}, \dots, \binom{m}{m}$ are all nonzero in $F$, then  the matrices $A_m$ and $T_m$ are invertible and represent $PGL_2(F)$-equivariant maps 
$\bP( D_m V)  \to \bP(\text{Sym}^m(V^*))$ and $\bP( D_m V)  \to \bP(\text{Sym}^m V)$, respectively. 
We now treat the converse of assertions (2) and (4). If the matrix $A \in GL_{m+1}(F)$ represents a  $PGL_2(F)$-equivariant linear isomorphism $\bP( D_m V)  \to \bP(\text{Sym}^m(V^*))$, then there is a function $\chi: GL_2(F) \to F^{\times}$ such that   then we have $(g^{[m]})^t A g^{[m]}  = \chi(g)  A$ for each $g \in GL_2(F)$.
Taking  indices $i,j$ such that $A_{ij} \neq 0$, and comparing the left and right sides of 
 $(g^{[m]})^t A g^{[m]}  = \chi_g  A$, it follows that there is a homogeneous polynomial $\tilde \chi(a,b,c,d)$ of degree $2m$ in the polynomial ring $F[a,b,c,d]$ such that $\chi(g) = \tilde \chi(g)$. It is also clear that $\chi(g)$ is a homomorphism.  We recall that any rational homomorphism $GL_2(F) \to F^{\times}$ is of the form $\chi_r(g)=\det(g)^s$ for $s\in \bZ$. If $|F|=q$ is finite then $s \in \bZ/(q-1)\bZ$ i.e. $0 \leq s \leq q-2$.
 Similarly, assertion (4) is equivalent to the existence of  $T \in GL_{m+1}(F)$ for which $ T g^{[m]}  = \chi'(g)  g^m T$ for some rational homomorphism $\chi':GL_2(F) \to F^\times$. Let $\chi(g)=\chi_r(g)$ and let $\chi'(g) = \chi_{r'}(g)$. By taking the determinant of $(g^{[m]})^t A g^{[m]}  = \det(g)^r  A$
 and $T g^{[m]}  = \det(g)^{r'}  g^m T$ and using the fact  that det$(g^{[m]}) =\det(g^m)  = \det(g)^{m(m+1)/2}$ we get $\det(g)^{(m+1)(m-r)}=1$ and $\det(g)^{r'(m+1)}=1$. If $F$ is infinite, this immediately forces $r=m$ and $r'=0$. At this stage, we have $(g^{[m]})^t A g^{[m]}  = \det(g)^m  A$  and $T g^{[m]}  =  g^m T$, and we appeal to parts 1) and 3) to conclude that $A = A_m$ and $T=T_m$, (up to a scalar multiple), respectively. If $F = \bF_q$ is finite, we will again show that $r=m$ and $r'=0$ (and hence 
 conclude $A = A_m$ and $T=T_m$), but it requires a more careful analysis.
 
We take $0 \leq r, r' \leq q-2$.  We first consider assertion 2). Taking $g = \bbsm a & \\ & d\besm$, we have $\chi(g)=(ad)^r$, and the $(ij)$-th entry of the equation $(g^{[m]})^t A g^{[m]}  = \det(g)^r  A$ reads $(a^{2m-(i+j-2)}d^{(i+j-2)}-(ad)^r)T_{ij}=0$. In other words, if $T_{ij} \neq 0$ we have:
\[a^{2m-(i+j-2)-r}d^{i+j-2-r}=1 \quad \text{ for all $a,d \in F^\times$}.\]
This in turn gives:
   \[ i+j-2 \equiv r \mod (q-1), \qquad   2(m-r) \equiv 0 \mod (q-1).\]
If $q$ is even, the above conditions are equivalent to $m=r$ (and $i+j-2=m$) as desired. If $q$ is odd, the  condition $2(m-r) \equiv 0 \mod (q-1)$ together with the assumption that $m \leq q-2$ gives $r \in \{m-\tfrac{q-1}{2}, m, m + \tfrac{q-1}{2}\}$. We argue that the cases $r=m \pm \tfrac{q-1}{2}$ are not possible: if $r=m - \tfrac{q-1}{2}$ then we have $T_{ij} \neq 0$ only if $i+j-2 \equiv m-\tfrac{q-1}{2}$ which is equivalent to $i+j-2  \in \{m - \tfrac{q-1}{2}, m+ \tfrac{q-1}{2}\}$. Similarly,  if $r=m + \tfrac{q-1}{2}$, we again get $i+j-2  \in \{m - \tfrac{q-1}{2}, m+ \tfrac{q-1}{2}\}$. In this case the $(q-2-m)$ columns $m+2- \tfrac{q-1}{2}, \dots, \tfrac{q-1}{2}$ of $A$ have all entries $0$ which contradicts the fact that $\det(A) \neq 0$. Therefore, $m=q-2$ and  the $(q-1) \times (q-1)$ matrix  $A$ is of the form 
\[ A = \left( \begin{array}{c|c} 
  \begin{smallmatrix} &&\alpha_{\frac{q-1}{2}}\\  & \iddots& \\   \alpha_1& &    \end{smallmatrix}   &  0 \\ 
  \hline 
  0 &   \begin{smallmatrix} &&\beta_{\frac{q-1}{2}}\\ & \iddots& \\  \beta_1& & \end{smallmatrix}
\end{array} \right). \]
Next we take  $g=\bbsm 1 & 0\\ 1& 1 \besm$ for which $\det(g)$ and hence $\chi(g)=1$. We note that  $g^{[m]}$ is represented by the lower triangular matrix whose $ij$-th entries are $\binom{i-1}{j-1}$ and hence, the $(m+1,1)$-th entry of the equation $(g^{[m]})^t A g^{[m]} = \chi(g) A$ reads $ \beta_1  = 0$, which  contradicts $\det(A) \neq 0$.\\

For assertion 4), again taking  $g = \bbsm a & \\ & d\besm$, we have $\chi(g)=(ad)^{r'}$, and the $(ij)$-th entry of the equation $ T g^{[m]}  = \det(g)^{r'}  g^m T$  reads $(a^{m-j+1)}d^{j-1}-a^{m-i+1+r'} d^{i-1+r'})T_{ij}=0$. In other words, if $T_{ij} \neq 0$ we have:
\[a^{i-j-r'}d^{j-i-r'}=1 \quad \text{ for all $a,d \in F^\times$}.\]
If $q$ is even, this is equivalent to  $r'\equiv 0 \mod (q-1)$  (and  $j-i \equiv 0 \mod (q-1)$) as required.  
If $q$ is odd, we get 
\[ j-i \equiv r' \equiv 0 \mod (q-1)/2. \]
We will show that the case $r'  \equiv  (q-1)/2 \mod (q-1)$ is not possible: indeed, if $r' = (q-1)/2$, then the nonzero entries of $T_{ij}$ occur only if $|i-j|=(q-1)/2$. The $(q-2-m)$ columns with indices $m+2- \tfrac{q-1}{2}, \dots, \tfrac{q-1}{2}$ have all entries $0$ which contradicts $\det(T) \neq 0$. Therefore, we must have $m=q-2$ and hence the 
 $(q-1) \times (q-1)$ matrix  $T$ is of the form 
\[ T = \left( \begin{array}{c|c} 
  0 & \begin{smallmatrix} \alpha_1&&\\  &\ddots& \\   &&\alpha_{\frac{q-1}{2}} \end{smallmatrix}    \\   \hline 
   \begin{smallmatrix} \beta_1&&\\ &\ddots& \\  &&\beta_{\frac{q-1}{2}} \end{smallmatrix} & 0
\end{array} \right). \]
Next taking $g=\bbsm 1 & 0\\ 1& 1 \besm$ for which $\det(g)$ and hence $\chi'(g)=1$. We note that  
$(g^{[m]})_{ij}=\tbinom{i-1}{j-1}$  and  $(g^m)_{ij} = \tbinom{q-1-j}{q-1-i}$. The $(1,1)$-entry of the equation $ T g^{[m]}  = \chi'(g)  g^m T$ reads  $\alpha_1=0$ contradicting the fact  that $\det(T) \neq 0$. 
\end{proof}

\begin{rem} \label{exteriordual}
 We recall the standard fact that  for a finite-dimensional vector space $W$ over an arbitrary field $F$, unlike the case with Sym$^m(W)$, there is a $GL(W)$ equivariant isomorphism between alternating power vector spaces $\wedge^m W$ and $(\wedge ^m W^*)^*$: for $w_1, \dots, w_m \in W$ and $\phi_1, \dots, \phi_m \in W^*$ the determinant of the $m \times m$ matrix whose $ij$-th entry is $\phi_i(w_j)$ gives a nondegenerate $GL(W)$-invariant bilinear map $\wedge^m W \times \wedge^m W^* \to F$, or equivalently  a $GL(W)$-equivariant linear isomorphism between  $\wedge^m W$ and $(\wedge ^m W^*)^*$.
\end{rem}

{\bf Conflict of interest statement:} On behalf of all authors, the corresponding author states that there is no conflict of interest. This manuscript has no associated data.

\bibliographystyle{amsplain}

\end{document}